\documentclass[11pt,reqno]{amsart}
\usepackage[T1]{fontenc}
\usepackage{lmodern}
\usepackage{amsmath,amssymb,amsthm,mathtools}
\usepackage{mathrsfs}
\usepackage[margin=1.15in]{geometry}
\usepackage{microtype}
\usepackage{enumitem,needspace}
\usepackage{xcolor}
\usepackage[colorlinks=true,linkcolor=blue!45!black,citecolor=blue!45!black,urlcolor=blue!45!black]{hyperref}
\hypersetup{pdftitle={Universal existence and generic uniqueness for quadratic BSDE systems in a two-player stochastic differential game},pdfauthor={Siyi Wang},pdfsubject={Existence, uniqueness and stability of bounded solutions of quadratic BSDE systems}}
\newcommand{\R}{\mathbb R}
\newcommand{\N}{\mathbb N}
\newcommand{\E}{\mathbb E}
\newcommand{\PP}{\mathbb P}
\newcommand{\FF}{\mathcal F}
\newcommand{\Sp}{\mathcal S}
\newcommand{\Hp}{\mathcal H}
\newcommand{\BMO}{\mathrm{BMO}}
\newcommand{\cE}{\mathcal E}
\newcommand{\ind}{\mathbf 1}
\newcommand{\dd}{\,\mathrm d}
\newcommand{\dt}{\,\mathrm dt}
\newcommand{\ds}{\,\mathrm ds}
\newcommand{\Sol}{\operatorname{Sol}}
\newcommand{\diam}{\operatorname{diam}}
\newcommand{\conv}{\operatorname{conv}}
\newcommand{\dist}{\operatorname{dist}}
\DeclarePairedDelimiter{\norm}{\lVert}{\rVert}
\numberwithin{equation}{section}
\newtheorem{theorem}{Theorem}[section]
\newtheorem{proposition}[theorem]{Proposition}
\newtheorem{lemma}[theorem]{Lemma}
\newtheorem{corollary}[theorem]{Corollary}
\theoremstyle{definition}
\newtheorem{definition}[theorem]{Definition}
\theoremstyle{remark}

\setlist[enumerate,1]{label=\textup{(\roman*)},leftmargin=2.2em}
\allowdisplaybreaks[1]
\title[Universal existence and generic uniqueness for quadratic BSDE systems]{Universal existence and generic uniqueness for quadratic BSDE systems in a two-player stochastic differential game}
\author{Siyi Wang}
\AddToHook{cmd/@setauthors/after}{%
  {\centering\footnotesize
   Financial Technology Thrust\\
   The Hong Kong University of Science and Technology (Guangzhou)\par}}
\thanks{Email: \href{mailto:swangbz@connect.ust.hk}{\texttt{swangbz@connect.ust.hk}}.}
\date{November 23, 2026}
\subjclass[2020]{60H10, 60G44, 54E52}
\keywords{Quadratic BSDE systems, stochastic differential games, Nash equilibria, generic uniqueness, Baire category, BMO martingales}
\begin{document}

\begin{abstract}
We study two-dimensional systems of quadratic backward stochastic differential equations in the augmented natural filtration of a one-dimensional Brownian motion.
The generator has components $q_{\theta,i}(z)=z_i(z_1+z_2)-\tfrac12\alpha_i z_i^2$, $i=1,2$, where $\theta=(\alpha_1,\alpha_2)\in\R^2$.
The system describes a stochastic differential game in which two players choose efforts $u_1,u_2$, the drift of a common diffusion is $u_1+u_2$, and player $i$ pays $u_i^2/2$ per unit time and has risk sensitivity $\alpha_i-1$.
For every bounded terminal payoff vector $\xi=(\xi_1,\xi_2)$, which may depend on the entire state path, the map $(Y,Z)\mapsto Z$ is a bijection from solutions in $\Sp^\infty\times\BMO$ to Nash equilibria in BMO controls.
Whether the risk-neutral system with $\alpha_1=\alpha_2=1$ is well posed for every such $\xi$ was left open in \cite[Section~6.4]{Jackson2023}.
For $\alpha_1\alpha_2\ne0$ with $\alpha_1^{-1}+\alpha_2^{-1}=1$, existence and uniqueness in $\Sp^\infty\times\BMO$ for every bounded terminal condition were established in \cite[Corollary~2.25]{FHT2025}.
We prove that every bounded terminal condition $\xi$ admits a solution in $\Sp^\infty\times\BMO$ for every $\theta\in\R^2$.
The solution is unique for every bounded $\xi$ if $(\alpha_1-1)(\alpha_2-1)=1$, but for every $\theta$ satisfying $(\alpha_1-1)(\alpha_2-1)\ne1$ we construct a bounded $\xi$ admitting at least two distinct solutions in $\Sp^\infty\times\BMO$.
Nevertheless, for each $\theta\in\R^2$, uniqueness is generic in $\xi$: for every $M>0$, the space $X_M$ of $[-M,M]^2$-valued terminal conditions, equipped with the topology of convergence in probability, contains a dense $G_\delta$ set $\mathcal G_M^\theta$ on which the solution is unique.
Moreover, if $\xi_n\to\xi\in\mathcal G_M^\theta$ in $X_M$, every choice of solutions for $\xi_n$ converges to the unique solution for $\xi$ in $\Sp^p\times\Hp^p$ for every $1\le p<\infty$.
\end{abstract}
\maketitle
\section{Introduction and main results}\label{sec:intro}

In the stochastic differential game of \cite[Section~6.4]{Jackson2023}, two players choose efforts $u_1,u_2$ and the common state diffusion has drift $u_1+u_2$.
Player $i$ pays $u_i^2/2$ per unit time and receives a bounded terminal payoff $\xi_i$ that may depend on the entire state path.
The Nash equilibria in BMO controls are characterized by a system of quadratic backward stochastic differential equations (BSDEs).
Whether the system has a unique solution in $\Sp^\infty\times\BMO$ for every bounded $\xi=(\xi_1,\xi_2)$ was left open.

Let $z_i$ denote the Brownian integrand of player $i$'s continuation value at a given time. Against an opponent's effort $v$, the expression $z_i(a+v)-a^2/2$ is maximized at $a=z_i$. At equilibrium $v=z_j$ for $j\ne i$, giving the generator component $z_i z_j+z_i^2/2$.
If player $i$ evaluates her net payoff by its exponential certainty equivalent with risk sensitivity $\gamma_i\in\R$, the corresponding generator component is $z_i z_j+(1-\gamma_i)z_i^2/2$. At $\gamma_i=0$ the criterion is ordinary expectation.

Let $W$ be a one-dimensional Brownian motion under the reference probability $\PP$, with its augmented natural filtration $(\FF_t)_{0\le t\le T}$, where $0<T<\infty$.
Write $\theta=(\alpha_1,\alpha_2)\in\R^2$, with $\alpha_i=1+\gamma_i$ in the game above. The BSDE is
\begin{equation}\label{eq:bsde}
 Y_t=\xi+\int_t^Tq_\theta(Z_s)\ds-\int_t^TZ_s\dd W_s,
 \qquad \xi\in L^\infty(\FF_T;\R^2),
\end{equation}
with generator
\begin{equation}\label{eq:generator}
 q_{\theta,i}(z)=z_i(z_1+z_2)-\tfrac{\alpha_i}{2}z_i^2,
 \qquad z=(z_1,z_2)\in\R^2,\quad i=1,2 .
\end{equation}
In the game with drift $u_1+u_2$ and running costs $u_i^2/2$, player $i$'s unique best response to $u_j$ is the Brownian integrand $H_i^{u_j}$ of her continuation value (Lemma~\ref{lem:game-best-response}). At a Nash equilibrium $u_i=H_i^{u_j}$ for both players, so their value equations give \eqref{eq:bsde} with $Z=u$.
Conversely, given a solution $(Y,Z)\in\Sp^\infty\times\BMO$, let $Q_i^{Z_j}$ be the measure induced by $u_i=0$ and $u_j=Z_j$, where $j=3-i$. Under $Q_i^{Z_j}$ the process $W_t^{Q_i^{Z_j}}=W_t-\int_0^t Z_{j,s}\ds$ is Brownian. Substitution in the $i$th equation cancels the cross term $Z_iZ_j$, giving the scalar value equation \eqref{eq:game-best-response} against $v=Z_j$ with generator $z_i\mapsto(1-\gamma_i)z_i^2/2$. Its unique bounded solution has integrand $Z_i$, so $Z_i$ is the best response to $Z_j$. Thus $(Y,Z)\mapsto u=Z$ is a bijection from bounded solutions to Nash equilibria in BMO controls, and $Y_i$ is player $i$'s equilibrium continuation value.
The choice $\alpha_1=\alpha_2=1$ gives the risk-neutral game of \cite[equation~(6.11)]{Jackson2023}. More generally, suppose that the drift is $b_1u_1+b_2u_2$, player $i$ pays $r_i u_i^2/2$ per unit time, and $b_i\ne0$, $r_i>0$, $\gamma_i\in\R$ for $i=1,2$. For terminal payoffs $\eta_i$ in the original units, set $\kappa_i=b_i^2/r_i$. The BSDE parameters are then $\alpha_i=1+\gamma_i/\kappa_i$, its terminal condition has components $\xi_i=\kappa_i\eta_i$, and its solutions give the equilibrium efforts $u_i=Z_i/b_i$ (Proposition~\ref{prop:game-verification}).

For the same $b_i,r_i,\gamma_i$, Section~\ref{subsec:game-extension} considers running costs $r_i(u_i-S_{ij}u_j)^2/2$, where $j=3-i$ and $S$ is a constant $2\times2$ matrix with zero diagonal. Existence of a Nash equilibrium in BMO controls for every bounded terminal payoff vector is equivalent to invertibility of $I-S$ (Propositions~\ref{prop:game-tracking-regular} and~\ref{prop:game-tracking}). If $I-S$ is singular, then for each $M>0$ the space of $[-M,M]^2$-valued terminal payoff vectors $\eta$, equipped with the topology of convergence in $\PP$-probability, contains a dense $G_\delta$ set on which no Nash equilibrium in BMO controls exists (Theorem~\ref{thm:game-tracking-generic}).

A bounded solution is a pair $(Y,Z)$ satisfying \eqref{eq:bsde} with $Y$ bounded, continuous and adapted, and $Z$ predictable with $\int_0^T|Z_t|^2\dt<\infty$ almost surely.
Proposition~\ref{prop:bounds} shows that every such $Z\cdot W$ is a martingale of bounded mean oscillation.
Thus the set $\Sol_\theta(\xi)$ of bounded solutions is exactly the set of solutions in $\Sp^\infty\times\BMO$.
Equality of solutions means indistinguishability of $Y$ and equality of $Z$ almost everywhere for $dt\otimes d\PP$.

For a general two-dimensional quadratic BSDE, nonexistence can occur even for bounded $\xi$, as shown in \cite[Theorem~2.1]{FreiDosReis2011}.
The small-data theorem in \cite{Tevzadze2008} gives a solution when $\|\xi\|_{L^\infty}$ is sufficiently small.
For $\alpha_1=\alpha_2=1$ and $\xi=g(W_T)$ with $g$ bounded and H\"older continuous, the Markovian result in \cite[Proposition~3.6]{XZ2018} gives an equilibrium.
For \eqref{eq:generator} with $\alpha_1\alpha_2\ne0$ and $\alpha_1^{-1}+\alpha_2^{-1}=1$, existence and uniqueness in $\Sp^\infty\times\BMO$ for every bounded $\xi$ were proved in \cite[Corollary~2.25]{FHT2025}. We call the curve $(\alpha_1-1)(\alpha_2-1)=1$ in the full parameter plane the critical curve. Its only point on the coordinate axes is $(0,0)$. There $q_{(0,0)}(z)=(z_1+z_2)z$, and the scalar equation for $Y_1+Y_2$ has generator $v\mapsto v^2$. Hence existence and uniqueness in $\Sp^\infty\times\BMO$ at the origin also follow from \cite[Theorem~3.2]{CheriditoNam2015}.

\begin{theorem}[Universal existence and a criterion for universal uniqueness]\label{thm:universal}
Fix $0<T<\infty$ and $\theta\in\R^2$.
For every $\xi\in L^\infty(\FF_T;\R^2)$, equation~\eqref{eq:bsde} has a solution in $\Sp^\infty\times\BMO$.
If $(\alpha_1-1)(\alpha_2-1)=1$, the solution is unique in $\Sp^\infty\times\BMO$ for every such $\xi$.
If $(\alpha_1-1)(\alpha_2-1)\ne1$, there exists $\xi\in L^\infty(\FF_T;\R^2)$ for which \eqref{eq:bsde} has at least two distinct solutions in $\Sp^\infty\times\BMO$.
\end{theorem}

In the risk-neutral game, $\alpha_1=\alpha_2=1$ gives $(\alpha_1-1)(\alpha_2-1)=0$. The game therefore admits a Nash equilibrium for every bounded terminal payoff vector $\xi$, and for some such $\xi$ it admits at least two distinct equilibria.
Even when $(\alpha_1-1)(\alpha_2-1)\ne1$, uniqueness is generic among the $[-M,M]^2$-valued terminal conditions for each $M>0$.

For $M>0$, let
\begin{equation}\label{eq:terminalspace}
 X_M=\{\xi\in L^2(\FF_T;\R^2): |\xi_i|\le M\text{ a.s.},\ i=1,2\},
 \qquad d(\xi,\eta)=\norm{\xi-\eta}_{L^2}.
\end{equation}
Then $(X_M,d)$ is complete and separable.
Its topology agrees with convergence in probability and with convergence in $L^p$ for every $1\le p<\infty$.
Genericity below refers to this topology on $X_M$.
The norms in the stability statement are $\|Y\|_{\Sp^p}=\|\sup_{t\le T}|Y_t|\|_{L^p}$ and $\|Z\|_{\Hp^p}=\|(\int_0^T|Z_t|^2\dt)^{1/2}\|_{L^p}$.

\begin{theorem}[Generic uniqueness and stability]\label{thm:main}
Fix $\theta\in\R^2$ and $M>0$. There is a dense $G_\delta$ set $\mathcal G_M^\theta\subset X_M$ such that every $\xi\in\mathcal G_M^\theta$ admits a unique solution $(Y^\xi,Z^\xi)\in\Sp^\infty\times\BMO$.
If $\xi_n\in X_M$ converge in probability to $\xi\in\mathcal G_M^\theta$, then every choice $(Y^n,Z^n)\in\Sol_\theta(\xi_n)$ satisfies
\begin{equation}\label{eq:mainstability}
 \|Y^n-Y^\xi\|_{\Sp^p}+\|Z^n-Z^\xi\|_{\Hp^p}\longrightarrow0
 \qquad (1\le p<\infty).
\end{equation}
\end{theorem}

For fixed $\theta$ and $M$, if $X_M$ contains a terminal condition with multiple solutions, then such terminal conditions form a dense meagre subset of $X_M$ (Corollary~\ref{cor:exceptional}).
In this case every $\xi\in\mathcal G_M^\theta$ can be approached by terminal conditions with multiple solutions. Along any such sequence, every choice of solutions converges to $(Y^\xi,Z^\xi)$ as in \eqref{eq:mainstability}.

Appendix~\ref{sec:parameters} proves generic uniqueness for terminal conditions with values in the three-point set of Theorem~\ref{thm:encoding} and extends Theorem~\ref{thm:main} to joint variation of $\theta$ and $\xi$.

Related stability estimates for multidimensional quadratic BSDEs use quantitative BMO bounds~\cite{HR2019}, triangular structure~\cite{JZ2022}, or weak interactions~\cite{HSZ2026}.
Generic well-posedness with a fixed terminal condition and a varying generator was studied in \cite{BMO2004}.
Here $q_\theta$ is fixed. A dense set of terminal conditions at which \eqref{eq:mainstability} holds comes from solutions whose transformed states move along successive line segments (Definition~\ref{def:core}).
Every $\xi\in X_M$ taking at most two values has such a solution.
This gives a natural threshold for the finite-valued examples of nonuniqueness.

\begin{theorem}[Nonuniqueness for terminal conditions with discrete or absolutely continuous laws]\label{thm:terminal-laws}
Fix $0<T<\infty$ and $\theta\in\R^2$. Every terminal condition $\xi$ taking at most two values admits a unique solution in $\Sp^\infty\times\BMO$.
Suppose $(\alpha_1-1)(\alpha_2-1)\ne1$. For every integer $n\ge3$ there exists a bounded $\xi$ taking exactly $n$ values, each with positive $\PP$-probability, for which \eqref{eq:bsde} has at least two distinct solutions in $\Sp^\infty\times\BMO$.
There are also bounded $\xi^{\mathrm{at}}$ and $\xi^{\mathrm{ac}}$, each admitting at least two distinct solutions in $\Sp^\infty\times\BMO$. Under $\PP$, the law of $\xi^{\mathrm{at}}$ is purely atomic with infinitely many atoms, and the law of $\xi^{\mathrm{ac}}$ has a density in $C_c^\infty(\R^2)$.
\end{theorem}

At $\alpha_1=\alpha_2=1$, the generator \eqref{eq:generator} belongs to both classes in \cite[Questions~6.25 and~6.26]{Jackson2023}. The three-valued terminal condition in Proposition~\ref{prop:three-counterexample} therefore rules out universal uniqueness in either class.
For a fixed bounded $\xi$, the cardinality of $\Sol_\theta(\xi)$ is one or the continuum (Corollary~\ref{cor:cardinality}).

\subsection{Outline of the proofs}

We begin in Section~\ref{subsubsec:outline-segment} with an explicit solution formula for the BSDE \eqref{eq:bsde} between two stopping times $\sigma\le\tau$.
Set $u_\theta(y)_i=\int_0^{y_i}e^{-\alpha_i s}\dd s$. The formula applies to bounded terminal conditions $Y_\tau=\xi$ for which $u_\theta(\xi)$ lies almost surely on a line segment whose endpoints are known at time $\sigma$.
Sections~\ref{subsubsec:outline-existence}--\ref{subsubsec:outline-generic} then address the following three questions about the solution set $\Sol_\theta(\xi)$ on $[0,T]$.
\begin{enumerate}[label=\textup{(\arabic*)},nosep]
\item For fixed $\theta\in\R^2$, is $\Sol_\theta(\xi)$ nonempty for every bounded terminal condition $\xi$?
\item For which $\theta$ is $\Sol_\theta(\xi)$ a singleton for every bounded $\xi$? For each remaining $\theta$, how can we construct two distinct solutions with the same bounded terminal condition $\xi$?
\item Fix $\theta\in\R^2$ and $M>0$. At which $\xi\in X_M$ is the solution unique and stable under perturbations in $X_M$?
\end{enumerate}

\subsubsection{An explicit solution formula when \texorpdfstring{$u_\theta(\xi)$}{the transformed terminal condition} lies on a line segment}\label{subsubsec:outline-segment}
On $[\sigma,\tau]$, any bounded solution $(Y,Z)$ with $Y_\tau=\xi$ has transformed state $U=u_\theta(Y)$ satisfying $U_t=\E^{\PP^Z}[u_\theta(\xi)\mid\FF_t]$ by \eqref{eq:conditionalU}.
The measure $\PP^Z$ depends on the unknown integrand through its Girsanov kernel $Z_1+Z_2$.
Write $U_\tau=A+er_\tau$, where $A,e$ are $\FF_\sigma$-measurable and $r_\tau$ is $\FF_\tau$-measurable with values in $[0,1]$. Assume $|e|>0$ almost surely and $A+e[0,1]$ is contained in a fixed compact rectangle in $u_\theta(\R^2)$.
Conditional expectation under $\PP^Z$ keeps $U_t=A+er_t$ on that segment. The scale function $\gamma_{A,e}$ makes $\gamma_{A,e}(r)$ a bounded martingale under the original measure $\PP$. The segment formula \eqref{eq:segmentformula} therefore gives
\[
 Y_t=u_\theta^{-1}\!\left(A+e\,\gamma_{A,e}^{-1}\!\left(
       \E^{\PP}\!\left[\gamma_{A,e}(r_\tau)\mid\FF_t\right]\right)\right),
 \qquad \sigma\le t\le\tau .
\]
The proof of Lemma~\ref{lem:segment} recovers $Z$ from the Brownian representation $\dd P_t=\zeta_t\dd W_t$ of $P_t=\E^{\PP}[\gamma_{A,e}(r_\tau)\mid\FF_t]$, giving
$Z_{i,t}=e_i\zeta_t/\bigl((1-\alpha_iU_{i,t})\gamma_{A,e}'(r_t)\bigr)$ for $i=1,2$.
At a node of a binary tree the two transformed child values are the endpoints of such a segment, with a constant solution when they coincide. Solving backwards along finite trees and realizing their branches by Brownian exits gives piecewise segment solutions whose terminal conditions are dense in $X_M$ (Lemma~\ref{lem:tree-density}).

\subsubsection{Universal existence}\label{subsubsec:outline-existence}
Fix $\theta\in\R^2$ and suppose existence fails. Choose a bounded $\xi$ with no solution and a rectangle $B$ containing $u_\theta(\xi)$ almost surely. Choose the pair so that $w(B)$, the sum of the side lengths of $u_\theta^{-1}(B)$, is close to the infimum $w_*$ over all such pairs.
Bisect $B$ once at the midpoint $m$ of its $i$th side, choosing $i$ along the longer side of $u_\theta^{-1}(B)$. At each stage select the more probable of $G=\{(u_\theta(\xi))_i\le m\}$ and $G^c$ given the current history. If the conditional probability of the selected event reaches $1/4$ before $T$, a new interval starts at the first such time. We call this a return.
For a history $h$ at stage $n$, let $K_n(h)$ consist of the limits of transformed initial values of solutions whose terminal conditions approximate $\xi$ given $h$ over the entire remaining horizon, as in \eqref{eq:ex-approxdata}. On return at $h'$, the moment bounds for the Girsanov densities in Lemma~\ref{lem:density} keep $K_{n+1}(h')$ away from the face opposite the selected half. On nonreturn $u_\theta(\xi)$ lies in that half. For almost every starting history, one rectangle contains both kinds of endpoint (Lemma~\ref{lem:ex-roots}). The decrease in width has a uniform positive lower bound. For $w(B)$ sufficiently close to $w_*$, the new width is less than $w_*$, so every admissible terminal state on the current interval admits a solution.

At a return, the next solution must start at the value just reached. Starting from $K=(K_n)$, we replace each set by the closed convex hull of the transformed initial values attainable with the next-stage set as constraint on return. This iteration yields a fixed family $C=(C_n)$ satisfying \eqref{eq:ex-fixed}. Lemma~\ref{lem:fine-perturbation} allows us to move a return endpoint into $\operatorname{ri}C_{n+1}(h')$ with arbitrarily small change in the transformed initial value. The attainable initial values under this stricter constraint still have closed convex hull $C_n(h)$. By Theorem~\ref{thm:attainable-interior}, for almost every $h$ every $x\in\operatorname{ri}C_n(h)$ is then attained with the return endpoint in $\operatorname{ri}C_{n+1}(h')$ (Proposition~\ref{prop:ex-exact}). Section~\ref{subsec:ex-pasting} uses Lemma~\ref{lem:history-selection} to select solutions at these prescribed initial values. By \eqref{eq:ex-returns}, the conditional probability of another return is at most $2/3$. The concatenation therefore reaches $\xi$ after finitely many intervals almost surely, and the uniform energy bound gives $Z\in\BMO$.

\subsubsection{Universal uniqueness on the critical curve and nonuniqueness off it}
On the critical curve $(\alpha_1-1)(\alpha_2-1)=1$, the process $L=(\alpha_1-1)Y_1-Y_2$ satisfies a scalar BSDE. Its integrand determines a change of measure under which the equation for $Y_2$ is also scalar. Solving the two equations in succession determines $(Y,Z)$ uniquely, as shown at the beginning of Section~\ref{sec:nonuniqueness}.

Off the critical curve, let an auxiliary diffusion start at $a$, with continuation probability scaled by $c\in(0,1)$. In the linear coordinates of \eqref{eq:nu-transformed}, the reference solution's value at absorption defines $\xi_{a,c}$. A continuous function $x$ assigns to each state a correction to its reference continuation value in rotating coordinates. Solving one block backwards maps $x$ to $\mathcal F_c x$, and successive blocks agree at a fixed point. The reference solution gives the fixed point $0$ for every $c$. The linearization at $0$ has a simple eigenvalue $c\lambda_+$ with $\lambda_+>1$ (Lemma~\ref{lem:nu-operator}). Bifurcation from this simple eigenvalue~\cite[Theorem~1.7]{CR1971} gives the branch of nonzero fixed points in \eqref{eq:nu-branch} at $c=\lambda_+^{-1}$. Choose a nearby point $(c,x)$ on this branch with $x\ne0$, and then $a$ with $x(a)\ne0$. Lemma~\ref{lem:nu-paste} gives two solutions with terminal condition $\xi_{a,c}$ and different initial values.
In the original coordinates, Theorem~\ref{thm:encoding} extends both solutions to a three-valued terminal condition and preserves the entire solution set. Further scalar continuations give the discrete examples and the example with a smooth density in Theorem~\ref{thm:terminal-laws} (Section~\ref{subsec:terminal-laws}).

\subsubsection{Generic uniqueness and stability}\label{subsubsec:outline-generic}
Fix $\theta\in\R^2$ and $M>0$. Suppose $\xi_n\to\bar\xi$ in $X_M$, where $\bar\xi$ admits a piecewise segment solution $(\bar Y,\bar Z)$, and choose any $(Y^n,Z^n)\in\Sol_\theta(\xi_n)$.
Write $U^n=u_\theta(Y^n)$, $\bar U=u_\theta(\bar Y)$ and $Q_n=\PP^{Z^n}$. On each reference interval $[\sigma,\tau]$, set $\widehat U_t^n=\E^{Q_n}[\bar U_\tau\mid\FF_t]$. Since $U^n$ and $\widehat U^n$ are $Q_n$-martingales, convergence of $U_\tau^n$ to $\bar U_\tau$ controls their difference throughout the interval. The process $\widehat U^n$ stays on the reference segment and ends at $\bar U_\tau$. By \eqref{eq:forcing0}, the drift of its scaled scalar coordinate under $\PP$ has total variation tending to zero in $L^p(\PP)$ for every $1\le p<\infty$. Working backwards yields convergence in $\Sp^p\times\Hp^p$ for every $1\le p<\infty$ and every choice of solutions (Theorem~\ref{thm:corestability}).

At each such $\bar\xi$, stability forces the diameter in $\Sp^2\times\Hp^2$ of the union of solution sets over a ball about $\bar\xi$ to tend to zero with its radius. By Lemma~\ref{lem:abstract}, the points at which this diameter tends to zero form a $G_\delta$ set containing the dense set of reference terminal conditions. At each such point the solution set has diameter zero, and universal existence makes it a singleton. The uniform state and energy bounds of Proposition~\ref{prop:bounds} extend stability to $\Sp^p\times\Hp^p$ for every $1\le p<\infty$.

\section{Bounded solutions in martingale coordinates}\label{sec:estimates}

Fix $\theta=(\alpha_1,\alpha_2)\in\R^2$.
For each bounded solution $(Y,Z)$, both components of the transformed state $u_\theta(Y)$ are martingales under the measure $\PP^Z$ induced by $Z$.
The transformed equation gives uniform bounds for $Y$ and for the BMO norm of $Z$ in terms of the terminal range. When the transformed terminal state lies on a line segment known at the start of the interval, a scalar martingale determines the solution. Joining segment solutions along finite binary trees gives approximations to prescribed terminal conditions on the original Brownian filtration.

For $1\le p<\infty$ the norms of the state $Y$ and the integrand $Z$ are
\[
 \|Y\|_{\Sp^p}=\left(\E\sup_{t\le T}|Y_t|^p\right)^{1/p},
 \qquad \|Z\|_{\Hp^p}=\left(\E\left(\int_0^T|Z_t|^2\dt\right)^{p/2}\right)^{1/p}.
\]
We use $\Sp^\infty$ for essentially bounded continuous adapted states and write
\[
 \|Z\|_{\BMO}^2=\sup_{\tau\le T}
 \left\|\E\left[\int_\tau^T |Z_t|^2\dt\,\middle|\,\FF_\tau\right]\right\|_\infty ,
\]
where $\tau$ ranges over stopping times. On a stochastic interval $[\sigma,\tau]$, the state norms use the supremum over that interval and the integrand norms use the extension by zero outside $(\sigma,\tau]$. We omit the interval from the notation when it is fixed.

\subsection{Uniform estimates in \texorpdfstring{$\Sp^\infty\times\BMO$}{S\textasciicircum\unichar{"221E}\space\unichar{"00D7}\space BMO}}

For $\alpha,y\in\R$, let $\phi_\alpha(y)=\int_0^y e^{-\alpha s}\ds$. Thus $\phi_0(y)=y$ and $\phi_\alpha(y)=(1-e^{-\alpha y})/\alpha$ for $\alpha\ne0$.
Since $\phi_\alpha'(y)=e^{-\alpha y}=1-\alpha\phi_\alpha(y)>0$, the map
\begin{equation}\label{eq:chart}
 u_\theta(y)=(\phi_{\alpha_1}(y_1),\phi_{\alpha_2}(y_2))
\end{equation}
is a smooth bijection from $\R^2$ onto the open rectangle $D_\theta=\{x\in\R^2:1-\alpha_i x_i>0,\ i=1,2\}$.
Write $v_i=\phi_{\alpha_i}^{-1}$ and $v_\theta(x)=(v_1(x_1),v_2(x_2))$ for the inverse of $u_\theta$.
For fixed $\theta$ we abbreviate $u_\theta,v_\theta$ to $u,v$.
For a bounded solution $(Y,Z)$ set $b=Z_1+Z_2$ and
\begin{equation}\label{eq:coordinates}
 U=u(Y),\qquad V_i=(1-\alpha_i U_i)Z_i,\qquad
 a_{\theta,i}(x)=-\frac{1}{1-\alpha_i x_i},\quad i=1,2 .
\end{equation}
On $D_\theta$ we have $a_\theta=\nabla\log G_\theta$ with $G_\theta(x)=\exp(-v_1(x_1)-v_2(x_2))$.
The identity $\phi_\alpha''=-\alpha\phi_\alpha'$ and It\^o's formula give
\begin{equation}\label{eq:Uequation}
 \dd U_t=V_t\big(a_\theta(U_t)\cdot V_t\big)\dt+V_t\dd W_t,
 \qquad b_t+a_\theta(U_t)\cdot V_t=0 .
\end{equation}

\begin{proposition}[Uniform state and energy bounds]\label{prop:bounds}
Let $M>0$ and $L\ge0$. For every $\theta$ with $\max_i|\alpha_i|\le L$, every $\xi\in X_M$, and every bounded solution $(Y,Z)$ with terminal condition $\xi$,
\begin{equation}\label{eq:mainbounds}
 |Y_{i,t}|\le M\quad (0\le t\le T,\ i=1,2),\qquad
 \norm{Z}_{\BMO}^2\le C(M,L).
\end{equation}
In particular, for every $1\le p<\infty$,
\begin{equation}\label{eq:allmoments}
 \sup_{\max_i|\alpha_i|\le L}\ \sup_{\xi\in X_M}\ \sup_{(Y,Z)\in\Sol_\theta(\xi)}
 \E\Big(\int_0^T|Z_t|^2\dt\Big)^{p/2}<\infty.
\end{equation}
\end{proposition}
\begin{proof}
First let $K\subset D_\theta$ be a compact rectangle containing the range of $U$. Such a rectangle exists since $Y$ is bounded.
Set $\lambda=1+\tfrac12\sup_K|a_\theta|^2$ and $H(x)=\exp(\lambda|x|^2)$. For $x\in K$, $a=a_\theta(x)$ and $w\in\R^2$, completing the square gives
\begin{equation}\label{eq:energyidentity}
\begin{aligned}
 \frac{\tfrac12D^2H(x)[w,w]+(\nabla H(x)\cdot w)(a\cdot w)}{H(x)}
 &=\lambda|w|^2-\tfrac12(a\cdot w)^2
   +2\bigl(\lambda x\cdot w+\tfrac12a\cdot w\bigr)^2\\
 &\ge |w|^2 .
\end{aligned}
\end{equation}
Since $H\ge1$, localization and conditional Fatou give
\[
 \E\!\left[\int_\rho^T|V_s|^2\ds\,\middle|\,\FF_\rho\right]
 \le \sup_K H-\inf_K H
\]
for every stopping time $\rho\le T$. The factors $(1-\alpha_i U_i)^{-1}$ are bounded on $K$, so $Z\in\BMO$.

Consequently $b\cdot W$ is a continuous BMO martingale and its stochastic exponential is uniformly integrable~\cite{Kazamaki1994}. Let $Q=\PP^Z$ be given by
\begin{equation}\label{eq:Q}
 D=\cE(b\cdot W)_T,\qquad \dd Q=D\dd\PP,
 \qquad W_t^Q=W_t-\int_0^t b_s\ds .
\end{equation}
Under $Q$, \eqref{eq:Uequation} reads $\dd U=V\dd W^Q$. Since $U$ is bounded,
\begin{equation}\label{eq:conditionalU}
 U_t=\E^Q[u(\xi)\mid\FF_t].
\end{equation}
When $\alpha_i\ne0$, applying the affine function $x\mapsto1-\alpha_i x$ to the $i$th coordinate of \eqref{eq:conditionalU} gives $e^{-\alpha_iY_{i,t}}=\E^Q[e^{-\alpha_i\xi_i}\mid\FF_t]$. When $\alpha_i=0$, that coordinate gives $Y_{i,t}=\E^Q[\xi_i\mid\FF_t]$.
In particular, $U$ lies in the deterministic rectangle $u([-M,M]^2)$, and the monotonicity of each $\phi_{\alpha_i}$ yields $|Y_i|\le M$.

Apply \eqref{eq:energyidentity} again on $K=u([-M,M]^2)$. There $|U|\le\sqrt2 M e^{LM}$, $|a_\theta(U)|\le\sqrt2 e^{LM}$, and $1-\alpha_iU_i\ge e^{-LM}$.
The conditional energy bound $\|V\|_{\BMO}^2\le\sup_K H-\inf_K H$ and the identity $Z_i=V_i/(1-\alpha_iU_i)$ give the uniform BMO bound for $Z$ in \eqref{eq:mainbounds}.
Iteration yields
\begin{equation}\label{eq:factorial}
 \E\left[\Big(\int_\rho^T|Z_t|^2\dt\Big)^k\,\middle|\,\FF_\rho\right]
 \le k!\norm{Z}_{\BMO}^{2k},\qquad k\in\N,
\end{equation}
and hence \eqref{eq:allmoments}.
\end{proof}

For $\alpha_1,\alpha_2>0$ one also has $\|Z\|_{\BMO}^2\le(e^{8M}-1)/\min(\alpha_1,\alpha_2)$.
Indeed, for $S=Y_1+Y_2$, It\^o's formula gives $\dd e^{2S}=e^{2S}\sum_i\alpha_i Z_i^2\dt+2e^{2S}b\dd W$. The estimate follows by conditional expectation and $e^{-4M}\le e^{2S}\le e^{4M}$.

Identity \eqref{eq:conditionalU} places $U$ in the closed convex hull of the essential range of $u(\xi)$. More generally, for stopping times $\sigma\le\tau\le T$, if an $\FF_\sigma$-measurable closed convex set contains $U_\tau$ almost surely, it contains $U_t$ for all $\sigma\le t\le\tau$ almost surely.

\subsection{Moment bounds for the Girsanov densities}

\begin{lemma}[Moment estimates for the densities]\label{lem:density}
There exist $r>1$, $\varepsilon>0$ and $C<\infty$, depending only on $M$ and $\theta$, such that for every bounded solution with terminal condition $\xi\in X_M$ the Girsanov density $D$ in \eqref{eq:Q} satisfies
\begin{equation}\label{eq:densitymoments}
 \E D=1,\qquad \E D^r\le C,\qquad \E D^{-\varepsilon}\le C.
\end{equation}
Let $D_t=\cE(b\cdot W)_t$ denote the density process, so $D_T=D$. For all stopping times $\sigma\le\tau\le T$,
\begin{equation}\label{eq:conditional-density}
 \E\big[(D_\tau/D_\sigma)^r+(D_\tau/D_\sigma)^{-\varepsilon}\mid\FF_\sigma\big]\le 2C .
\end{equation}
The constants are uniform when $\theta$ ranges over a compact subset of $\R^2$.
\end{lemma}
\begin{proof}
Proposition~\ref{prop:bounds} bounds $\norm{b}_{\BMO}$ uniformly. The martingale property and positive moment estimate follow from the reverse H\"older inequality for stochastic exponentials of continuous BMO martingales~\cite{Kazamaki1994}. For the negative moment, put $L=b\cdot W$ and $A=\langle L\rangle_T$. The identity
\[
 D^{-\varepsilon}
 =\cE(-2\varepsilon L)_T^{1/2}
   \exp\big((\varepsilon/2+\varepsilon^2)A\big)
\]
and Cauchy--Schwarz give
\[
 \E D^{-\varepsilon}
 \le\left(\E e^{(\varepsilon+2\varepsilon^2)A}\right)^{1/2}.
\]
Choose $\varepsilon$ so that $(\varepsilon+2\varepsilon^2)\norm{b}_{\BMO}^2<1$. The exponential series and \eqref{eq:factorial} bound the right-hand side. Uniformity for $\theta$ in compact subsets of $\R^2$ follows from the uniform bound in Proposition~\ref{prop:bounds}.
For \eqref{eq:conditional-density}, apply conditional reverse H\"older to the stopped exponential. For the inverse power use the same identity with the increment $L_\tau-L_\sigma$ and its bracket, conditional Cauchy--Schwarz, and the conditional estimate \eqref{eq:factorial} for $b$.
\end{proof}

We record the two forms of H\"older's inequality used below. For $p\ge1$, a random variable $H$, and $r'=r/(r-1)$,
\begin{align}
 \E^Q|H|^p
 &\le(\E D^r)^{1/r}(\E|H|^{pr'})^{1/r'},\label{eq:PtoQ}\\
 \E|H|^p
 &\le(\E D^{-\varepsilon})^{1/(1+\varepsilon)}
 \left(\E^Q|H|^{p(1+\varepsilon)/\varepsilon}\right)^{\varepsilon/(1+\varepsilon)}.
 \label{eq:QtoP}
\end{align}
Consequently, if $(H_n)$ is uniformly bounded and converges to zero in $\PP$-probability, then $\E^{Q_n}|H_n|^p\to0$ for every $1\le p<\infty$, provided the densities of $Q_n$ satisfy \eqref{eq:densitymoments} uniformly in $n$. Conversely, convergence to zero in $L^p(Q_n)$ for every $1\le p<\infty$ implies convergence to zero in $L^p(\PP)$ for every such $p$, by \eqref{eq:QtoP}.

\subsection{Solutions on random line segments}
For a segment $A+e[0,1]$ contained in $D_\theta$, define
\begin{equation}\label{eq:gamma}
 \gamma_{A,e}(r)
 =\frac{\displaystyle\int_0^rG_\theta(A+ex)^{-2}\dd x}
        {\displaystyle\int_0^1G_\theta(A+ex)^{-2}\dd x},
 \qquad 0\le r\le1.
\end{equation}
We write $\gamma^\theta_{A,e}$ when the dependence on $\theta$ is needed. This is an increasing smooth bijection of $[0,1]$. If the segments lie in a fixed compact rectangle in $D_\theta$, $\gamma'$, $1/\gamma'$ and $\gamma''$ have uniform bounds. Indeed, $G_\theta$ is bounded above and away from zero there, and $|e|$ is bounded by the diameter of the rectangle.

\begin{lemma}[Segment formula]\label{lem:segment}
Let $\sigma\le\tau$ be stopping times. Suppose $A,e$ are $\FF_\sigma$-measurable, $|e|>0$ almost surely, and $A+e[0,1]$ lies in a fixed compact rectangle in $D_\theta$. If $r_\tau$ is $\FF_\tau$-measurable with values in $[0,1]$, there is a unique solution $(Y,Z)$ in $\Sp^\infty\times\BMO$ on $[\sigma,\tau]$ with terminal condition $Y_\tau=v(A+er_\tau)$. For $\sigma\le t\le\tau$, its state is given by
\begin{equation}\label{eq:segmentformula}
 P_t=\E[\gamma_{A,e}(r_\tau)\mid\FF_t],\qquad
 U_t=A+e\gamma_{A,e}^{-1}(P_t),\qquad
 Y_t=u_\theta^{-1}(U_t).
\end{equation}
\end{lemma}
\begin{proof}
Consider any bounded solution. The argument of Proposition~\ref{prop:bounds}, applied on $[\sigma,\tau]$, gives $Z\in\BMO$. The change of measure \eqref{eq:Q}, with its kernel restricted to this interval, implies that $U$ is a bounded martingale taking values in the $\FF_\sigma$-measurable segment. Projection onto its direction gives $U=A+er$ and $V=eH$, where $r=e\cdot(U-A)/|e|^2$ and $H=e\cdot V/|e|^2$. These identities and the following It\^o calculation may first be established on the $\FF_\sigma$-measurable events $\{|e|\ge1/m\}$ and then extended by letting $m\to\infty$. Equation \eqref{eq:Uequation} becomes
\[
 \dd r=k(r)H^2\dt+H\dd W,
 \qquad k(x)=a_\theta(A+ex)\cdot e.
\]
Since $\gamma''=-2k\gamma'$, It\^o's formula makes $\gamma(r)$ a bounded $\PP$-martingale. Thus every bounded solution satisfies \eqref{eq:segmentformula}.

Conversely, let $\dd P=\zeta\dd W$ on $[\sigma,\tau]$ and put
\[
 r=\gamma^{-1}(P),\quad H=\frac{\zeta}{\gamma'(r)},\quad
 U=A+er,\quad V=eH,\quad Y=v(U),\quad
 Z_i=\frac{V_i}{1-\alpha_iU_i}.
\]
It\^o's formula applied to $\gamma^{-1}$ verifies \eqref{eq:Uequation} and hence \eqref{eq:bsde}. The multipliers of $\zeta$ in $Z$ are uniformly bounded. As $0\le P\le1$, we have $\zeta\in\BMO$, proving the assertion. All parameters of $\gamma$ are $\FF_\sigma$-measurable, so the formula applies with these random parameters on the stochastic interval.
\end{proof}

\subsection{Approximation of terminal conditions by finite binary trees}\label{subsec:finite-trees}

In a binary tree, Lemma~\ref{lem:segment} determines the transformed state at a vertex from the two possible states at its children and their conditional probabilities. Embedding the successive choices as Brownian exits lets us approximate a prescribed random vector $\eta$ on the original probability space while keeping every leaf in the range specified for $\eta$.

\begin{definition}\label{def:core}
A bounded solution is a \emph{piecewise segment solution} if there are a deterministic integer $N$, stopping times
\[
 0=\sigma_0\le\sigma_1\le\cdots\le\sigma_N=T,
\]
and $\FF_{\sigma_j}$-measurable pairs $(A_j,e_j)$ such that its transformed state $U$ lies in $A_j+e_j[0,1]$ on $[\sigma_j,\sigma_{j+1}]$. The segments must lie in one deterministic compact rectangle in $D_\theta$, and there must be a deterministic constant $c>0$ such that $|e_j|\ge c$ almost surely for every $j$. The constant may depend on the solution and the chosen segments. Constant pieces may be represented using any such segment through the constant value. Let $\mathcal D_M^\theta$ be the set of $\xi\in X_M$ for which $\Sol_\theta(\xi)$ contains a piecewise segment solution.
\end{definition}

\begin{lemma}[Approximation on a prescribed terminal range]\label{lem:tree-density}
Let $K$ be a nonempty compact subset of $D_\theta$ and let $\eta$ be an $\FF_T$-measurable random vector with values in $K$. There are piecewise segment solutions $(Y^m,Z^m)$ on $[0,T]$ such that $u(Y^m_T)$ takes finitely many values in $K$,
\[
 u(Y^m_T)\longrightarrow\eta\quad\hbox{in }L^p,\qquad 1\le p<\infty,
\]
and $u(Y^m)$ stays in $\conv K$.
\end{lemma}
\begin{proof}
First approximate $\eta$ by a map of finitely many Brownian observations taking finitely many values in $K$.
For this, quantize $\eta$ using a finite net in $K$.
The events determining its labels can be approximated in probability by the algebra generated by rational intervals for $W_t$ at rational times $0<t<T$, since this algebra generates $\FF_T$ up to completion.
Making the resulting label sets disjoint gives a function $f(W_{t_1},\ldots,W_{t_k})$, $0<t_1<\cdots<t_k<T$, whose discontinuities have measure zero under the joint Gaussian law.
It suffices to approximate each such function in probability, since $K$ is bounded.

The exit times of the embedded tree must all occur before $T$, while the Brownian observations up to $t_k$ must retain their original values. To achieve both, fix $s$ with $t_k<s<T$ and choose a deterministic $C^1$ increasing time change $A$ on $[0,T)$ with $A(t)=t$ on $[0,s]$, $A'(t)>0$, and $A(t)\uparrow\infty$ as $t\uparrow T$.
Then $B_r=\int_0^{A^{-1}(r)}\sqrt{A'(t)}\,\dd W_t$ is a Brownian motion in the time-changed filtration and agrees with $W$ up to $s$.
For $h>0$ set $\rho_0=0$ and let $\rho_{j+1}$ be the first time after $\rho_j$ when $|B_r-B_{\rho_j}|=h$.
Write $n_i=\lfloor t_i/h^2\rfloor$ and $N=n_k$, taking $h$ small enough that $1\le n_1<\cdots<n_k$.
The successive exit times have independent increments, with mean $h^2$ and variance $O(h^4)$ by Brownian scaling and the finite second moment of the exit time from $(-1,1)$.
Thus $\rho_{n_i}\to t_i$ in $L^2$. By continuity, $B_{\rho_{n_i}}\to W_{t_i}$ in probability for each $i$.
The choice of $f$ gives
\[
 f(B_{\rho_{n_1}},\ldots,B_{\rho_{n_k}})
 \ \longrightarrow\ f(W_{t_1},\ldots,W_{t_k})\quad\hbox{in probability}.
\]

Label each leaf of the depth-$N$ tree of exit signs by the value of $f$ on that path.
Proceed from the leaves to the root. If the children have labels $a,b$, label their parent by $a+(b-a)\gamma_{a,b-a}^{-1}(1/2)$, and give it label $a$ when $a=b$.
All labels lie in $\conv K$.
In physical time put $\sigma_j=A^{-1}(\rho_j)$.
These are stopping times and $\sigma_N<T$ almost surely because $\rho_N<\infty$.
On $[\sigma_j,\sigma_{j+1}]$ the conditional probability of an upper exit is
\[
 P_t=\tfrac12\big(1+(B_{A(t)}-B_{\rho_j})/h\big).
\]
Indeed the Brownian motion stopped at the two neighboring levels is a bounded martingale in the time-changed filtration.
With $a$ the lower child and $b$ the upper child, Lemma~\ref{lem:segment} gives the transformed state $U_t=a+(b-a)\gamma_{a,b-a}^{-1}(P_t)$ on $[\sigma_j,\sigma_{j+1}]$. For coincident labels use the constant solution.
This process starts at the parent label and ends at the selected child.
Set $\sigma_{N+1}=T$ and keep the leaf label constant on $[\sigma_N,\sigma_{N+1}]$.
The finitely many segment solutions agree at their endpoints. Their concatenation has bounded state and finite pathwise energy, and Proposition~\ref{prop:bounds} gives its BMO bound.
The tree has finitely many edges, and hence a positive lower bound on the lengths of its nonconstant segments whenever there are any. Constant pieces use a fixed direction in a larger compact rectangle.
This gives a piecewise segment solution with the displayed leaf value at time $T$.
A diagonal choice of cylinder approximations and mesh sizes proves the lemma.
\end{proof}

Appendix~\ref{app:trajectory} shows that, given a bounded solution $(Y,Z)$, one can choose piecewise segment solutions converging to $(Y,Z)$ in $\Sp^p\times\Hp^p$ for every $1\le p<\infty$, with the terminal condition varying along the sequence.

\section{Attainable initial values under random terminal constraints}\label{sec:attainability}

Fix $\theta=(\alpha_1,\alpha_2)\in\R^2$ and write $U=u(Y)$ for the coordinates \eqref{eq:chart}.
Given a stopping time $\tau$ and a random set $H$, let $I(H)$ consist of the initial values $U_0$ of solutions $(Y,Z)\in\Sp^\infty\times\BMO$ of \eqref{eq:bsde} on $[0,\tau]$ such that $U_\tau\in H$ almost surely.
The points of $I(H)$ are deterministic because $\FF_0$ is trivial up to null sets.
When two solutions are joined at a stopping time, the endpoint of the first prescribes the initial value of the second, conditional on the stopped past.
The following theorem gives a set of initial values for which this matching is exact.

\begin{theorem}[Attainment of relative interior points]\label{thm:attainable-interior}
Let $\tau$ be a stopping time of the augmented natural Brownian filtration, with $0<\tau\le T$ almost surely.
Suppose that $H$ has graph in $\FF_\tau\otimes\mathcal B(\R^2)$ and values contained in a common compact rectangle in $D_\theta$.
If $I(H)\ne\varnothing$, its closure is convex and
\begin{equation}\label{eq:attainable-interior}
 \operatorname{ri}\overline{I(H)}\subset I(H).
\end{equation}
Here the relative interior is taken in the affine hull of $\overline{I(H)}$.
\end{theorem}

Start with two solutions whose transformed terminal states belong to $H$. Choose a stopping time $\vartheta$ at which both transformed states are still close to their initial values. An event in $\FF_\vartheta$ selects which solution to follow on $[\vartheta,\tau]$. We solve backwards on $[0,\vartheta]$ with the selected value of $U_\vartheta$ as transformed terminal state.
As the switching probability varies, the transformed initial value of the joined solution traces a continuous curve approximating the segment between the two transformed initial values.
This proves convexity of $\overline{I(H)}$ and gives exact attainment in affine dimension one by the intermediate value theorem.
In dimension two we switch twice among three solutions.
The limiting map of the two switching probabilities covers the open triangle of initial values with degree $1$ or $-1$.
The estimates below give continuity of the perturbed map and uniform approximation on the boundary, so that each interior point retains its degree and is attained.

\subsection{Perturbations of segment solutions}

Abbreviate \(q_\theta,a_\theta\) to \(q,a\).
Distances and projections in \(D_\theta\) are Euclidean.
Fix a compact rectangle \(D\subset D_\theta\) and a compact family of ordered pairs \((A,B)\in D^2\) with \(\inf|B-A|>0\).
For \(S=[A,B]\), let \(\pi_S\) denote nearest-point projection.
We perturb the terminal state away from $S$ and select a solution which varies continuously with both the segment and the terminal state.

\begin{lemma}[Perturbations of a segment solution]\label{lem:segment-tube}
There are constants \(\varepsilon_0>0\) and \(C<\infty\) common to the family.
Let \(\sigma\le\tau\le T\) be stopping times, \(S=[A,B]\), and let \(\chi\) be an \(\FF_\tau\)-measurable \(D_\theta\)-valued terminal state with
\(\|\dist(\chi,S)\|_\infty\le\varepsilon\le\varepsilon_0\).
There is a choice of bounded solution \((Y,Z)\) on \([\sigma,\tau]\) with \(u(Y_\tau)=\chi\) such that, for the segment solution \((\bar Y,\bar Z)\) with transformed terminal state \(\pi_S\chi\),
\begin{equation}\label{eq:tube-estimate}
 \|Y-\bar Y\|_{\Sp^\infty}+\|Z-\bar Z\|_{\BMO}\le C\varepsilon.
\end{equation}
Fix \(\sigma,\tau\). Let \((A_m,B_m)\to(A,B)\) in the family and put \(S_m=[A_m,B_m]\). Suppose \(\chi_m\to\chi\) in probability and \(\|\dist(\chi_m,S_m)\|_\infty\le\varepsilon_0\) for each \(m\). Write \((Y^m,Z^m)\) for the solution selected for \(A_m,B_m,\chi_m\). Then \((Y^m,Z^m)\to(Y,Z)\) in \(\Sp^p\times\Hp^p\) for every \(1\le p<\infty\).
\end{lemma}

\begin{proof}
We give a uniform inverse estimate for the linear equation along a segment solution.
Put \(L=|B-A|\), let \(\widehat e=(B-A)/L\) and \(n=(-\widehat e_2,\widehat e_1)\), and write
\[
 U=A+\widehat e s+nv,\qquad V=\widehat e h+nj,\qquad
 k(s,v)=a(A+\widehat e s+nv)\cdot\widehat e,\quad b(s,v)=a(A+\widehat e s+nv)\cdot n .
\]
On a fixed enlargement of the interval \(0\le s\le L\), choose \(g\) with
\[
 g(0)=0,\qquad g'(0)=1,\qquad g''(s)=-2k(s,0)g'(s).
\]
For \(0\le r\le1\), the scale in \eqref{eq:gamma} is \(\gamma_{A,B-A}(r)=g(Lr)/g(L)\). The enlargement can be chosen uniformly so that the points involved stay in a compact subset of \(D_\theta\).
The derivatives of \(g\) through order three and \(1/g'\) are uniformly bounded there.
Along the segment solution, \(v=j=0\), and \(g(s_t)\) is the conditional expectation of its terminal value by Lemma~\ref{lem:segment}.
In particular the martingale integrand \(g'(s)h\), and hence \(h\), has uniformly bounded BMO norm.

For a progressively measurable integrable drift \(f\) on \([\sigma,\tau]\), set
\[
 \|f\|_{\ast,[\sigma,\tau]}=\sup_{\sigma\le\rho\le\tau}
       \left\|\E_\rho\int_\rho^\tau |f_t|\dt\right\|_\infty ,
 \qquad \E_\rho[\cdot]=\E[\cdot\mid\FF_\rho].
\]
We suppress the interval in \(\|f\|_\ast\) when it is fixed. For every \(\eta\in L^\infty(\FF_\tau;\R^2)\) and every \(f\) with \(\|f\|_\ast<\infty\), the linearized equation and its estimate are
\begin{equation}\label{eq:tube-inverse}
 \dd X=-\{Dq(\bar Z)H+f\}\dt+H\dd W,\quad X_\tau=\eta,
 \qquad
 \|X\|_{\Sp^\infty}+\|H\|_{\BMO}\le C\{\|\eta\|_\infty+\|f\|_\ast\}.
\end{equation}
We obtain existence and uniqueness in \(\Sp^\infty\times\BMO\) by solving two scalar equations in succession.
To see the structure of this equation use the coordinates
\(\psi(y)=(g(\widehat e\cdot(u(y)-A)),n\cdot(u(y)-A))\).
For general \((s,v,h,j)\) the forward drifts of the two coordinates are
\[
 \bigl(g'k+\tfrac12g''\bigr)h^2+g'bhj,\qquad
 j(kh+bj).
\]
The first coefficient \(g'(s)k(s,0)+g''(s)/2\) vanishes for every \(s\).
Write \((P,N)=D\psi(\bar Y)X\) and
\((I,J)=D\psi(\bar Y)H+D^2\psi(\bar Y)[\bar Z,X]\).
It\^o's formula, or differentiation of the two displayed drifts at \(v=j=0\), gives
\begin{align}
 \dd N &=\{k(s,0)hJ+f_N\}\dt+J\dd W,\notag\\
 \dd P &=\{g'(s)\partial_v k(s,0)h^2N+g'(s)b(s,0)hJ+f_P\}\dt+I\dd W,
 \label{eq:tube-triangular}
\end{align}
where \((f_P,f_N)=-D\psi(\bar Y)f\) and the terminal vector is \(D\psi(\bar Y_\tau)\eta\).

We first solve the equation for \((N,J)\). Put \(\lambda=k(s,0)h\) on the interval and zero elsewhere.
Under the measure \(Q_\lambda\) with density \(\dd Q_\lambda/\dd\PP=\cE(-\lambda\cdot W)_T\), the term \(\lambda J\) disappears.
Here and below the dot denotes stochastic integration when one argument is \(W\).
The conditional reverse H\"older estimate for a stochastic exponential with bounded BMO norm applies uniformly to this density \cite{Kazamaki1994}.
For any integer \(m\ge1\), conditional iteration over the time-ordered simplex gives
\[
 \E_\rho\left(\int_\rho^\tau |f_t|\dt\right)^m\le m!\,\|f\|_\ast^m .
\]
Thus conditional H\"older's inequality under the change of measure bounds \(N\) by \(C(\|\eta\|_\infty+\|f\|_\ast)\).
The conditional expectation representation and Brownian martingale representation give \(J\).
Applying It\^o's formula to \(N^2\) under \(\PP\), with localization, controls \(\|J\|_{\BMO}\) by the same quantity.
Indeed the mixed term is bounded conditionally by \(2\|N\|_\infty\|\lambda\|_{\BMO}(\E_\rho\int_\rho^\tau J^2\dt)^{1/2}\), and can be absorbed.
With \((N,J)\) determined, the equation for \((P,I)\) has known drift. Its terms involving \(N,J\) are controlled by \(\|h^2N\|_\ast\le\|h\|_{\BMO}^2\|N\|_\infty\) and \(\|hJ\|_\ast\le\|h\|_{\BMO}\|J\|_{\BMO}\).
Conditional expectation under \(\PP\) bounds \(P\), and It\^o's formula for \(P^2\) bounds \(\|I\|_{\BMO}\).
The scalar representations determine \((N,J)\) and then \((P,I)\) uniquely.
The inverse Jacobian of \(\psi\) is bounded along the segment.
Recovering \(H\) from \(I,J\) introduces a term bounded by \(C|\bar Z||X|\), which has the required BMO bound.
This proves \eqref{eq:tube-inverse}.

For the nonlinear equation take \(\eta=v_\theta(\chi)-v_\theta(\pi_S\chi)\).
It satisfies \(\|\eta\|_\infty\le C\varepsilon\).
Starting with \(H^{(0)}=0\), use \eqref{eq:tube-inverse} at step \(l+1\) with \(f=q(H^{(l)})\).
The quadratic identity \(q(\bar Z+H)=q(\bar Z)+Dq(\bar Z)H+q(H)\) is exact, and
\[
 \big\|q(H)-q(H')\big\|_\ast
 \le C(\|H\|_{\BMO}+\|H'\|_{\BMO})\|H-H'\|_{\BMO}.
\]
Choosing \(C_1\) sufficiently large and then decreasing \(\varepsilon_0\) makes the closed BMO ball \(\{H:\|H\|_{\BMO}\le C_1\varepsilon\}\) invariant under the iteration \(H^{(l)}\mapsto H^{(l+1)}\), with contraction factor at most \(1/2\).
For the fixed point \(H\), let \(X\) be the state in \eqref{eq:tube-inverse} with \(f=q(H)\). Then \((\bar Y+X,\bar Z+H)\) is the desired solution.

We detail continuity since the second use of the lemma will have both a moving segment and a random terminal value at a stopping time.
For a converging sequence of endpoints and terminal states the projected scalar terminal coordinates, after applying their respective \(g\)'s, converge in every finite \(L^p\) by boundedness and continuity of projection.
Doob's and the Burkholder--Davis--Gundy inequalities give convergence of the associated conditional expectations and their integrands in every finite \(\Sp^p,\Hp^p\).
Inverting the scales gives the same convergence of \((\bar Y^m,\bar Z^m)\) to \((\bar Y,\bar Z)\).
Indeed the inverse scales and their derivatives vary continuously on their ranges, and the integrands have moments of all orders.

For the inverse along the fixed limiting segment solution, if \(\eta_m\) and \(\int_\sigma^\tau|f_m|\dt\) tend to zero in every finite \(L^p\), its solutions tend to zero in every finite \(\Sp^p\times\Hp^p\).
To estimate the solution \((N_m,J_m)\) of the \(N\) equation in \eqref{eq:tube-triangular}, transfer the data from \(\PP\) to \(Q_\lambda\) by H\"older's inequality, apply the martingale maximal and square-function inequalities, and transfer back to \(\PP\).
Both transfers allow a loss of exponent since the density has a positive moment of order greater than one and a negative moment of positive order.
The negative moment bound follows from the BMO energy moments as in Lemma~\ref{lem:density}.
In the equation for \(P_m\), the absolute time integrals of \(h^2N_m\) and \(hJ_m\) then tend to zero in every finite \(L^p\), by H\"older's inequality and the energy moments of \(h\).
The same argument applies when recovering \(H_m\) in the original coordinates.

Compare now the Picard iterates along the moving segments with the corresponding iterate along the fixed segment.
Writing their difference with the fixed linear operator adds the source
\((Dq(\bar Z^m)-Dq(\bar Z))H^m\), where \(H^m\) is the output of the current linear step.
Its absolute time integral tends to zero in every finite \(L^p\) by Cauchy--Schwarz in time, the convergence of \(\bar Z^m\) in every \(\Hp^p\), and the uniform BMO bounds for \(H^m\).
The difference of the preceding quadratic sources tends to zero by the same product estimate and induction on the step.
The fixed inverse therefore gives convergence at each step, and uniform contraction controls the tails in all the indicated norms.
All representations above are conditional expectations up to \(\tau\) restricted to \([\sigma,\tau]\).
In particular the convergence of the states holds at every stopping time in this interval, including the random initial interface \(\sigma\).
\end{proof}

\subsection{Attainment of relative interior points by switching}

For distinct \(A,B\) and \(0\le p\le1\), let \(M_p(A,B)\) denote the common transformed initial state of segment solutions with \(U_\tau\in\{A,B\}\) almost surely and \(\PP(U_\tau=B)=p\). Set \(M_p(A,A)=A\).
For \(A\ne B\), with \(L,\widehat e,g\) as above,
\begin{equation}\label{eq:segment-mean}
 M_p(A,B)=A+\widehat e\,g^{-1}\big((1-p)g(0)+p\,g(L)\big).
\end{equation}
For distinct endpoints the map \(p\mapsto M_p(A,B)\) is a homeomorphism onto \([A,B]\) and depends continuously on the ordered pair \((A,B)\).
For a segment solution with \(U_\tau\in\{A,B\}\), the state at a stopping time \(\sigma\le\tau\) is \(U_\sigma=M_p(A,B)\), where \(p=\PP(U_\tau=B\mid\FF_\sigma)\).

We shall use the following source of random choices in the original filtration.
If \(0\le\sigma<\vartheta\le T\) almost surely, there is an \(\FF_\vartheta\)-measurable variable uniform on \([0,1]\) conditionally on \(\FF_\sigma\).
Indeed, given the past stopped at \(\sigma\), consider the pair consisting of \(\vartheta-\sigma\) and the path of Brownian increments from \(\sigma\) stopped at that duration.
Its conditional law on the standard Borel space of stopped paths has no atom.
For a specified stopped path of positive duration, choose a rational time strictly between zero and that duration.
Equality with the specified path forces the Brownian increment at that rational time to have a specified value, an event of conditional probability zero by the strong Markov property.
This argument holds simultaneously for all candidate atoms on the full set of pasts where the future has Wiener law and \(\vartheta-\sigma>0\) almost surely.
A Borel injection of stopped paths into \([0,1]\) preserves atomlessness.
Evaluating the conditional distribution function of the resulting code at the code itself gives the asserted uniform variable.

\begin{proof}[Proof of Theorem~\ref{thm:attainable-interior}]
Take distinct \(A,B\in I(H)\), and denote the transformed trajectories of two realizing solutions by \(U^A,U^B\).
For \(\delta>0\), stop at \(\vartheta\), the earlier of \(\tau\) and the first time either trajectory is at distance \(\delta\) from its initial point.
Continuity gives \(\vartheta>0\) almost surely.
Choose a uniform variable \(J\) in \(\FF_\vartheta\).
For \(0\le p\le1\) prescribe at \(\vartheta\) the state
\[
 \chi_p=\ind_{\{J\le p\}}U^B_\vartheta+\ind_{\{J>p\}}U^A_\vartheta .
\]
It is within \(\delta\) of the corresponding point \(B\) or \(A\), and is continuous in \(L^2\) as a function of \(p\).
Lemma~\ref{lem:segment-tube} supplies an initial solution up to \(\vartheta\) whose transformed initial state \(F_\delta(p)\) depends continuously on \(p\).
On \((\vartheta,\tau]\) follow the selected trajectory with its own integrand.
The pieces agree at \(\vartheta\), and the resulting terminal state belongs to \(H(\omega)\) even on \(\{\vartheta=\tau\}\).
Thus \(F_\delta(p)\in I(H)\).
Projection onto \([A,B]\) is nonexpansive, and both \(g\) and \(g^{-1}\) are Lipschitz on the segment.
The conditional expectation formula for the projected solution and \eqref{eq:tube-estimate} therefore give
\begin{equation}\label{eq:binary-switch-error}
 \sup_{p\in[0,1]}|F_\delta(p)-M_p(A,B)|\le C\delta .
\end{equation}
Here and below \(\delta\) is small enough to apply the tube estimate, and \(u\) is Lipschitz on the common bounded state range.
Letting \(\delta\downarrow0\) gives \([A,B]\subset\overline{I(H)}\).
Approximating endpoints by elements of \(I(H)\) yields convexity of its closure.
If the affine dimension of this closure is zero, the second assertion is immediate.
In dimension one, choose attainable \(A,B\) on opposite sides of an arbitrary relative interior point.
The continuous curve \(F_\delta\) lies in their affine line, and by \eqref{eq:binary-switch-error} its endpoints still straddle the point for small \(\delta\).
The intermediate value theorem gives exact attainment.

In dimension two, let \(A,B,C\in I(H)\) be noncollinear and use all three trajectories in defining \(\vartheta\).
We insert a stopping time strictly between zero and \(\vartheta\).
Put \(m=\E\vartheta>0\), take the continuous version of \(M_t=\E_t\vartheta\), and let \(\sigma\) be the first time \(M_t-t\le m/2\).
Since \(M_0=m\) and \(M_\vartheta=\vartheta\), continuity gives \(0<\sigma<\vartheta\) almost surely.
Choose a uniform \(J_1\in\FF_\sigma\) and a \(J_2\in\FF_\vartheta\) uniform conditionally on \(\FF_\sigma\).
Use the same \(J_1,J_2\) for all switching probabilities \((p,q)\in[0,1]^2\).

On \([\sigma,\vartheta]\) solve with terminal state \(U^B_\vartheta\) on \(\{J_2\le p\}\) and \(U^A_\vartheta\) otherwise, using the tube around \([A,B]\).
Denote its transformed state at \(\sigma\) by \(D_p\).
The conditional version of the argument giving \eqref{eq:binary-switch-error} yields
\(\|D_p-M_p(A,B)\|_\infty\le C\delta\), uniformly in \(p\).
Lemma~\ref{lem:segment-tube} also gives continuity of \(p\mapsto D_p\) in \(L^2\).
On \([0,\sigma]\) prescribe the terminal state
\[
 \chi_{p,q}=\ind_{\{J_1\le q\}}U^C_\sigma+\ind_{\{J_1>q\}}D_p .
\]
This state lies within \(C'\delta\) of the corresponding endpoint of \([M_p(A,B),C]\).
These segments form a compact family of positive minimum length by noncollinearity.
Moreover \((p,q)\mapsto\chi_{p,q}\) is jointly continuous in \(L^2\), since \(D_p\) is continuous in \(L^2\), all states are uniformly bounded, and \(J_1\) has no atoms.
The moving-segment assertion of Lemma~\ref{lem:segment-tube} gives a continuous transformed initial state \(F_\delta(p,q)\).
On \(\{J_1\le q\}\) continue along the \(C\) solution after \(\sigma\).
On the complement, use the solution just constructed up to \(\vartheta\) and then continue along \(A\) or \(B\) as selected by \(J_2\).
At each \((p,q)\), the terminal state is one of \(U^A_\tau,U^B_\tau,U^C_\tau\) on each path, so \(F_\delta(p,q)\in I(H)\).
On \([0,\sigma]\), \eqref{eq:tube-estimate} bounds the initial error from projecting \(\chi_{p,q}\) onto \([M_p(A,B),C]\). In the segment formula \eqref{eq:segmentformula}, replacing this projection by \(M_p(A,B)\ind_{\{J_1>q\}}+C\ind_{\{J_1\le q\}}\) changes the initial state by at most a constant times \(\delta\), by the uniform Lipschitz bounds on the segment scales and their inverses. Thus
\[
 \sup_{(p,q)\in[0,1]^2}
 \big|F_\delta(p,q)-M_q(M_p(A,B),C)\big|\le C''\delta .
\]

The ideal map \(F(p,q)=M_q(M_p(A,B),C)\) takes the boundary of the square into the boundary of the triangle \(ABC\).
Its restriction to the open square is a homeomorphism onto the open triangle.
Indeed the ray from \(C\) through an interior point determines its unique intersection with \([A,B]\), hence \(p\), and the segment scale then determines \(q\) continuously.
Thus for any point \(z\) in the open triangle the Brouwer degree of \(F\) at \(z\) on the open square is \(1\) or \(-1\).
For sufficiently small \(\delta\) the straight-line homotopy from \(F\) to \(F_\delta\) keeps the image of the boundary away from \(z\).
Invariance of degree gives \(z\in F_\delta((0,1)^2)\subset I(H)\).
Finally, every interior point of \(\overline{I(H)}\) lies strictly inside a triangle with vertices in \(I(H)\).
To obtain it, first choose a small triangle around the point inside the closure, then approximate its vertices by attainable points.
\end{proof}

\subsection{Terminal perturbations with a random tolerance}

For the existence proof we also need to move a terminal state into the relative interior of a random convex set while keeping its initial value close to the original one.
Given the reference solution and a state tolerance \(d>0\), we use the accumulated energy of \(Z\) to choose a positive random radius \(r\) that applies to every perturbation \(h\) of \(Y_\tau\) with \(|h|\le r\).

\Needspace{7\baselineskip}
\begin{lemma}[Terminal perturbations with a random tolerance]\label{lem:fine-perturbation}
Let \((Y,Z)\) be a bounded solution on \([\sigma,\tau]\) and let \(d>0\).
There exists a strictly positive \(\FF_\tau\)-measurable random variable \(r\) such that every \(\FF_\tau\)-measurable vector \(h\) with \(|h|\le r\) is realized by a bounded solution \((Y',Z')\) on this same interval satisfying
\[
 Y'_\tau=Y_\tau+h,\qquad \|Y'-Y\|_{\Sp^\infty[\sigma,\tau]}\le d .
\]
\end{lemma}

\begin{proof}
The correction \((X,H)=(Y'-Y,Z'-Z)\) solves
\begin{equation}\label{eq:fine-difference}
 \dd X=-\big(Dq(Z)H+q(H)\big)\dt+H\dd W .
\end{equation}
We first give a local estimate uniform over stochastic intervals \([s,t]\) on which \(\int_s^t |Z_r|^2\dd r\le a\) almost surely.
For a prescribed drift \(f\) and terminal value \(v\), ordinary conditional expectation gives
\(\|X\|_{\Sp^\infty}\le\|v\|_\infty+\|f\|_{\ast,[s,t]}\) for \(\dd X=-f\dt+H\dd W\).
Martingale representation and It\^o's formula give
\(\|H\|_{\BMO}^2\le\|v\|_\infty^2+2\|X\|_{\Sp^\infty}\|f\|_{\ast,[s,t]}\).
For \(f(G)=Dq(Z)G+q(G)\), conditional Cauchy--Schwarz yields
\[
 \|f(G)-f(G')\|_{\ast,[s,t]}\le
 C\big(\sqrt a+\|G\|_{\BMO}+\|G'\|_{\BMO}\big)\|G-G'\|_{\BMO}.
\]
First choose \(a>0\) and a ball in \(\BMO\) small enough that the iteration \(G\mapsto H\) is a contraction there with factor at most \(1/2\).
Then choose \(u_0>0\) so that \(\|v\|_\infty\le u_0\) makes the ball invariant.
The fixed points satisfy, for a deterministic \(L\ge2\),
\begin{equation}\label{eq:fine-slice}
 \|X(v)-X(w)\|_{\Sp^\infty}+\|H(v)-H(w)\|_{\BMO}
 \le L\|v-w\|_\infty,\qquad \|v\|_\infty,\|w\|_\infty\le u_0 .
\end{equation}
This follows by subtracting the two fixed point equations and absorbing the contraction term.
The solution for \(v=0\) is zero.
All constants are independent of the slice and its endpoints.

Set \(t_0=\sigma\) and, successively, stop at
\[
 t_{k+1}=\tau\wedge\inf\bigg\{t\ge t_k:\int_{t_k}^{t}|Z_s|^2\ds\ge a\bigg\}.
\]
The energy is finite almost surely, so \(N=\min\{k:t_k=\tau\}\) is finite almost surely.
Write \(A_k=\{N>k\}=\{t_k<\tau\}\in\FF_{t_k}\).
Choose \(K>2L\), \(0<b\le\min(u_0,d)\), and set \(b_k=bK^{-k}\) and \(r=b_N\).
On \(\{\sigma=\tau\}\) set the sole state correction equal to \(h\).
We give the construction on the active paths, taking each local correction zero on paths inactive at its left endpoint.

Fix a truncation depth \(m\).
Set the active entry correction \(x_m^m=0\).
For \(k=m-1,\ldots,0\) solve \eqref{eq:fine-difference} on \([t_k,t_{k+1}]\) by \eqref{eq:fine-slice}, with terminal correction
\begin{equation}\label{eq:fine-entry}
 v_k^m=h\ind_{\{N=k+1\}}+x_{k+1}^m\ind_{A_{k+1}},
 \qquad x_k^m=\ind_{A_k}X_k^m(t_k).
\end{equation}
Here \((X_k^m,H_k^m)\) denotes the resulting slice solution.
Both terms of \(v_k^m\) are \(\FF_{t_{k+1}}\)-measurable, since \(t_{k+1}=\tau\) on the first event.
On \(A_k^c\) the interval has length zero and the displayed terminal correction is zero.
On \(\{N=k+1\}\) the slice pays \(h\), whose magnitude is at most \(b_{k+1}\) there.
Backward induction gives \(\|v_k^m\|_\infty\le b_{k+1}\) and
\(\|X_k^m\|_{\Sp^\infty}+\|H_k^m\|_{\BMO}\le Lb_{k+1}\le b_k\).

If \(m'>m\), the active entry corrections at \(t_m\) differ by at most \(b_m\).
At each earlier right endpoint only the continuation term of \eqref{eq:fine-entry} transmits a difference.
Consequently the distance between the two pairs on slice \(k<m\) in \(\Sp^\infty\times\BMO\) is at most \(L^{m-k}b_m\), by \eqref{eq:fine-slice}.
Since \(K>L\), the pairs on each fixed slice converge in those norms.
The quadratic product bound passes \eqref{eq:fine-difference} to the limit there, and \eqref{eq:fine-entry} gives agreement at each continuing endpoint.
On each path with \(N\ge1\), concatenate the \(N\) active slices, retaining \(h\) at the final endpoint of slice \(N-1\).
The resulting state is adapted and continuous, and its supremum norm is at most \(b\).
Each slice integrand extended by zero outside \((t_k,t_{k+1}]\) has global BMO norm at most \(b_k\), by conditioning at the start of the slice when necessary.
The supports are disjoint, so the squared BMO norm of their sum is bounded by \(\sum_{k\ge0}b_k^2\).
Summing the equations up to the almost surely finite last slice establishes \eqref{eq:fine-difference} on the whole interval.
Adding the original solution proves the lemma.
\end{proof}

\section{Universal existence}\label{sec:existence}

We prove the existence assertion of Theorem~\ref{thm:universal}. Fix $\theta\in\R^2$. For every bounded terminal condition $\xi$ we seek a solution of \eqref{eq:bsde} in $\Sp^\infty\times\BMO$ on the given Brownian filtration.
Write $u=u_\theta$, $v=v_\theta$ and $U=u(Y)$. If existence fails, choose a rectangle $B\subset D_\theta$ containing $u(\xi)$ for some terminal condition $\xi$ without a solution. We choose $B$ so that the sum of the side lengths of $v(B)$ is nearly minimal among all such rectangles.
We bisect $B$ once and at each stage select the more probable of these two halves conditional on the current history. A return occurs if its conditional probability falls to $1/4$ before the terminal time. Given the history at that return, Lemma~\ref{lem:density} bounds the probability of the half selected at the preceding stage away from zero under the Girsanov measures of the approximating solutions.
The means of the transformed terminal state under these measures, and hence their limits, stay a fixed distance from the opposite face of $B$. By the choice of $B$, the resulting rectangle is narrow enough to give solvability on each stopped interval.
On return, the transformed initial state of the next solution must equal the transformed terminal state of the interval just completed.

For a stopped history $h$ at stage $n$, write $\xi_h$ for $\xi$ as a function of the rescaled future Brownian path with past $h$. Let $K_n(h)$ consist of the limits of $u(Y^j_0)$ for solutions on the entire remaining interval whose terminal conditions converge to $\xi_h$ in $L^2$ and lie in vanishing enlargements of $v(B)$.
On the current stopped interval, write $I_n(H)(h)$ for the initial values $U_0$ attainable with endpoint constraint $H$ on return and the prescribed terminal state on nonreturn.
Starting from $K=(K_n)$, we apply the same operation simultaneously at all stages. It replaces the set at stage $n$ by the closed convex hull of the initial values attainable with the current constraint at stage $n+1$. Iterating this restriction gives compact convex sets $C_n(h)\subset K_n(h)$ for which we prove
\[
  \operatorname{ri}C_n(h)\ \subset\ I_n(\operatorname{ri}C_{n+1})(h).
\]
Lemma~\ref{lem:fine-perturbation} moves a return endpoint into the next relative interior with arbitrarily small change in the initial value.
Theorem~\ref{thm:attainable-interior} then gives exact attainment throughout the current relative interior.
This inclusion holds for every point of $\operatorname{ri} C_n(h)$ outside a single null set of histories at stage $n$.
It therefore permits measurable selection at the particular value delivered by the preceding interval.

\subsection{Stopping times for a rectangle of nearly minimal width}\label{subsec:ex-stopping}

Let $\mathcal W=C_0([0,1];\R)$ carry Wiener measure $\mathbb W$ and its augmented Brownian filtration.
It suffices to prove existence on this space.
Indeed, a bounded terminal condition on an augmented natural Brownian filtration has a version that is a Borel function of the driving path.
For a fixed past up to time $s<T$, put $\ell=T-s$.
The change of variables
\[
 W^{(s)}_r=\frac{W_{s+\ell r}-W_s}{\sqrt{\ell}},\qquad
 \widetilde Y_r=Y_{s+\ell r},\qquad
 \widetilde Z_r=\sqrt{\ell}\,Z_{s+\ell r},\qquad 0\le r\le1,
\]
preserves the equation by quadratic homogeneity.
After a stopping time $s$, the conditional law of $W^{(s)}$ given the stopped past is $\mathbb W$.
Lemma~\ref{lem:history-selection} realizes the selected integrands on the given Brownian filtration.
In the construction we take $T=1$.

For a compact rectangle $B=\prod_{j=1}^2[a_j,b_j]\subset D_\theta$, define
\[
 \ell_j=v_j(b_j)-v_j(a_j),\qquad w(B)=\ell_1+\ell_2 .
\]
Let $\mathscr B$ consist of those $B$ for which some bounded $\xi$ on $\mathcal W$ satisfies $u(\xi)\in B$ almost surely and $\Sol_\theta(\xi)=\varnothing$.
Under the contradiction hypothesis $\mathscr B\ne\varnothing$, so $w_*=\inf_{B\in\mathscr B}w(B)$ is finite.
Since $q_\theta$ is independent of $Y$, translating $Y$ by a constant vector preserves solvability for the correspondingly translated terminal condition.
The identity
\[
 \phi_\alpha(y+t)=\phi_\alpha(t)+e^{-\alpha t}\phi_\alpha(y)
\]
shows that this translation acts by a coordinatewise affine transformation on $u(Y)$ and preserves $w(B)$.
We may therefore translate $v(B)$ to $[0,\ell_1]\times[0,\ell_2]$. Any terminal condition in this rectangle has sup norm at most $w(B)$, so the small-data existence theorem~\cite[Proposition~1]{Tevzadze2008} gives $w_*>0$.

Set $L_*=w_*+1$ and consider normalized rectangles $B\in\mathscr B$ with $w(B)\le L_*$.
Then $v(B)\subset[0,L_*]^2$. For every solution whose terminal condition lies in $[-1,L_*+1]^2$, Proposition~\ref{prop:bounds} and Lemma~\ref{lem:density} give, for its density $D$ in \eqref{eq:Q},
\begin{equation}\label{eq:ex-density}
 \E_{\mathbb W}D^{1+p}\le A,\qquad \E_{\mathbb W}D^{-b}\le A
\end{equation}
for some $p,b>0$ and $A\ge1$ independent of the rectangle and the solution.
Given the stopped past, the conditional law of the rescaled future is almost surely $\mathbb W$, so the same estimates apply.
Fix $\varepsilon=1/4$ and set $\delta=(\varepsilon^{1+b}/A)^{1/b}$.
If $\mathbb W(E)=\varepsilon$, then H\"older's inequality gives
\begin{equation}\label{eq:ex-event}
 \varepsilon
 \le \big(\E_{\mathbb W}[D\ind_E]\big)^{b/(1+b)}
      \big(\E_{\mathbb W}D^{-b}\big)^{1/(1+b)},
 \qquad \E_{\mathbb W}[D\ind_E]\ge\delta .
\end{equation}

Choose a coordinate $i=i(B)$ with $\ell_i=\max_j\ell_j$, so $\ell_i\ge w_*/2$.
We shorten $[a_i,b_i]$ at either end by $\delta(b_i-a_i)/2$. Since $b_i-a_i=\int_0^{\ell_i}\phi_{\alpha_i}'(y)\,\mathrm dy$, the corresponding decrease in $w(B)$ is at least
\[
 \frac{\delta\ell_i}{2}\,
 \frac{\min_{[0,\ell_i]}\phi_{\alpha_i}'}{\max_{[0,\ell_i]}\phi_{\alpha_i}'}
 =\frac{\delta\ell_i}{2}e^{-|\alpha_i|\ell_i}
 \ \ge\ \frac{\delta w_*}{4}e^{-L_*\max_j|\alpha_j|}=:\gamma>0 .
\]
Choose a normalized $B\in\mathscr B$ with $w(B)<w_*+\min(1,\gamma)$ and a bounded $\xi$ witnessing $B\in\mathscr B$.
For this $B$ take $i=i(B)$ as above. Write $\eta=u(\xi)$ and $m=(a_i+b_i)/2$.
The two half rectangles are specified by the events $G=\{\eta_i\le m\}$ and $G^c$.
Define $B^-$ by replacing $b_i$ in $B$ by $b_i-\delta(b_i-a_i)/2$, and $B^+$ by replacing $a_i$ by $a_i+\delta(b_i-a_i)/2$.
The lower half of $B$ is contained in $B^-$ and the upper half in $B^+$. Both $B^-$ and $B^+$ have width strictly less than $w_*$.

Starting at $\sigma_0=0$, proceed as follows whenever $\sigma_n<1$.
At $\sigma_n$ choose the more probable of $G,G^c$ conditionally on $\FF_{\sigma_n}$, breaking ties in favor of $G$.
Denote the chosen event by $E_n$, with its label chosen $\FF_{\sigma_n}$-measurably.
On $[\sigma_n,1]$ let $p^{(n)}_t=\mathbb W(E_n\mid\FF_t)$, taking its continuous version.
Let $\sigma_{n+1}$ be its first hitting time of $\varepsilon$ after $\sigma_n$, with value $1$ if it does not hit before $1$.
Set all subsequent times equal to $1$ once time $1$ is reached.
Write $A_n=\{\sigma_n<1\}$ for the event that stage $n$ is reached and $\mathcal R_n=\{\sigma_{n+1}<1\}$ for its return event.
On $A_n$, the next time is strictly later than $\sigma_n$, since $p^{(n)}_{\sigma_n}\ge1/2$.
On $\mathcal R_n$, $p^{(n)}_{\sigma_{n+1}}=\varepsilon$. The other half then has conditional probability $1-\varepsilon=3/4$ and is selected at $\sigma_{n+1}$.
On $A_n\setminus\mathcal R_n$, $E_n$ occurs almost surely, since a continuous conditional probability ending at zero would cross $\varepsilon$ before time $1$.
Optional sampling thus gives, on $A_n$,
\begin{equation}\label{eq:ex-returns}
 \mathbb W(\mathcal R_n\mid\FF_{\sigma_n})
 =\frac{1-p^{(n)}_{\sigma_n}}{1-\varepsilon}\le\frac23 .
\end{equation}
Consequently $\mathbb W(A_n)\le(2/3)^n$, and almost surely one of the intervals ends at time $1$.

\subsection{Solvability on one interval for prescribed terminal conditions}

Choose a bounded Borel version of $\xi$ as a function of the canonical path, taking values in $v(B)$ everywhere.
A history at the start of stage $n$ is $h=(s,w_{\cdot\wedge s})$ with $s=\sigma_n<1$.
For $\omega\in\mathcal W$, let $C_h\omega$ agree with $w$ up to $s$ and satisfy $(C_h\omega)(s+(1-s)r)=w(s)+\sqrt{1-s}\,\omega(r)$ for $0\le r\le1$.
Set $\xi_h=\xi\circ C_h$ and $\eta_h=u(\xi_h)$.
For a stopped history $h$, define the rescaled next time and the return event by
\[
 \tau_h(\omega)=\frac{(\sigma_{n+1}(C_h\omega)-s)^+}{1-s},\qquad
 \mathcal R_h=\{\tau_h<1\}.
\]
For fixed $h$, we work on the unit Wiener space of future increments and write $\tau=\tau_h$ and $\mathcal R=\mathcal R_h$. On $\mathcal R$ let $h'$ be the history of $C_h\omega$ stopped at $s+(1-s)\tau$, and give $h'$ a fixed value on $\mathcal R^c$.
At almost every active history, $\sigma_{n+1}(C_h\omega)>s$ almost surely, so $\tau$ is the rescaled stopping time and lies in $(0,1]$.
The subprobability law of active stage-$n$ histories is denoted by $\mu_n$.
If $P_n(h,\cdot)$ is the subprobability kernel giving the next history on return, the strong Markov property gives
\begin{equation}\label{eq:ex-historylaws}
 \mu_{n+1}=\mu_n P_n .
\end{equation}
In particular, for each fixed Borel set $N$ with $\mu_{n+1}(N)=0$, one has $P_n(h,N)=0$ for $\mu_n$-almost every $h$.
These laws are determined by $\xi$ and the Brownian path, so the same $\mu_n$ and $P_n$ apply to every choice of solution.

Write $B_Y=v(B)$, and let $B_Y^{\,e}$ be the rectangle formed by enlarging each coordinate interval of $B_Y$ by $e$ at both ends.
For an active history $h$, define $K_n(h)$ to be the set of all limits $x=\lim_j u(Y^j_0)$ along bounded solutions on the remaining interval such that
\begin{equation}\label{eq:ex-approxdata}
 Y^j_1\in B_Y^{\,e_j}\quad\mathbb W\text{-a.s.},\qquad
 \|Y^j_1-\xi_h\|_{L^2(\mathbb W)}\longrightarrow0,\qquad e_j\downarrow0 .
\end{equation}
Here $Y^j_1$ is the terminal state at time $1$ on the rescaled Wiener space. Each approximating solution therefore runs through any subsequent returns to the original terminal time.

\begin{lemma}[Range of the continuation values]\label{lem:ex-roots}
The sets $K_n(h)$ are nonempty compact convex subsets of $B$.
At an active history $h$ of stage $n$, let $B_h=B^-$ if $G$ is selected at $h$, and $B_h=B^+$ if $G^c$ is selected there.
For $\mu_n$-almost every $h$, one has $K_{n+1}(h')\subset B_h$ for $P_n(h,\cdot)$-almost every return history $h'$, and $\eta_h\in B_h$ $\mathbb W$-almost surely on $\mathcal R_h^c$.
\end{lemma}
\begin{proof}
Apply Lemma~\ref{lem:tree-density} on the remaining Wiener interval with $K=B$ and target $\eta_h$. Each approximating terminal state belongs to $B$. In original coordinates its value belongs to $B_Y$ and converges in $L^2$ to $\xi_h$ because $v$ is Lipschitz on $B$.
By \eqref{eq:conditionalU} their transformed initial values $u(Y_0^j)$ lie in $B$, hence have a convergent subsequence.
The same identity places every limit in \eqref{eq:ex-approxdata} in $B$.
A diagonal choice of approximating solutions shows that $K_n(h)$ is closed.

Take $x,x'\in K_n(h)$ with respective approximating states $Y^j,\widetilde Y^j$ and transformed terminal values $\eta^j,\widetilde\eta^j$.
At level $j$, take the random set $\{\eta^j,\widetilde\eta^j\}$ as the constraint on the transformed terminal state.
Both approximating solutions satisfy this constraint.
By Theorem~\ref{thm:attainable-interior} the closure of the attainable initial values for this constraint is convex.
For each $\lambda\in[0,1]$ we can thus choose a solution with state $Y^{j,\lambda}$ satisfying
\[
 \big|u(Y^{j,\lambda}_0)-(1-\lambda)u(Y^j_0)-\lambda u(\widetilde Y^j_0)\big|<1/j .
\]
Since $Y^{j,\lambda}_1\in\{Y^j_1,\widetilde Y^j_1\}$ almost surely, its squared error from $\xi_h$ is bounded by the sum of the two squared terminal errors, and its terminal range is in the larger of the two enlarged rectangles.
Taking limits proves $(1-\lambda)x+\lambda x'\in K_n(h)$.

On almost every return history, the conditional probability of the event selected at its parent is $\varepsilon$.
Freeze one such history $h'$.
For any $x\in K_{n+1}(h')$ take approximating solutions as in \eqref{eq:ex-approxdata}, with Girsanov densities $D^{(j)}_1$.
Uniform boundedness in the enlarged rectangle and convergence in $L^2$ imply convergence of $u(Y^j_1)$ to $\eta_{h'}$ in every finite $L^q$.
By \eqref{eq:ex-density} and H\"older's inequality,
\[
 x=\lim_j\E_{\mathbb W}[D^{(j)}_1 u(Y^j_1)]
   =\lim_j\E_{\mathbb W}[D^{(j)}_1\eta_{h'}].
\]
Write $G_{h'}=\{\eta_{h',i}\le m\}$ for the event on the future path space corresponding to $G$.
If $G$ was selected at $h$, then $\mathbb W(G_{h'})=\varepsilon$. Equation~\eqref{eq:ex-event} gives $\E_{\mathbb W}[D^{(j)}_1\ind_{G_{h'}}]\ge\delta$. Since $\eta_{h',i}\le m$ on $G_{h'}$ and $\eta_{h',i}\le b_i$ everywhere, we obtain $x_i\le b_i-\delta(b_i-a_i)/2$.
If $G^c$ was selected at $h$, applying \eqref{eq:ex-event} to $G_{h'}^c$ gives $x_i\ge a_i+\delta(b_i-a_i)/2$.
Thus $K_{n+1}(h')\subset B_h$. On nonreturn, $\eta_h=u(\xi_h)$ lies in the selected half of $B$, which is contained in $B_h$.
\end{proof}

At any active history satisfying Lemma~\ref{lem:ex-roots}, prescribe any $\FF_\tau$-measurable transformed terminal state $X\in K_{n+1}(h')$ on return and set $X=\eta_h$ on nonreturn.
The resulting variable is measurable at $\tau$, since nonreturn means $\tau=1$.
By Lemma~\ref{lem:ex-roots}, $X\in B_h$ almost surely and $w(B_h)<w_*$. The definition of $w_*$ therefore gives a solution on $[0,1]$ with terminal condition $Y_1=v(X)$.
Since $X$ is $\FF_\tau$-measurable, \eqref{eq:conditionalU} gives $U_t=X$ for $\tau\le t\le1$. Hence $Y_\tau=v(X)$.
Thus each such choice of $X$ gives a solution on the stopped interval. For this solution to join the preceding one, its transformed initial value $u(Y_0)$ must equal the transformed state reached on the preceding interval.

\subsection{Measurable selection of solutions at stopping times}

Suppose that for almost every stopped history $h$ there is a solution with the prescribed initial value $x(h)$ and the required endpoint constraint. We need to choose these solutions measurably in $h$.
We encode a solution on Wiener space by its initial value $y$ in the original coordinates and its predictable integrand $z$. The conditions on the initial and terminal states define a Borel relation in $(h,x,y,z)$, where $x$ is the prescribed initial value in transformed coordinates.
Take as the history space the Borel set of pairs $(s,w)\in[0,1)\times\mathcal W$ satisfying $w=w_{\cdot\wedge s}$.
Each $\mu_n$ is carried by this space, and stopping, time rescaling and concatenation of paths are Borel operations.
The conditional probability $\mathbb W(G\mid\FF_t)$ has a Borel path representative.
For instance, at rational times one may integrate the Borel indicator of $G$ against the Wiener law of an appended future and then use the continuous extension on the full-measure set where it exists.
Giving it an arbitrary continuous path elsewhere yields a Borel map into $C([0,1])$.
The hitting times, return indicators and next histories above thus have Borel representatives and satisfy the stopping-time and adaptedness assertions in the augmented filtration.
To identify the frozen stopping time, write $p_t=\mathbb W(G\mid\FF_t)$ and $G_h=C_h^{-1}(G)$. For each rational $r\in[0,1]$, disintegration at $\sigma_n$ and the independent Brownian increments after $\sigma_n+(1-\sigma_n)r$ give, for $\mu_n$-almost every $h=(s,w)$,
\[
 p_{s+(1-s)r}(C_h\omega)=\mathbb W(G_h\mid\FF_r)(\omega)
 \qquad\text{for $\mathbb W$-almost every $\omega$.}
\]
Intersecting over rational $r$ and taking continuous versions identifies the two processes on $[0,1]$. Thus $\tau_h$ is the first hitting time of $\varepsilon$ by the frozen conditional probability of the half selected at $h$, with value $1$ if there is no earlier hit. On return, the strong Markov property at $\tau_h$ gives Wiener law for the rescaled increments after the return history.
Let $\mathscr H_n$ be a Borel set of full $\mu_n$-measure on which these assertions hold and the next history on return has subprobability law $P_n(h,\cdot)$. By \eqref{eq:ex-historylaws} we may take $P_n(h,\mathscr H_{n+1}^c)=0$ for every $h\in\mathscr H_n$, discarding the parents from which an exceptional set is reached with positive probability after finitely many returns.

Fix the separable Hilbert space $\mathcal H=\Hp^2(\mathbb W;\R^2)$ of predictable integrands on the unit interval.
A deterministic $y\in\R^2$ and $z\in\mathcal H$ determine the continuous adapted process
\begin{equation}\label{eq:ex-forward}
 \mathcal Y(y,z)_t
 =y-\int_0^tq_\theta(z_r)\,\mathrm dr+\int_0^t z_r\,\mathrm d\widehat W_r ,
\end{equation}
where $\widehat W$ denotes the canonical coordinate process.
This map into $L^1(\mathbb W;C([0,1];\R^2))$ is continuous.
The drift difference is bounded in $L^1$ of the uniform norm by
$c\|z-z'\|_{\mathcal H}(\|z\|_{\mathcal H}+\|z'\|_{\mathcal H})$, and the martingale difference is bounded by $c\|z-z'\|_{\mathcal H}$.
Both spaces of equivalence classes admit jointly measurable representatives.
Indeed, select the first member of a fixed countable dense family within $2^{-2j}$ of the input and take the pointwise limit of these approximants, using a default value where the limit fails.
For $\mathcal H$ one may take simple predictable processes from the raw canonical filtration, and for states one takes the limit in path space.
These representatives give Borel tests in $(h,x,y,z)$ for the equality $u(y)=x$, the rectangle constraint on the state, and the vanishing of $z$ after $\tau$.
The endpoint tests require $\mathcal Y(y,z)_\tau=\xi_h$ on $\mathcal R^c$ and membership of $u(\mathcal Y(y,z)_\tau)$ in the prescribed set at $h'$ on $\mathcal R$.
For the latter test we integrate the indicator of the complement of the graph over $\mathcal R$ and require the integral to vanish.
This applies also to the Borel graph of a relative interior.
Almost sure bounds and terminal $L^2$ error bounds for solutions on the remaining interval are Borel by the same argument.
Every pair $(y,z)$ with bounded state satisfying these conditions yields a bounded solution with integrand in $\BMO$ by Proposition~\ref{prop:bounds}.

The condition $x\in K_n(h)$ can be written with a sequence of solutions on the remaining interval whose initial-value error, terminal $L^2$ error and rectangle enlargement are each at most $1/j$ at level $j$.
The set of pairs $(h,x)$ satisfying it is analytic by projection from a countable product of witness spaces.
For every open set $O$, the event $K_n(h)\cap O\ne\varnothing$ is therefore universally measurable.
The countably many such tests on a basis determine the compact set, so $K_n$ has a Borel compact-valued version under $\mu_n$ (equivalently, view it as a random variable in the compact-set space with its Hausdorff metric).
Agreement here is equality of the entire fiber off one null set.
For exact one-interval initial values with a Borel endpoint constraint, the same projection argument gives measurable open-set tests for their closure.
Its closed convex hull is measurable as well, by taking convex combinations of a countable measurable dense family.
We use Borel versions under $\mu_n$ for these compact-valued maps.
Changing a version at the next level on a $\mu_{n+1}$-null set preserves the preceding constraint almost surely for $\mu_n$-almost every parent, by \eqref{eq:ex-historylaws}.

\begin{lemma}[Selection at a prescribed initial value]\label{lem:history-selection}
At stage $n$, take $\mu=\mu_n$ or, for $n\ge1$, fix $h_0\in\mathscr H_{n-1}$ and take $\mu=P_{n-1}(h_0,\cdot)$.
Let $h\mapsto x(h)$ be a Borel version under $\mu$ of the prescribed initial value in transformed coordinates.
Suppose that for $\mu$-almost every $h$ there is a solution $(Y,Z)$ with $u(Y_0)=x(h)$ satisfying the Borel endpoint constraints above and a common deterministic bound on $\|Y\|_{\Sp^\infty}$.
Then these solutions have a $\mu$-measurable selection. Their integrands can be realized predictably on the original Brownian filtration in the first case and on the Wiener space of future increments from $h_0$ in the second.
For each integer $j\ge1$, the same selection conclusion holds if, for $\mu$-almost every $h$, there is a bounded solution on the remaining interval satisfying
\[
 |u(Y_0)-x(h)|\le 1/j,\qquad
 \|Y_1-\xi_h\|_{L^2(\mathbb W)}\le 1/j,\qquad
 Y_1\in B_Y^{\,1/j}\quad\mathbb W\text{-a.s.}
\]
\end{lemma}
\begin{proof}
Substituting $x=x(h)$ in the Borel relation gives a Borel set of triples $(h,y,z)$ whose projection onto the history space has full $\mu$-measure.
The Jankov--von Neumann theorem~\cite[Proposition~7.49]{BertsekasShreve} gives a universally measurable selector on this projection and hence a Borel version $(y(h),z(h))$ under $\mu$.
The same argument with the error inequalities selects approximating solutions under $\mu$, including when $\mu$ is the return measure from a fixed parent.

To realize the selected solutions, fix a countable dense family of bounded simple predictable integrands in $\mathcal H$.
For each $j$, choose the first member within $2^{-j}$ of $z(h)$ in $\mathcal H$.
Its index is Borel in $h$ because $h\mapsto z(h)$ is Borel as an $\mathcal H$-valued map.
For a stage starting at $s=\sigma_n$, a deterministic grid time $r$ of such an integrand becomes the stopping time $\sigma_n+(1-\sigma_n)r$.
The choice of the simple integrand is known at $\sigma_n$, and each of its coefficients uses only the rescaled increments observed by its left endpoint.
Multiplying by $(1-\sigma_n)^{-1/2}$ and restricting to active histories therefore gives a predictable integrand in physical time.
Conditional on the stopped past, the time change is an isometry of the corresponding $\Hp^2$ norms.
Thus the pulled-back approximants converge in $\Hp^2$ whenever $\int\|z(h)\|_{\mathcal H}^2\,\mu(\mathrm dh)<\infty$.
This integrability holds for all solutions selected here by the uniform energy bound on a common rectangle.
For simple integrands the drift and stochastic integral in \eqref{eq:ex-forward} agree with their physical-time versions by the time change.
The drift bound above and the stochastic integral isometry pass the identity and the endpoint values to the limit.
The same argument applies under the return measure when attaching a tail to a fixed parent.
\end{proof}

\subsection{Exact continuation from the relative interior}

We now construct sets $C_n(h)\subset K_n(h)$ whose relative interiors supply the initial values for Lemma~\ref{lem:history-selection} at successive stopping times.
Let $C=(C_n)_{n\ge0}$ be a family of measurable nonempty compact convex sets with $C_n(h)\subset K_n(h)$ for $\mu_n$-almost every $h$.
On the $n$th interval let $I_n(H)(h)$ be the set of transformed initial values $u(Y_0)$ of bounded solutions satisfying
\begin{equation}\label{eq:ex-payments}
 U_\tau=\eta_h\ \text{on }\mathcal R^c,\qquad
 U_\tau\in H(h')\ \text{on }\mathcal R .
\end{equation}
Here $H$ is a next-stage set with Borel graph, and the constraint is imposed $\mathbb W$-almost surely in the frozen problem.
A solution on $[0,\tau]$ can be coded by keeping it constant thereafter.
Define
\[
 F(C)_n(h)=\overline{\conv}\, I_n(C_{n+1})(h).
\]
The operator is monotone for componentwise inclusion modulo the history laws, by \eqref{eq:ex-historylaws}.
For $\mu_n$-almost every $h$, $F(C)_n(h)$ is nonempty, since a measurable choice from $C_{n+1}$ on return prescribes a solvable one-interval problem.
By the preceding subsection and \eqref{eq:conditionalU}, $F(C)_n$ has a measurable version with nonempty compact convex values.

\begin{lemma}[Continuation of approximate solutions]\label{lem:ex-invariance}
For every $n$, $F(K)_n(h)\subset K_n(h)$ for $\mu_n$-almost every $h$.
\end{lemma}
\begin{proof}
Fix $h\in\mathscr H_n$ such that the chosen version of $K_{n+1}$ agrees with its defining fibers for $P_n(h,\cdot)$-almost every return history. By \eqref{eq:ex-historylaws} this holds for $\mu_n$-almost every parent. Fix a solution realizing $x\in I_n(K_{n+1})(h)$.
On return its transformed endpoint $X=U_\tau$ is measurable at $\tau$. In the augmented natural filtration the stopped Brownian path and the stopping time generate $\FF_\tau$ up to null sets. With the parent fixed, the return history determines both, so $X$ has a Borel version as a function of that history under the parent's return measure.
For each integer $j\ge1$, the definition of $K_{n+1}(h')$ gives a solution on the remaining interval from $h'$ with transformed initial value within $1/j$ of $X$, terminal $L^2(\mathbb W)$ error from $\xi_{h'}$ at most $1/j$, and terminal range in $B_Y^{\,1/j}$.
The target $X$ is specified on returns from this fixed $h$, so we use the approximation assertion of Lemma~\ref{lem:history-selection} with $\mu=P_n(h,\cdot)$ to select the tail.
If $y'$ is the chosen tail's initial value in original coordinates, translate its entire state by $v(X)-y'$.
This vector is known at the start of the tail and has norm at most $c/j$, since $v$ is Lipschitz on the common compact rectangle $u([-1,L_*+1]^2)\subset D_\theta$.
Since $q_\theta$ depends only on the integrand, the translated tail solves the equation with the same integrand and starts at $v(X)=Y_\tau$.
On nonreturn retain the given solution, whose terminal condition is $\xi_h$.
The pasted solution has transformed initial value $x$ and terminal condition in $B_Y^{\,(1+c)/j}$. Integrating the squared conditional error over the return measure bounds its terminal $L^2(\mathbb W)$ error by $(1+c)/j$.
The states agree at the stopping time and are uniformly bounded on both intervals.
Thus $x\in K_n(h)$ by \eqref{eq:ex-approxdata}, and closedness and convexity of $K_n(h)$ give the assertion.
\end{proof}

The inclusion $F(K)_n\subset K_n$ permits us to restrict the possible values at each stage to those consistent with continuation at the next.
We iterate this restriction until the same sets occur on both sides. Starting with $C^0=K$, set $C^{\nu+1}=F(C^\nu)$.
At each nonzero countable limit ordinal $\lambda$, set
\[
 C^\lambda_n(h)=\bigcap_{\nu<\lambda}C^\nu_n(h).
\]
By Lemma~\ref{lem:ex-invariance} and monotonicity this defines a decreasing family modulo each $\mu_n$.
At a countable limit, all preceding inclusions hold simultaneously off a null set.
The intersection is then nonempty by compactness and is again a measurable compact convex set.
For example, along a cofinal increasing sequence of ordinals the distances to the decreasing compact sets increase to the distance to their intersection.
At a limit step, monotonicity gives $F(C^\lambda)\subset C^\lambda$. The reverse inclusion would require the closed convex hull of the attainable initial values to commute with this decreasing intersection of endpoint constraints. We continue the iteration until a successor step gives equality.
At every stage Borel representatives may be defined arbitrarily on its exceptional sets.

This recursion stabilizes at a countable ordinal modulo all the history laws.
To prove this, fix a countable dense set $\{d_j:j\ge1\}$ in $B$ and associate to each family $C$ the number
\[
 \sum_{n\ge0}\sum_{j\ge1}2^{-n-j-1}
 \int\bigl(1\wedge\dist(d_j,C_n(h))\bigr)\,\mu_n(\mathrm dh).
\]
The number increases whenever a successor inclusion is strict on a set of positive measure at any level.
Indeed, distinct nested compact sets have different truncated distance functions somewhere on the dense set.
If every successor step below the first uncountable ordinal were strict, the intervals between successive values would contain uncountably many distinct rational numbers.
Thus there exists a countable $\nu$ at which $C=C^\nu$ satisfies
\begin{equation}\label{eq:ex-fixed}
 C_n(h)=\overline{\conv}\,I_n(C_{n+1})(h)
 \qquad\text{for $\mu_n$-almost every $h$, for every $n$.}
\end{equation}

Write $\operatorname{ri}$ for relative interior in the affine hull of the indicated fiber.
For each nonempty random compact convex set $C_n$, its relative interior has a Borel graph and admits a measurable point $c_n$.
For example, take a measurable dense sequence in the compact set, choose its first simplex of maximum affine dimension under a fixed enumeration, and take the barycenter.
This point belongs to the relative interior, including in dimension zero.
For such a point $c$ of a fiber $C$, one has
\[
 x\in\operatorname{ri}C
 \quad\Longleftrightarrow\quad
 (1+1/k)x-c/k\in C\quad\text{for some integer }k\ge1 .
\]
Indeed, the forward implication follows by extending the segment from $c$ past $x$ within the affine hull, and the reverse implication writes $x$ as a convex combination with strictly positive weight on $c$.
This also gives the claimed graph measurability.

\begin{proposition}[Exact continuation from the relative interior]\label{prop:ex-exact}
For each $n$, at $\mu_n$-almost every active history $h$,
\begin{equation}\label{eq:ex-exact}
 \operatorname{ri}C_n(h)
 \ \subset\ I_n(\operatorname{ri}C_{n+1})(h).
\end{equation}
\end{proposition}
\begin{proof}
Fix $h\in\mathscr H_n$ where \eqref{eq:ex-fixed} holds and the next-stage fibers have the stated properties for $P_n(h,\cdot)$-almost every return history.
Take any solution counted by $I_n(C_{n+1})(h)$ and write $X=U_\tau$ for its transformed endpoint on return.
Fix $d>0$.
Lemma~\ref{lem:fine-perturbation} on $[0,\tau]$ gives an $\FF_\tau$-measurable radius $r>0$ for this solution such that any terminal change of size at most $r$ in original coordinates can be attained with state error at most $d$ in $\Sp^\infty$.
On $\mathcal R$ replace $X$ by
\[
 X_\lambda=(1-\lambda)X+\lambda c_{n+1}(h').
\]
Choose a strictly positive $\FF_\tau$-measurable $\lambda\le1$ small enough that $|v(X_\lambda)-v(X)|\le r$.
For example, if $L$ bounds the Lipschitz constant of $v$ on $B$, take
$\lambda=\min\{1,r/(1+L|X-c_{n+1}(h')|)\}$.
Then $X_\lambda\in\operatorname{ri}C_{n+1}(h')$ also for singleton fibers.
On nonreturn keep $U_\tau=\eta_h$.
The perturbation lemma supplies a solution satisfying this new terminal constraint whose transformed initial value differs from that of the original solution by at most $c d$ for $d\le1$.
Here $c$ comes from the Lipschitz bound for $u$ on the common bounded state range.

For this $h$, write $J=I_n(\operatorname{ri}C_{n+1})(h)$.
Letting $d\downarrow0$ in the preceding construction gives
\[
 J\ \subset\ I_n(C_{n+1})(h)\ \subset\ \overline J .
\]
In particular $J\ne\varnothing$, and \eqref{eq:ex-fixed} gives $\overline{\conv}\,J=C_n(h)$.
Apply Theorem~\ref{thm:attainable-interior} to \eqref{eq:ex-payments} with $H=\operatorname{ri}C_{n+1}$.
The stopping time $\tau$ is strictly positive, and the terminal constraint has a measurable graph and takes values in $B$.
The theorem gives both convexity of $\overline J$ and $\operatorname{ri}\overline J\subset J$.
Hence $\overline J=C_n(h)$ and $\operatorname{ri}C_n(h)\subset J$, as claimed.
By \eqref{eq:ex-historylaws} and \eqref{eq:ex-fixed}, the histories fixed at the start form a set of full $\mu_n$-measure. On this same set the inclusion holds for every $x\in\operatorname{ri} C_n(h)$.
\end{proof}

\subsection{Pasting solutions on the given Brownian filtration}\label{subsec:ex-pasting}

\begin{proof}[Proof of the existence assertion in Theorem~\ref{thm:universal}]
Choose $x_0\in\operatorname{ri}C_0$ at the deterministic initial history.
Proposition~\ref{prop:ex-exact} gives a solution on the first interval with $U_0=x_0$ and transformed endpoint in $\operatorname{ri}C_1$ on return.
For $n\ge1$, suppose the preceding interval has been constructed. On $A_n$, its transformed endpoint $x_n=U_{\sigma_n}$ has a Borel version $x_n(h)$ under $\mu_n$ because it is $\FF_{\sigma_n}$-measurable, and $x_n(h)\in\operatorname{ri}C_n(h)$ for $\mu_n$-almost every $h$.
For $\mu_n$-almost every such $h$, Proposition~\ref{prop:ex-exact} gives a solution with transformed initial value $x_n(h)$, and Lemma~\ref{lem:history-selection} selects one under $\mu_n$ with a predictable integrand on the original filtration.
On return its transformed endpoint $x_{n+1}=U_{\sigma_{n+1}}$ lies in $\operatorname{ri}C_{n+1}(h')$ and gives the prescribed initial value for the next interval.
Equation~\eqref{eq:ex-historylaws} transfers the next-stage exceptional set to a null set of returns.
These assertions hold simultaneously over the countably many stages. By \eqref{eq:ex-returns}, almost every path reaches time $1$ after finitely many stages, with transformed terminal state $U_1=u(\xi)$.

The transformed state stays in $B$ throughout every interval by \eqref{eq:conditionalU}.
The process $Y$ obtained by pasting is therefore uniformly bounded.
For $m\ge1$, stop the concatenation at $\sigma_m$ and keep the state constant thereafter. This gives a solution on $[0,1]$ with terminal condition $Y_{\sigma_m}\in v(B)$ and integrand $\ind_{\{t\le\sigma_m\}}Z_t$. Each of its finitely many pieces has finite energy almost surely.
Applying Proposition~\ref{prop:bounds} to this stopped solution gives, for every stopping time $\rho\le1$,
\[
 \E_{\mathbb W}\!\left[\int_\rho^1\ind_{\{t\le\sigma_m\}}|Z_t|^2\dt
          \,\middle|\,\FF_\rho\right]\le c
\]
with $c$ independent of $m$ and $\rho$.
As $\sigma_m\uparrow1$ almost surely, conditional monotone convergence gives $Z\in\BMO$.
The stopped integrands converge to $Z$ in $\Hp^2$, so their stochastic integrals converge in $\Sp^2$ and their drift integrals converge in $L^1$ of the uniform norm.
Passing to the limit in the finite concatenations gives a solution $(Y,Z)\in\Sp^\infty\times\BMO$ of \eqref{eq:bsde} with $Y_1=\xi$ almost surely.
Thus $(Y,Z)\in\Sol_\theta(\xi)$, contradicting the choice of $\xi$.
Finally, the path representation and deterministic time change at the start of the proof transfer the resulting solution to the prescribed Brownian motion and horizon, with the prescribed terminal condition up to almost sure equality.
\end{proof}

\section{Universal uniqueness on the critical curve and nonuniqueness off it}\label{sec:finite}\label{sec:nonuniqueness}
Fix $\theta=(\alpha_1,\alpha_2)\in\R^2$. Suppose first that $(\alpha_1-1)(\alpha_2-1)=1$.
For $\alpha_1\alpha_2\ne0$, existence and uniqueness in $\Sp^\infty\times\BMO$ for every bounded terminal condition were established in \cite[Corollary~2.25]{FHT2025}.
For $L=kY_1-Y_2$ to satisfy a scalar quadratic BSDE, we require $kq_{\theta,1}(z)-q_{\theta,2}(z)=d(kz_1-z_2)^2$ for every $z\in\R^2$. Comparing the coefficients of $z_2^2$, $z_1z_2$ and $z_1^2$ gives successively $2d=\alpha_2-2$, $k(\alpha_2-1)=1$ and $k=\alpha_1-1$. Thus on the critical curve, including the origin, we take $k=\alpha_1-1$ and $h=kZ_1-Z_2$. Then
\[
 kq_{\theta,1}(Z)-q_{\theta,2}(Z)=\tfrac12(\alpha_2-2)h^2,\qquad
 q_{\theta,2}(Z)=\tfrac1k hZ_2+\tfrac12\alpha_2 Z_2^2.
\]
The first identity determines $(L,h)$ by the scalar quadratic BSDE with terminal condition $k\xi_1-\xi_2$. Indeed, under $\PP$, for a bounded scalar terminal condition $\eta$ and generator $\lambda z^2/2$, the unique bounded solution has value $\lambda^{-1}\log\E[e^{\lambda\eta}\mid\FF_t]$ for $\lambda\ne0$ and $\E[\eta\mid\FF_t]$ for $\lambda=0$. Martingale representation determines its BMO integrand.
Let $\mathbb Q$ have density $\cE((h/k)\cdot W)_T$ with respect to $\PP$. Under $\mathbb Q$, the equation for $Y_2$ has generator $\alpha_2 Z_2^2/2$. The scalar formula with $\eta=\xi_2$, $\lambda=\alpha_2$ and conditional expectation under $\mathbb Q$ determines $Y_{2,t}$. Martingale representation under $\mathbb Q$ determines $Z_2$, which belongs to $\BMO(\PP)$ by the BMO Girsanov isomorphism. The identities defining $L,h$ now recover $(Y_1,Z_1)$ and prove uniqueness. At $(0,0)$, $Y_{2,t}=\E^{\mathbb Q}[\xi_2\mid\FF_t]$. At $(2,2)$, $L_t=Y_{1,t}-Y_{2,t}=\E[\xi_1-\xi_2\mid\FF_t]$.

For $(\alpha_1-1)(\alpha_2-1)\ne1$ we construct two distinct solutions in $\Sp^\infty\times\BMO$ with the same bounded terminal condition. We first ask how many distinct terminal values are needed for nonuniqueness. Every terminal condition taking at most two values has a unique solution. Three values suffice to encode any bounded terminal condition and preserve its entire solution set.
\Needspace{8\baselineskip}
\subsection{Uniqueness for two-valued terminal conditions}
\begin{proposition}\label{prop:binary}
For every $\theta\in\R^2$ and every terminal condition $\xi=a\ind_{E^c}+b\ind_E$, with $a,b\in\R^2$ and $E\in\FF_T$, equation~\eqref{eq:bsde} has a unique solution in $\Sp^\infty\times\BMO$. If $\xi\in X_M$, then $\xi\in\mathcal D_M^\theta\subset\mathcal G_M^\theta$. More generally, both conclusions hold whenever $u(\xi)$ lies almost surely on a deterministic line segment contained in $D_\theta$.
\end{proposition}
\begin{proof}
For $a\ne b$, put $A=u(a)$, $B=u(b)$ and $e=B-A$. Lemma~\ref{lem:segment} gives
\begin{equation}\label{eq:binaryexplicit}
 P_t=\PP(E\mid\FF_t),\qquad
 r_t=\gamma_{A,e}^{-1}(P_t),\qquad
 Y_t=v_\theta(A+er_t).
\end{equation}
If $\dd P_t=\zeta_t\dd W_t$, the corresponding integrand $Z$ is
\begin{equation}\label{eq:binarycontrol}
 Z_{i,t}=\frac{e_i\zeta_t}
 {\big(1-\alpha_i(A_i+e_ir_t)\big)\gamma_{A,e}'(r_t)}.
\end{equation}
For $a=b$, the solution is constant, by \eqref{eq:conditionalU}. For a general transformed terminal value $U_T=A+er_T$ with $0\le r_T\le1$, replace $P_t$ in \eqref{eq:binaryexplicit} by $\E[\gamma_{A,e}(r_T)\mid\FF_t]$. In either case the solution uses a single segment in Definition~\ref{def:core}, so $\xi\in\mathcal D_M^\theta$ whenever $\xi\in X_M$. Section~\ref{sec:baire} proves $\mathcal D_M^\theta\subset\mathcal G_M^\theta$.
\end{proof}

\subsection{A reduction from bounded to three-valued terminal conditions}
\begin{theorem}[Reduction to three-valued terminal conditions]\label{thm:encoding}
Fix $\theta\in\R^2$ and $M>0$. On a Brownian extension to $[0,3T]$ there are three distinct deterministic vectors depending only on $\theta,M$ and a map $\Theta$ on $X_M$ taking each $\xi$ to a terminal condition which assumes each of these three values with positive probability. For each $\xi\in X_M$, restriction to $[0,T]$ and extension by two segment solutions give a bijection of solution sets in $\Sp^\infty\times\BMO$ on the respective intervals,
\begin{equation}\label{eq:solbijection}
 \Sol_{\theta,[0,T]}(\xi)
 \longleftrightarrow
 \Sol_{\theta,[0,3T]}(\Theta(\xi)).
\end{equation}
The map $\Theta$ satisfies the bounds
\begin{equation}\label{eq:amplitudegrowth}
 |\Theta_i(\xi)|\le M',\qquad
 M'=\begin{cases}
 \min\{4M,\,M+(\log 2)/\min_j|\alpha_j|\},&\alpha_1\alpha_2\ne0,\\
 4M,&\alpha_1\alpha_2=0,
 \end{cases}
\end{equation}
and, for a constant depending only on $M$ and $\theta$,
\begin{equation}\label{eq:encodingcontinuous}
 \norm{\Theta(\xi)-\Theta(\zeta)}_{L^2}^2
 \le C\E|\xi-\zeta|,\qquad \xi,\zeta\in X_M.
\end{equation}
\end{theorem}

\begin{proof}
We enclose $u([-M,M]^2)$ in the interior of a triangle in $D_\theta$. Two successive binary choices will send each point of this set to the three vertices, with the segment formula fixing the continuation on both intervals.
For $i=1,2$, put $\varepsilon_i=-1$ if $\alpha_i>0$ and $\varepsilon_i=1$ otherwise. The function
\[
 f_i(x)=\varepsilon_i\phi_{\alpha_i}(\varepsilon_i x)
       =\int_0^x e^{|\alpha_i|t}\,\dd t
\]
is increasing and convex. Choose $m_-,m_+>M$, set $l_i=\phi_{\alpha_i}(-\varepsilon_i m_-)$ and $h_i=\phi_{\alpha_i}(\varepsilon_i m_+)$, and take the three vertices
\begin{equation}\label{eq:triangle}
 A=(l_1,l_2),\qquad B=(h_1,l_2),\qquad C=(l_1,h_2).
\end{equation}
For $|y_i|\le M$, the fraction of the $i$th side from $l_i$ to $\phi_{\alpha_i}(y_i)$ is
\[
 \frac{f_i(\varepsilon_i y_i)-f_i(-m_-)}
      {f_i(m_+)-f_i(-m_-)}.
\]
With $m_-=3M/2$ and $m_+=4M$, convexity bounds this fraction by $(M+m_-)/(m_++m_-)=5/11$. If $\alpha_1\alpha_2\ne0$, we may also take $m_-=m_+=K$ with $K=M+(\log 2)/\min_i|\alpha_i|$. In this case the fraction is at most
\[
 \frac{e^{|\alpha_i|M}-e^{-|\alpha_i|K}}{e^{|\alpha_i|K}-e^{-|\alpha_i|K}}
 <e^{|\alpha_i|(M-K)}\le\tfrac12 .
\]
Take the smaller available $m_+$. The two fractions are positive and their sum is strictly less than one. Hence $u([-M,M]^2)$ is a compact subset of the interior of $\Delta=\conv\{A,B,C\}$, and $v_\theta(A),v_\theta(B),v_\theta(C)$ satisfy \eqref{eq:amplitudegrowth}.

Write $U\in\operatorname{int}\Delta$ as $U=A+x(B-A)+s(C-A)$ and put
\[
 v_L(s)=A+s(C-A),\qquad v_R(s)=B+s(C-B),\qquad e_s=v_R(s)-v_L(s).
\]
Then $U=v_L(s)+e_s x/(1-s)$. Define the three probabilities by the segment scales in \eqref{eq:gamma},
\begin{align}
 p(x,s)&=\gamma_{v_L(s),e_s}\!\left(\frac{x}{1-s}\right),
 \label{eq:pfirst}\\
 p_L(s)&=\gamma_{A,C-A}(s),
 \label{eq:pleft}\\
 p_R(s)&=\gamma_{B,C-B}(s).
 \label{eq:pright}
\end{align}
These functions are smooth on $x,s>0$, $x+s<1$, and bounded away from zero and one on the compact set of coordinates corresponding to $u([-M,M]^2)$.

Extend the Brownian motion from $[0,T]$ to $[0,3T]$ using independent future increments, with the augmented natural filtration on the enlarged interval. Let $\Phi$ be the standard normal distribution function and set
\[
 H_1=\Phi\left(\frac{W_{2T}-W_T}{\sqrt T}\right),\qquad
 H_2=\Phi\left(\frac{W_{3T}-W_{2T}}{\sqrt T}\right).
\]
These are independent uniform random variables, independent of $\FF_T$.

Given $\xi\in X_M$, take $x,s$ to be the coordinates of $u(\xi)$ in $\Delta$ and define
\begin{equation}\label{eq:IJ}
 I=\ind_{\{H_1\le p(x,s)\}},\qquad
 J=\ind_{\{H_2\le(1-I)p_L(s)+Ip_R(s)\}}.
\end{equation}
Define the transformed terminal value $\widehat\eta$ and the corresponding terminal condition $\Theta(\xi)$ by
\begin{equation}\label{eq:encoding}
 \widehat\eta=A\ind_{\{I=0,J=0\}}
             +B\ind_{\{I=1,J=0\}}+C\ind_{\{J=1\}},
 \qquad \Theta(\xi)=v_\theta(\widehat\eta).
\end{equation}
The strict bounds on the three probabilities give positive probability to each of the three distinct terminal values.

On $[T,2T]$, solve the segment problem with endpoints $v_L(s),v_R(s)$ and terminal event $\{I=1\}$. The processes in Lemma~\ref{lem:segment} satisfy $P_T=p(x,s)$ and $r_T=\gamma_{v_L(s),e_s}^{-1}(P_T)=x/(1-s)$. Thus the transformed initial value is $u(\xi)$ and the terminal value is
\begin{equation}\label{eq:middle}
 U_{2T}=(1-I)v_L(s)+Iv_R(s).
\end{equation}
On $[2T,3T]$, use the segment $[A,C]$ on $\{I=0\}$ and $[B,C]$ on $\{I=1\}$, with terminal event $\{J=1\}$. By \eqref{eq:pleft}--\eqref{eq:IJ}, its initial value is \eqref{eq:middle} and its terminal value is $\widehat\eta$. The two solutions concatenate to a bounded solution on $[T,3T]$ with $Z\in\BMO$ and initial value $\xi$ in the original coordinates. This extends any solution on $[0,T]$ and defines the forward map in \eqref{eq:solbijection}.

For the reverse map, consider any bounded solution with terminal condition $\Theta(\xi)$. The random variable $I$ is $\FF_{2T}$-measurable. On the event $\{I=0\}$ its transformed terminal value belongs to $[A,C]$, and on $\{I=1\}$ it belongs to $[B,C]$. Lemma~\ref{lem:segment} shows that its restriction to $[2T,3T]$ coincides with the constructed solution, and hence gives \eqref{eq:middle}. At time $T$, both endpoints $v_L(s),v_R(s)$ are known. Applying the same lemma on $[T,2T]$ forces $U_T=u(\xi)$. Restriction to $[0,T]$ is therefore a solution with terminal condition $\xi$. This also proves that the extension and restriction maps are inverse.

To prove \eqref{eq:encodingcontinuous}, let $(x,s)$ and $(x',s')$ be the coordinates of $u(\xi)$ and $u(\zeta)$ in $\Delta$. Use the same $H_1,H_2$ to form the indicators $(I,J)$ for $\xi$ and $(I',J')$ for $\zeta$. The functions $p,p_L,p_R$ are Lipschitz on the common compact rectangle in these coordinates. Since $H_1$ is uniform and independent of $\FF_T$, we have $\PP(I\ne I'\mid\FF_T)=|p(x,s)-p(x',s')|$. At time $2T$, both $I$ and $I'$ are known and $H_2$ is independent of $\FF_{2T}$. Hence $\PP(J\ne J'\mid\FF_{2T})$ equals $|p_L(s)-p_L(s')|$ on $\{I=I'=0\}$ and $|p_R(s)-p_R(s')|$ on $\{I=I'=1\}$. Thus
\[
 \PP(\Theta(\xi)\ne\Theta(\zeta))
 \le C\E|u(\xi)-u(\zeta)|
 \le C\E|\xi-\zeta|.
\]
The three possible terminal values are fixed and bounded, yielding the claim.
\end{proof}

\paragraph{Changing the time horizon.}
For $c>0$, the rescaling $\widetilde W_s=c^{-1/2}W_{cs}$, $\widetilde Y_s=Y_{cs}$, $\widetilde Z_s=\sqrt c\,Z_{cs}$ preserves \eqref{eq:bsde}, since $q_\theta(\sqrt c\,z)=c\,q_\theta(z)$. On augmented natural Brownian filtrations the terminal condition and the solution admit measurable path representations. To obtain a three-valued terminal condition at a prescribed time $S>0$, realize the original problem on $[0,S/3]$ and use the Brownian increments on $[S/3,2S/3]$ and $[2S/3,S]$ for the two extensions in Theorem~\ref{thm:encoding}.

\subsection{Failure of universal uniqueness for three-valued terminal conditions}
\begin{proposition}[Nonuniqueness for a three-valued terminal condition]\label{prop:three-counterexample}
For every $\theta\in\R^2$ satisfying $(\alpha_1-1)(\alpha_2-1)\ne1$ and every $0<T<\infty$, there exists $\xi\in L^\infty(\FF_T;\R^2)$ taking exactly three values, each with positive $\PP$-probability, for which \eqref{eq:bsde} admits at least two distinct solutions in $\Sp^\infty\times\BMO$.
\end{proposition}

We construct a bounded reference solution from an auxiliary diffusion with initial state $a$. At the end of each block the probability of continuing depends on the current state and a parameter $c\in(0,1)$. In the linear coordinates of \eqref{eq:nu-transformed}, the terminal condition $\xi_{a,c}$ is the value of the reference solution at absorption. If the diffusion continues from state $a'$, let $x(a')$ be the correction to the reference value in rotating coordinates. Solving one block backwards gives the correction $\mathcal F_cx$ at its start, so successive blocks agree when $x=\mathcal F_cx$. The periodic coefficients are chosen so that $D_x\mathcal F_c(0)$ has a simple eigenvalue $c\lambda_+$ with $\lambda_+>1$. A branch of nonzero fixed points then bifurcates at $c=\lambda_+^{-1}$. For a pair $(c,x)$ on this branch and an $a$ with $x(a)\ne0$, the fixed points $0$ and $x$ give different initial values for the same $\xi_{a,c}$. Applying Theorem~\ref{thm:encoding} to $\mathsf T^{-1}\xi_{a,c}$ in the original coordinates gives the three-valued terminal condition.

\subsubsection{A change of coordinates and periodic coefficients}\label{subsec:phase-coordinates}
We first reduce pairs of opposite signs to pairs of the same sign. After interchanging coordinates if necessary, assume $\alpha_1<0<\alpha_2$ and set $\omega=1-\alpha_1^{-1}-\alpha_2^{-1}\ne0$. Define the linear map $S$ and the pair $\beta=(\beta_1,\beta_2)$ by
\[
 \begin{array}{c|c|c}
 & S(y_1,y_2) & (\beta_1,\beta_2)\\ \hline
 \omega<0 &(y_1-(\alpha_2/\alpha_1)y_2,-\alpha_2\omega y_2)
           &(\alpha_1,1/\omega)\\
 \omega>0 &(-\alpha_1\omega y_1,y_2-(\alpha_1/\alpha_2)y_1)
           &(1/\omega,\alpha_2).
 \end{array}
\]
Expanding the quadratics gives $S q_\theta(z)=q_\beta(Sz)$. In the displayed ordering the determinants are $-\alpha_2\omega$ and $-\alpha_1\omega$, respectively. In both cases $\beta_1\beta_2>0$, and the new reciprocal sum is $1-1/\alpha_2$ or $1-1/\alpha_1$, respectively. For a pair already of the same sign take $S=I$, $\beta=\theta$. When exactly one parameter is zero, take $S$ to be the permutation for which $\beta=S\theta$ satisfies $\beta_1=0$. In the mixed-sign case we include any initial permutation in the map $S$ from the original coordinates.

For $\beta_1\beta_2>0$ put $R=\beta_1^{-1}+\beta_2^{-1}$, $s=R-1$, $k=|s|$, and
\[
 e=\frac{(\beta_2-\beta_1)\sqrt{k}}{2\sqrt{\beta_1\beta_2}},\qquad
 T_0(y_1,y_2)=\left(\frac{y_1+y_2}{R},
 \frac{\beta_1 y_1-\beta_2 y_2}{R\sqrt{\beta_1\beta_2 k}}\right).
\]
Here $R\ne0$ because the two parameters have the same sign, and $k>0$ by the hypothesis of Proposition~\ref{prop:three-counterexample}. Write $Q=T_0 q_\beta T_0^{-1}$ and let $\mathsf L$ denote the linear map giving the upper row of the traceless part of $DQ$. Substitution of
$T_0^{-1}(x,y)=(x/\beta_1+\sqrt{k/(\beta_1\beta_2)}y,\,
x/\beta_2-\sqrt{k/(\beta_1\beta_2)}y)$ yields
\[
 Q(x,y)=\big((R-\tfrac12)x^2-\tfrac k2y^2,\ sxy+ey^2\big),
 \qquad \mathsf L(x,y)=(Rx/2-ey,-ky).
\]
If $\beta_1=0$ and $\beta_2\ne0$, put $\sigma=\operatorname{sgn}(\beta_2)$, $e=-\sigma\sqrt{|\beta_2|}/2$, and $T_0(y_1,y_2)=(y_1+y_2,\sqrt{|\beta_2|}\,y_2)$. The corresponding maps are
\[
 Q(x,y)=\big(x^2-\tfrac{\sigma}{2}y^2,\ xy+ey^2\big),
 \qquad \mathsf L(x,y)=(x/2-ey,-\sigma y).
\]
In each case the invertible map $\mathsf T=T_0 S$ preserves boundedness and the BMO property. The pair $(\mathcal Y,\mathcal Z)=(\mathsf T Y,\mathsf T Z)$ satisfies
\begin{equation}\label{eq:nu-transformed}
 \dd\mathcal Y=-Q(\mathcal Z)\dt+\mathcal Z\dd W,\qquad Q=\mathsf T q_\theta\mathsf T^{-1}.
\end{equation}
Write $A(z)=DQ(z)$ and $t(z)=\operatorname{tr}A(z)/2$. The determinants of $\mathsf L$ in the two formulas above are $-Rk/2$ and $-\sigma/2$, respectively. If $(u,w)=\mathsf L z$, then
\[
 A(z)=t(z)I+
 \begin{pmatrix}u&w\\ w&-u\end{pmatrix}
 \quad\text{or}\quad
 A(z)=t(z)I+
 \begin{pmatrix}u&w\\ -w&-u\end{pmatrix}.
\]
The symmetric form occurs for $\beta_1\beta_2>0$ with $R<1$ and for $\beta_1=0$, $\beta_2<0$. The second form occurs for $\beta_1\beta_2>0$ with $R>1$ and for $\beta_1=0$, $\beta_2>0$.
Let $R_\vartheta=\exp(\vartheta J)$, where
$J=\left(\begin{smallmatrix}0&-1\\1&0\end{smallmatrix}\right)$.
For a periodic function $f$, let $\overline f=(2\pi)^{-1}\int_0^{2\pi}f(\vartheta)\dd\vartheta$.

In the drift of the linearized equation \eqref{eq:nu-triangular}, rotation adds $J$ to the matrix $R_{-\vartheta}A(z)R_\vartheta$ multiplying the Brownian integrands. We prescribe the upper-right entry of $R_{-\vartheta}A(z)R_\vartheta$ to be $1$ to cancel $J_{12}=-1$. Together with the choice of phase drift below, this gives a closed scalar equation for the first coordinate and an exponential weight in the formula for its initial value.

\begin{lemma}[Periodic coefficients for the reference solution]\label{lem:nu-phase}
There are real trigonometric polynomials $z:\mathbb S^1\to\R^2$ and $r,h,\ell:\mathbb S^1\to\R$ such that
\begin{equation}\label{eq:nu-rotation}
 R_{-\vartheta}A(z(\vartheta))R_\vartheta
 =\begin{pmatrix}r(\vartheta)&1\\\ell(\vartheta)&h(\vartheta)\end{pmatrix},
 \qquad \overline{Q(z)-rz}=0.
\end{equation}
If $F$ is the zero-mean periodic solution of $F''/2=Q(z)-rz$, then the range of $v=z+F'$ is not contained in any affine line. Moreover, $\ell=1$ for the symmetric form of $A$, and $\ell+1\ge0$ with strict inequality somewhere for the second form.
\end{lemma}

Appendix~\ref{app:phase} constructs the coefficients. The zero-mean condition gives the periodic function $F$ used in the bounded reference solution. Appendix~\ref{app:operator} uses the full affine span of the range of $v$ to prove smoothing and accessibility for the diffusion. There the condition on $\ell$ ensures that the spectral radius of the first diagonal operator exceeds the norm of the second.

\subsubsection{A bounded reference solution}
Fix the functions in Lemma~\ref{lem:nu-phase}. Choose $b>3$ and then $L>1$, sufficiently large for Lemma~\ref{lem:nu-operator} below. Write $D_a$ for the open disk of radius $a$ in $\R^2$, and set
\[
 C_0=\overline D_{b\sqrt L}\times\mathbb S^1,\quad
 C=\overline D_{(b+1)\sqrt L}\times\mathbb S^1,\quad
 K=2b\sqrt L,\quad K_1=(2b+1)\sqrt L.
\]
Choose a smooth radial function $\rho:\R^2\to[0,1]$, equal to one on $\overline D_K$, strictly positive on $D_{K_1}$, and zero outside $D_{K_1}$. Under the original probability measure, define
\begin{equation}\label{eq:nu-state}
 \begin{split}
 \dd\vartheta_s&=\rho(M_s)\dd B_s-\rho(M_s)^2r(\vartheta_s)\ds,\\
 \dd M_s&=\rho(M_s)v(\vartheta_s)\dd B_s
       -\rho(M_s)^2r(\vartheta_s)v(\vartheta_s)\ds.
 \end{split}
\end{equation}
Here $B$ is a one-dimensional Brownian motion and the phase is taken modulo $2\pi$. Write $X=(M,\vartheta)$ and $A_s=\int_0^s\rho(M_u)^2\dd u$. The coefficients are smooth and bounded with bounded derivatives. The closed disk $\overline D_{K_1}$ is invariant, since both coefficients vanish on its boundary and pathwise uniqueness makes a path constant after its first boundary hit.

Set
\[
 y(M,\vartheta)=M-F(\vartheta),\qquad
 \mathcal Y_s^0=y(X_s),\qquad \mathcal Z_s^0=\rho(M_s)z(\vartheta_s).
\]
In stochastic differentials and integrals we write $\rho$ for $\rho(M_s)$ and evaluate each phase function at $\vartheta_s$ when its argument is omitted. Using $v=z+F'$ and $F''/2=Q(z)-rz$, It\^o's formula gives
\[
 \dd\mathcal Y^0=-\rho^2Q(z)\ds+\rho z\dd B
               =-Q(\mathcal Z^0)\ds+\mathcal Z^0\dd B.
\]
Thus $|\mathcal Y^0|\le K_1+\|F\|_\infty$ and $|\mathcal Z^0|\le\|z\|_\infty$.
For $g\in\{r,h\}$ and initial state $X_0=a$, write $Q_a^g$ for the probability measure on $[0,L]$ with density $\cE(\int_0^\cdot\rho(M_s)g(\vartheta_s)\dd B_s)_L$ and $\E_a^g$ for its expectation. The integrand is bounded, and
\[
 B_t^g=B_t-\int_0^t\rho(M_s)g(\vartheta_s)\ds
\]
is Brownian under $Q_a^g$. We abbreviate $Q_a^g$ to $Q^g$ when the initial state is understood. Under $Q^r$ the state dynamics reduce to
\[
 \dd\vartheta=\rho\dd B^r,\qquad\dd M=\rho v(\vartheta)\dd B^r.
\]
The reference solution and the nonlinear construction below remain defined under the original measure.

Choose a smooth function $\phi:\R^2\times\mathbb S^1\to[0,1]$, independent of the phase, equal to one on $C_0$, positive on $\operatorname{int}C$, and zero outside $\operatorname{int}C$.

\subsubsection{Binary interpolation and the local solution map}

For $a,b\in\R^2$ and an event $E\in\FF_T$ of probability $p$, let $\mathfrak B_p(a,b)$ denote the initial value $Y_0$ of the solution with terminal condition $a\ind_{E^c}+b\ind_E$. By Proposition~\ref{prop:binary} and \eqref{eq:segment-mean},
\[
 \mathfrak B_p(a,b)=v_\theta\!\big(M_p(u(a),u(b))\big).
\]
For stopping times $\sigma\le\tau$, bounded $\FF_\sigma$-measurable $a,b$ and $E\in\FF_\tau$, the solution on $[\sigma,\tau]$ with terminal condition $a\ind_{E^c}+b\ind_E$ at $\tau$ satisfies $Y_\sigma=\mathfrak B_{\PP(E\mid\FF_\sigma)}(a,b)$. This follows from Lemma~\ref{lem:segment} on $\{a\ne b\}$. On $\{a=b\}$, \eqref{eq:conditionalU} gives $Y=a$ throughout the interval.
The map $(p,a,b)\mapsto\mathfrak B_p(a,b)$ is smooth, including at $a=b$, since the derivative of $\gamma_{A,e}$ in \eqref{eq:gamma} remains positive at $e=0$. At $p=0,1$ the implicit function theorem applies after extending the segment slightly within $D_\theta$.

Since $q_\theta$ is independent of $Y$, $\mathfrak B_p(a,b)=a+\mathfrak B_p(0,b-a)$. In the linear coordinates of \eqref{eq:nu-transformed}, the initial value is $\mathfrak B_p^{\mathsf T}(a,b)=a+\Psi(p,b-a)$, where
\[
 \mathfrak B_p^{\mathsf T}(a,b)
   =\mathsf T\mathfrak B_p(\mathsf T^{-1}a,\mathsf T^{-1}b),\qquad
 \Psi(p,\delta)=\mathfrak B_p^{\mathsf T}(0,\delta).
\]
The map $\Psi$ is smooth from $[0,1]\times\R^2$ to $\R^2$. For $p\in[0,1]$ and $|\delta|$ sufficiently small,
\begin{equation}\label{eq:nu-binaryderivative}
 \Psi(p,0)=0,\qquad D_\delta\Psi(p,0)=pI,
 \qquad \Psi(0,\delta)=0,\qquad
 |\Psi(p,\delta)|\le C_\Psi p|\delta|.
\end{equation}
To compute the derivative in the original $Y$ coordinates, set $a=0$, $b=t\delta$ in \eqref{eq:binaryexplicit}. Since $u(0)=0$ and $Du(0)=I$, we have $e=u(t\delta)=t\delta+O(t^2)$ and $\gamma_{0,e}^{-1}(p)=p+O(t)$. Thus the initial value is $pt\delta+O(t^2)$. The fixed linear change of coordinates preserves the derivative $pI$. Put $\mathcal X=C(C;\R^2)$ with the supremum norm. For small $x\in\mathcal X$ and $c\in(0,1)$, define a continuous function on the state space by
\begin{equation}\label{eq:nu-active-terminal}
 \eta_{c,x}(a')=
 \begin{cases}
 \Psi\big(c\phi(a'),R_{\vartheta_{a'}}x(a')\big),&a'\in C,\\
 0,&a'\notin C.
 \end{cases}
\end{equation}
Here $\vartheta_{a'}$ is the phase coordinate of $a'$. By \eqref{eq:nu-binaryderivative}, the two expressions agree continuously at $\partial C$. The map $(c,x)\mapsto\eta_{c,x}$ is $C^2$ in the supremum norm, with uniform local bounds. Indeed, $\Psi(0,\delta)=0$ and its derivatives with respect to $\delta$ vanish at $p=0$, while differentiation with respect to $c$ supplies a factor $\phi$.

For $a'\in C$, the terminal correction $\eta_{c,x}(a')$ is the initial value of the segment solution with terminal correction $R_{\vartheta_{a'}}x(a')$ on continuation and $0$ on absorption. The conditional probability of continuation at $a'$ is $c\phi(a')$.

For a deterministic initial state $X_0=a\in C$, solve on $[0,L]$
\begin{equation}\label{eq:nu-local-bsde}
 \dd V_s=-\big(DQ(\mathcal Z_s^0)H_s+Q(H_s)\big)\ds+H_s\dd B_s,
 \qquad V_L=\eta_{c,x}(X_L).
\end{equation}
Since $Q$ is homogeneous quadratic, $(\mathcal Y^0+V,\mathcal Z^0+H)$ solves \eqref{eq:nu-transformed}. Define
\[
(\mathcal F_c x)(a)=R_{-\vartheta_0}V_0^{a,c,x}.
\]
For fixed $a,c,x$, the time-zero correction $V_0^{a,c,x}$ is deterministic.

\begin{lemma}[Local solution map]\label{lem:nu-local}
Fix $0<c_-<c_+<1$. There is $\varepsilon>0$, independent of $a\in C$
and $c\in[c_-,c_+]$, such that \eqref{eq:nu-local-bsde} has a unique
solution in a fixed small ball of $\Sp^\infty\times\BMO$ for
$\|x\|_\infty<\varepsilon$. The map
$(c,x)\mapsto\mathcal F_c x$ is $C^2$ into $\mathcal X$ and
$\mathcal F_c0=0$.
\end{lemma}

Appendix~\ref{app:local} proves the lemma. On continuation from $a'\in C$, the segment solution ends with correction $R_{\vartheta_{a'}}x(a')$, and the next local solution starts with correction $R_{\vartheta_{a'}}(\mathcal F_cx)(a')$. Thus the two corrections agree whenever $x=\mathcal F_cx$.

\subsubsection{Linearization and the positive operator}

Fix the initial state $a\in C$ and linearize \eqref{eq:nu-transformed} about the reference solution $(\mathcal Y^0,\mathcal Z^0)$. The variation satisfies $\dd\delta\mathcal Y=-DQ(\mathcal Z^0)\delta\mathcal Z\ds+\delta\mathcal Z\dd B$. Let $U,V$ denote the two scalar coordinates of the state variation in the rotating frame and $p,q$ their Brownian integrands. Then
\[
 \delta\mathcal Y=R_\vartheta\binom UV,\qquad
 \delta\mathcal Z=R_\vartheta\left[\binom pq+\rho J\binom UV\right].
\]
Equations \eqref{eq:nu-rotation} and \eqref{eq:nu-state}, together with
$\dd R_\vartheta=\rho R_\vartheta J\dd B-\rho^2R_\vartheta(rJ+I/2)\ds$, give
\begin{equation}\label{eq:nu-triangular}
 \begin{split}
 \dd U&=-\tfrac12\rho^2U\ds-\rho rp\ds+p\dd B,\\
 \dd V&=\rho^2\big((r-h)U+(\ell+\tfrac12)V\big)\ds
       -\rho\big((\ell+1)p+hq\big)\ds+q\dd B.
 \end{split}
\end{equation}
Under $Q^r$ the first equation becomes $\dd U=-\rho^2U\ds/2+p\dd B^r$, so $U_0=\E_a^r[e^{A_L/2}U_L]$.
On the space $C(C)$ of real-valued continuous functions on $C$, define
\[
 T_+f(a)=\E_a^r[e^{A_L/2}\phi(X_L)f(X_L)],
\]
\[
 T_-f(a)=\E_a^h\left[
 e^{-\int_0^L\rho(M_s)^2(\ell(\vartheta_s)+1/2)\ds}
 \phi(X_L)f(X_L)\right].
\]
For scalar or vector functions $f$ on $C$, we extend $\phi f$ by zero to the whole state space.
Differentiating \eqref{eq:nu-active-terminal} at $x=0$ in a direction $h_0\in\mathcal X$ gives
$\delta\mathcal Y_L=cR_{\vartheta_L}(\phi h_0)(X_L)$.
Solving \eqref{eq:nu-triangular} backwards therefore yields
\[
 D_x\mathcal F_c(0)=c\mathcal T,\qquad
 \mathcal T=\begin{pmatrix}T_+&0\\\mathcal B&T_-\end{pmatrix},
\]
where $\mathcal B:C(C)\to C(C)$ is bounded by Lemma~\ref{lem:nu-local}.

\begin{lemma}\label{lem:nu-operator}
For sufficiently large $b$ and then sufficiently large $L$, the operator $T_+$ is compact and strongly positive on $C(C)$, and
\begin{equation}\label{eq:nu-spectral}
 r(T_+)\ge p_*e^{L/2}>\max\{1,\|T_-\|\},
\end{equation}
where $p_*>0$ is independent of $L$. Consequently $r(T_+)$ is an algebraically simple positive eigenvalue, with a strictly positive eigenfunction and a strictly positive eigenfunctional.
\end{lemma}

Appendix~\ref{app:operator} proves compactness of $T_+$ by smoothing for the free diffusion and strong positivity by the support theorem. The return probability under $Q^r$ gives the lower bound for its spectral radius in \eqref{eq:nu-spectral}, while an estimate under $Q^h$ bounds $\|T_-\|$.

\subsubsection{Bifurcation of nonzero fixed points}

Write $\lambda_+=r(T_+)$ and $c_0=\lambda_+^{-1}\in(0,1)$. Choose a strictly positive $h_+\in C(C)$ and a strictly positive $\ell_+\in C(C)^*$ satisfying
$T_+h_+=\lambda_+h_+$, $T_+^*\ell_+=\lambda_+\ell_+$ and $\ell_+(h_+)=1$.
Since $\|c_0T_-\|<1$, the operator $I-c_0T_-$ is invertible. Set
\[
 k_+=(I-c_0T_-)^{-1}c_0\mathcal B h_+,
 \qquad e_0=\binom{h_+}{k_+}.
\]
Define $\mathcal H(c,x)=x-\mathcal F_cx$. Then
\[
 D_x\mathcal H(c_0,0)
 =\begin{pmatrix}I-c_0T_+&0\\-c_0\mathcal B&I-c_0T_-\end{pmatrix}.
\]
Elimination of the second row reduces this operator to the direct
sum of $I-c_0T_+$ and an invertible operator. It is Fredholm of index
zero. Its kernel is $\R e_0$ and its range is $\ker\Lambda$, where
$\Lambda(f,g)=\ell_+(f)$.
Moreover $\mathcal Te_0=e_0/c_0$ and
\[
\Lambda\big(D_{cx}\mathcal H(c_0,0)e_0\big)
 =-\Lambda(\mathcal Te_0)=-1/c_0\ne0.
\]
Thus Lemma~\ref{lem:nu-local} and the
Crandall--Rabinowitz theorem~\cite[Theorem 1.7]{CR1971} give a local curve $(c(\tau),x(\tau))$
and a remainder $o(\tau)$ in $\mathcal X$ with
\begin{equation}\label{eq:nu-branch}
 c(0)=c_0,\qquad x(\tau)=\tau e_0+o(\tau),\qquad
 x(\tau)=\mathcal F_{c(\tau)}x(\tau).
\end{equation}
For all sufficiently small $\tau\ne0$, $x(\tau)\ne0$ and
$c(\tau)\in(0,1)$.
Because $h_+$ has a strictly positive minimum, the first component of
$x(\tau)$ is nonzero at every initial state for sufficiently small
nonzero $\tau$.

\subsubsection{Construction on the Brownian filtration}

\begin{lemma}[Pasting]\label{lem:nu-paste}
Fix $a\in C$ and $c\in(0,1)$. There is a bounded terminal condition $\xi_{a,c}$ for \eqref{eq:nu-transformed} such that every sufficiently small fixed point $x=\mathcal F_cx$ gives a solution in $\Sp^\infty\times\BMO$ with terminal condition $\xi_{a,c}$ and initial value $y(a)+R_{\vartheta_a}x(a)$. The fixed point $x=0$ gives the reference solution. Thus $x(a)\ne0$ yields two distinct solutions for the same $\xi_{a,c}$.
\end{lemma}

\begin{proof}
Choose alternating deterministic construction and continuation intervals partitioning $[0,T)$, with positive lengths and endpoints tending to $T$. Write $t_n$ and $d_n$ for the left endpoint and length of the $n$th construction interval. On construction
interval number $n$, rescale the Brownian increment so its internal
time runs from $0$ to $L$. For a physical interval of length $d_n$,
the conversion is
\[
 s=\frac{L(t-t_n)}{d_n},\qquad
 \dd B_s=\sqrt{L/d_n}\dd W_t.
\]
Quadratic homogeneity makes
\eqref{eq:nu-transformed} invariant under this change, with the Brownian integrand multiplied
by $\sqrt{L/d_n}$ when passing from internal to physical time.

Starting at state $a$, run the reference state process on the construction
interval. On its succeeding continuation interval freeze $X$ and the reference
solution. If the endpoint state is $a'$, use the fresh Brownian
increment on the continuation interval to generate a uniform random variable
$U_n$ by its normal distribution function. Continue precisely when
$U_n<c\phi(a')$; otherwise absorb at the fixed value $y(a')$. Upon
continuation $a'\in\operatorname{int}C$, so it is a permissible initial
state for the next block. All future increments remain available in
the Brownian natural filtration.

The probability of surviving $n$ continuation decisions is at most $c^n$. Hence
absorption occurs after finitely many blocks almost surely. Let $N$ be the first block at which the continuation decision is absorption, and let $X_N^{\mathrm{end}}$ denote the state at the end of its construction interval. Set $\xi_{a,c}=y(X_N^{\mathrm{end}})$, and set $\xi_{a,c}=0$ on $\{N=\infty\}$. It is bounded by $K_1+\|F\|_\infty$ and is determined
only by the reference process and the continuation decisions.

For a fixed point $x$, use the local solution
$(\mathcal Y^0+V,\mathcal Z^0+H)$ on each construction interval. If the endpoint $a'$ lies outside $C$, then $\phi(a')=0$ and the solution stays constant at $y(a')$. If $a'\in C$, use on the continuation interval the segment solution with terminal condition
$y(a')+R_{\vartheta_{a'}}x(a')\ind_{\{U_n<c\phi(a')\}}$. Its value at the start of the continuation
interval is exactly $y(a')+\eta_{c,x}(a')$ by
\eqref{eq:nu-active-terminal}. On continuation its final value equals
the initial value of the next local solution because
$x=\mathcal F_cx$. On absorption it equals $y(a')$, and the solution
is constant thereafter. The resulting process is continuous and has
terminal condition $\xi_{a,c}$. It is uniformly bounded because the local
corrections are uniformly small and the segment solutions have uniformly
bounded endpoints.
At a later block, approximate its random initial state $A$ uniformly by finite-valued $C$-valued states measurable at the start of the block. For each approximation, use the local solution for its realized initial state on the fresh Brownian increments. By \eqref{eq:nu-initial-continuity} in Appendix~\ref{app:local}, the map $a\mapsto(V^{a,c,x},H^{a,c,x})$ on the compact set $C$ is uniformly continuous in $\Sp^p\times\Hp^p$ for a fixed $p>2$. Conditioning on the initial sigma-field thus makes the resulting corrections Cauchy in these norms. The common $\Sp^\infty$ and BMO bounds pass to the limit by conditional Fatou. The terminal corrections converge in $L^p$ to $\eta_{c,x}(X_L^A)$ and their drift integrals converge in $L^1$. The resulting local solution starts at $y(A)+R_{\vartheta_A}(\mathcal F_cx)(A)$.

We verify the BMO bound at arbitrary stopping times. Each construction interval has conditional
energy at most a common constant, by $|\mathcal Z^0|\le\|z\|_\infty$ in internal
time and Lemma~\ref{lem:nu-local}. On a continuation interval with endpoints $d_0,d_1$ in the linear coordinates, apply \eqref{eq:binarycontrol} to $\mathsf T^{-1}d_0,\mathsf T^{-1}d_1$ in the original coordinates and multiply the resulting integrand by $\mathsf T$. Writing $\dd P_t=\zeta_t\dd W_t$ for the continuation probability martingale, we obtain $\mathcal Z_t=K_t\zeta_t$ with $|K_t|$ uniformly bounded over the compact set of endpoint pairs. For coincident endpoints take $K_t=0$. The conditional variance of the continuation indicator is at most $1/4$, which gives a common bound for the BMO energy of $\mathcal Z$ on each continuation interval.
Let $C_1$ and $C_2$ bound the conditional energies of the integrand on a construction interval and a continuation interval, respectively. At the start of each continuation interval, the conditional probability of proceeding to the next construction interval is at most $c$. First summing over finitely many future intervals and then applying conditional monotone convergence bounds the remaining energy at the start of a construction interval by $(C_1+C_2)/(1-c)$. At a stopping time inside a continuation interval, the bound is
\[
 C_2+\frac{C_1+C_2}{1-c},
\]
since the conditional probability of the current continuation event is at most one. At a stopping time inside a construction interval, the same argument gives a bound with an additional $C_1$. These estimates also hold after deterministic time compression, which preserves the integral of the squared integrand. Thus the pasted integrand belongs to $\BMO$.

The local dynamics and continuity at block endpoints give the BSDE
up to any finite block. Global BMO gives $L^2$ convergence of the
stochastic integrals, and quadratic growth gives $L^1$ convergence of
the drift integrals. Almost sure finite absorption identifies the
terminal value as $\xi_{a,c}$, proving the equation on $[0,T]$.
Distinctness follows from the different deterministic initial values
$y(a)$ and $y(a)+R_{\vartheta_a}x(a)$.
\end{proof}

\begin{proof}[Proof of Proposition~\ref{prop:three-counterexample}]
Choose a sufficiently small nonzero point on the branch \eqref{eq:nu-branch}, and write $c=c(\tau)$ and $x=x(\tau)$. Apply Lemma~\ref{lem:nu-paste} on $[0,T/3]$ with a deterministic initial state $a\in C$. For this $a$ and $c$, set $\widehat\zeta=\xi_{a,c}$. Both fixed points $0$ and $x$ give solutions with terminal condition $\widehat\zeta$. Their initial values differ because $x(a)\ne0$.

Returning to the original $Y$ coordinates, set $\zeta=\mathsf T^{-1}\widehat\zeta$ and choose $M$ such that $|\zeta_i|\le M$. Apply Theorem~\ref{thm:encoding} with horizon $T/3$, using the increments of the given Brownian motion on $[T/3,2T/3]$ and $[2T/3,T]$. The two solutions extend to distinct solutions with the same terminal condition $\xi=\Theta(\zeta)$ at time $T$. Each of the three vertices is attained by $u(\xi)$ with positive probability because the range of $u(\zeta)$ is contained in a compact subset of the interior of the encoding triangle, where the probabilities \eqref{eq:pfirst}--\eqref{eq:pright} lie strictly between zero and one. Both extended solutions belong to $\Sp^\infty\times\BMO$ by Theorem~\ref{thm:encoding}.
\end{proof}

Together with the triangular case and the existence result of Section~\ref{sec:existence}, Proposition~\ref{prop:three-counterexample} completes the proof of Theorem~\ref{thm:universal}.

\subsection{Nonuniqueness beyond three-valued terminal conditions}\label{subsec:terminal-laws}

For each $\theta=(\alpha_1,\alpha_2)\in\R^2$ with $(\alpha_1-1)(\alpha_2-1)\ne1$ and every integer $n\ge4$, we extend the three-valued construction to a terminal condition taking exactly $n$ values and admitting multiple solutions. For the same $\theta$ we obtain multiple solutions for a terminal condition with a purely atomic law having infinitely many atoms and for one whose law has a density in $C_c^\infty(\R^2)$.
Take a three-valued terminal condition $\eta$ with multiple solutions on $[0,s]$, where $0<s<T$. We construct a new terminal condition with a larger range so that restriction to $[0,s]$ gives a bijection from its solution set onto $\Sol_{\theta,[0,s]}(\eta)$. On the added interval we change one coordinate using a bounded function of a fresh Brownian increment. The unchanged coordinate is $\FF_s$-measurable, so its Brownian integrand vanishes there and the continuation reduces to a scalar equation. We center the function of the fresh increment by subtracting the time-$s$ value of the scalar BSDE with that function as terminal condition. The continuation then satisfies $Y_s=\eta$. Applying the construction successively to the two coordinates gives a terminal condition whose law under $\PP$ has a smooth density.

\begin{lemma}[One coordinate known in advance]\label{lem:axis-restriction}
Fix $\theta\in\R^2$ and $0<s<t\le T$. Choose $i\in\{1,2\}$ and put $j=3-i$ and $\lambda_i=2-\alpha_i$.
For a bounded scalar random variable $H\in L^\infty(\FF_t)$, write
\[
 C^{\lambda}_r(H)=
 \begin{cases}
 \lambda^{-1}\log\E[e^{\lambda H}\mid\FF_r],&\lambda\ne0,\\
 \E[H\mid\FF_r],&\lambda=0,
 \end{cases}
 \qquad s\le r\le t.
\]
Let $\zeta\in L^\infty(\FF_t;\R^2)$ and suppose that $\zeta_j$ is $\FF_s$-measurable.
Restriction to $[0,s]$ gives a bijection from $\Sol_{\theta,[0,t]}(\zeta)$ to $\Sol_{\theta,[0,s]}(\eta)$, where
\[
 \eta_j=\zeta_j,\qquad \eta_i=C^{\lambda_i}_s(\zeta_i).
\]
\end{lemma}
\begin{proof}
For any bounded solution on $[s,t]$, the integrand $Z$ belongs to $\BMO$ by Proposition~\ref{prop:bounds}.
With $Z$ extended by zero outside $(s,t]$, \eqref{eq:conditionalU} and the $\FF_s$-measurability of $\zeta_j$ give
\[
 \phi_{\alpha_j}(Y_{j,r})
 =\E^{\PP^Z}\!\left[\phi_{\alpha_j}(\zeta_j)\mid\FF_r\right]
 =\phi_{\alpha_j}(\zeta_j),\qquad s\le r\le t.
\]
Since $\phi_{\alpha_j}$ is strictly increasing, $Y_j=\zeta_j$ and hence $Z_j=0$ on $[s,t]$.
It follows that $\dd Y_i=-(\lambda_i/2)Z_i^2\dt+Z_i\dd W$ and $Y_{i,t}=\zeta_i$.
For $\lambda_i\ne0$, the bounded martingale $e^{\lambda_iY_i}$ forces $Y_{i,r}=C^{\lambda_i}_r(\zeta_i)$.
For $\lambda_i=0$, the same identity follows by conditioning $Y_i$ itself.
Conversely, Brownian martingale representation determines the scalar integrand $Z_i$ in either formula.
It is BMO because the represented martingale is bounded and, in the exponential case, also bounded away from zero.
This uniquely determined solution on $[s,t]$ starts at $\eta$, so it can be appended to every bounded solution on $[0,s]$.
The states match at $s$, and the two conditional energy bounds add under concatenation.
\end{proof}

\begin{proof}[Proof of Theorem~\ref{thm:terminal-laws}]
Proposition~\ref{prop:binary} gives uniqueness for every $\theta\in\R^2$ and every terminal condition taking at most two values.
Now fix $\theta$ satisfying $(\alpha_1-1)(\alpha_2-1)\ne1$ and $0<T<\infty$. Set $s=T/3$, $t=2T/3$ and $\lambda_i=2-\alpha_i$ for $i=1,2$.
By Proposition~\ref{prop:three-counterexample}, there exists $\eta\in L^\infty(\FF_s;\R^2)$ taking exactly three values $v^1,v^2,v^3$, each with positive $\PP$-probability, for which there are at least two distinct solutions in $\Sp^\infty\times\BMO$ on $[0,s]$.
Write $A_k=\{\eta=v^k\}$ and denote the coordinate vectors by $e_1,e_2$.

For the discrete examples, let $H$ be a bounded function of $(W_T-W_s)/\sqrt{T-s}$, and put $m=C^{\lambda_1}_s(H)$, which is deterministic by independence.
For
\[
 \xi=\eta+\ind_{A_1}(H-m)e_1
\]
we have $\xi_2=\eta_2$ and $C^{\lambda_1}_s(\xi_1)=\eta_1$ on each $A_k$.
By Lemma~\ref{lem:axis-restriction}, restriction to $[0,s]$ is a bijection from $\Sol_{\theta,[0,T]}(\xi)$ to $\Sol_{\theta,[0,s]}(\eta)$.
For a prescribed $n\ge3$, choose $H$ to take $n-2$ distinct values of positive probability in an interval of length strictly less than $\min_{k=2,3}|v^1-v^k|$.
Such a choice is obtained by partitioning the standard normal distribution into $n-2$ intervals of positive mass.
For every $\lambda_1\in\R$, the scalar value $m$ lies between the extreme values of $H$.
Thus the $n-2$ distinct values produced on $A_1$ differ from both $v^2$ and $v^3$.
Independence gives each positive probability, so $\xi$ has exactly $n$ values.
The same construction with a countable partition of positive masses gives $\xi^{\mathrm{at}}$ with a purely atomic law and infinitely many atoms.
The values of $H$ may be chosen as a convergent sequence and its limit in the same short interval, assigning positive mass also to the limit.

For the smooth density, choose a nonnegative $f\in C_c^\infty(\R)$ with integral one and distribution function $F$.
Let $N_1,N_2$ be the standardized Brownian increments on $[s,t]$ and $[t,T]$, respectively, and set $H_i=F^{-1}(\Phi(N_i))$, using the generalized inverse and the standard normal distribution function $\Phi$.
Then $H_1,H_2$ have density $f$, are independent of each other, and are independent of $\FF_s$.
Set $m_1=C_s^{\lambda_1}(H_1)$ and $m_2=C_t^{\lambda_2}(H_2)$, both deterministic, and define
\[
 \zeta=\eta+(H_1-m_1)e_1,\qquad
 \xi^{\mathrm{ac}}=\zeta+(H_2-m_2)e_2.
\]
Apply Lemma~\ref{lem:axis-restriction} first on $[t,T]$, where $(\xi^{\mathrm{ac}})_1=\zeta_1$ is $\FF_t$-measurable, and then on $[s,t]$, where $\zeta_2=\eta_2$ is $\FF_s$-measurable.
The two restrictions compose to a bijection from $\Sol_{\theta,[0,T]}(\xi^{\mathrm{ac}})$ to $\Sol_{\theta,[0,s]}(\eta)$.
Under $\PP$, the law of $\xi^{\mathrm{ac}}$ has the density
\[
 x\longmapsto\sum_{k=1}^3\PP(A_k)
 f(x_1-v^k_1+m_1) f(x_2-v^k_2+m_2),
\]
which belongs to $C_c^\infty(\R^2)$.
In each construction the new terminal condition is bounded and uses only increments of the original Brownian motion before the prescribed time $T$.
\end{proof}

Let the terminal condition in Proposition~\ref{prop:three-counterexample} take the values $v^1,v^2,v^3$ with probabilities $p_1,p_2,p_3$. On the same Brownian filtration, this law can also be realized by a terminal condition with a unique solution. Put $s=T/2$ and $p=p_2/(p_1+p_2)$. Choose an event $E$ determined by $W_s$ and an event $F$ determined by $W_T-W_s$ with $\PP(E)=p_3$ and $\PP(F)=p$. By independence,
\[
 \widetilde\xi=v^3\ind_E+\big(v^1\ind_{F^c}+v^2\ind_F\big)\ind_{E^c}
\]
has the prescribed law. On $[s,T]$ the unique continuation is constant on $E$ and is the binary solution of Lemma~\ref{lem:segment} on $E^c$. Since $\PP(F\mid\FF_s)=p$, its initial value is $Y_s=v^3\ind_E+\mathfrak B_p(v^1,v^2)\ind_{E^c}$. This takes at most two values, so Proposition~\ref{prop:binary} gives a unique solution on $[0,s]$. Together the two restrictions determine the unique solution in $\Sp^\infty\times\BMO$ on $[0,T]$.

\subsection{Cardinality of the solution set: one or \texorpdfstring{$2^{\aleph_0}$}{2\textasciicircum\unichar{"2135}\unichar{"2080}}}\label{subsec:cardinality}
For a fixed $\xi$, two distinct solutions already yield a continuum of solutions. Their integrands differ in $\Hp^2$ on some deterministic interval $[s,T]$. Each event in $\FF_s$ selects one continuation on the event and the other on its complement. The resulting time-$s$ value is a bounded random vector, which universal existence allows us to use as the terminal condition on $[0,s]$. All the resulting solutions have terminal condition $\xi$.

\begin{corollary}[Cardinality of the solution set]\label{cor:cardinality}
For every $\theta\in\R^2$, $0<T<\infty$, and $\xi\in L^\infty(\FF_T;\R^2)$, the set $\Sol_\theta(\xi)$ of solutions in $\Sp^\infty\times\BMO$ is either a singleton or has cardinality $2^{\aleph_0}$.
\end{corollary}
\begin{proof}
Theorem~\ref{thm:universal} gives at least one solution.
Suppose that $(Y^a,Z^a)$ and $(Y^b,Z^b)$ are distinct solutions for this $\xi$.
Their integrands differ in $\Hp^2$, since the BSDE determines $Y$ from $Z$ and $\xi$.
Consequently there is a deterministic $s\in(0,T)$ for which
\[
 D=\int_s^T|Z^a_t-Z^b_t|^2\dt,\qquad 0<\E D<\infty.
\]
On $\R$ define the probability measure $\nu(B)=\E[D\ind_{\{W_s\in B\}}]/\E D$.
It is absolutely continuous with respect to the Gaussian law of $W_s$, so its distribution function is continuous.
For each $r\in(0,1)$ choose its $r$-quantile $x_r$ and put $A_r=\{W_s\le x_r\}\in\FF_s$.
For $0<r<r'<1$ these events are nested and $\E[D\ind_{A_{r'}\setminus A_r}]=(r'-r)\E D$.

On $(s,T]$ take solution $a$ on $A_r$ and solution $b$ on $A_r^c$.
Their selected value at $s$ is $\eta_r=Y^a_s\ind_{A_r}+Y^b_s\ind_{A_r^c}$, a bounded $\FF_s$-measurable vector.
Theorem~\ref{thm:universal} on $[0,s]$ gives a bounded solution with terminal condition $\eta_r$.
Concatenation produces a bounded solution on $[0,T]$ with terminal condition $\xi$ and integrand $Z^{(r)}$.
For $r<r'$ the two pasted integrands satisfy
\[
 \E\int_s^T|Z^{(r')}_t-Z^{(r)}_t|^2\dt=(r'-r)\E D>0.
\]
This gives at least $2^{\aleph_0}$ distinct solutions for the original $\xi$.
For the reverse bound, the sigma field generated by the Brownian path is countably generated modulo null sets, as is the Borel sigma field of the time interval.
Their product $L^2$ space is separable, and so is its subspace $\Hp^2$.
A separable metric space has cardinality at most $2^{\aleph_0}$, and the map $(Y,Z)\mapsto Z$ is injective for fixed $\xi$.
\end{proof}

\section{Generic uniqueness and stability}\label{sec:baire}\label{sec:core}

Fix $\theta\in\R^2$ and $M>0$. We first prove stability at any $\bar\xi\in\mathcal D_M^\theta$ by comparing a piecewise segment solution for $\bar\xi$ with every solution for a nearby terminal condition. This set of reference terminal conditions is dense in $X_M$. Indeed, for $\xi\in X_M$ apply Lemma~\ref{lem:tree-density} with $K=u_\theta([-M,M]^2)$ and $\eta=u_\theta(\xi)$. The approximating terminal states lie in $K$, where $v_\theta=u_\theta^{-1}$ is Lipschitz, so their inverse images converge to $\xi$ in $X_M$. We then obtain generic uniqueness and stability by estimating the diameter of the solution sets over small balls in $X_M$.

\subsection{Stability at piecewise segment solutions}\label{sec:stability}
For each solution $(Y^n,Z^n)$ being compared with $(\bar Y,\bar Z)$, we use the measure $Q_n=\PP^{Z^n}$ induced by $Z^n$. On each interval $[\sigma,\tau]$ of the reference partition, $\E^{Q_n}[u_\theta(\bar Y_\tau)\mid\FF_t]$ stays on the reference segment for $\sigma\le t\le\tau$. Its scalar coordinate on this segment allows us to propagate convergence from $\tau$ to the entire interval.

\begin{theorem}[Stability at piecewise segment solutions]\label{thm:corestability}
Let $(\bar Y,\bar Z)$ be a piecewise segment solution with terminal condition $\bar\xi\in X_M$. If $\xi_n\to\bar\xi$ in probability with $\xi_n\in X_M$, then every choice $(Y^n,Z^n)\in\Sol_\theta(\xi_n)$ satisfies
\begin{equation}\label{eq:coreconvergence}
 \norm{Y^n-\bar Y}_{\Sp^p}+\norm{Z^n-\bar Z}_{\Hp^p}\longrightarrow0
 \qquad(1\le p<\infty).
\end{equation}
Consequently, $\Sol_\theta(\bar\xi)=\{(\bar Y,\bar Z)\}$.
\end{theorem}

\begin{proof}
Apply the transformation \eqref{eq:coordinates} to obtain $(U^n,V^n)$ and $(\bar U,\bar V)$. By Proposition~\ref{prop:bounds}, $U^n$ and $\bar U$ lie in the compact rectangle $u_\theta([-M,M]^2)\subset D_\theta$, and $\norm{V^n}_{\Hp^p}$ is bounded uniformly in $n$ for each $1\le p<\infty$. Lemma~\ref{lem:density} gives the bounds \eqref{eq:densitymoments} for
\[
 b_n=Z_1^n+Z_2^n,\qquad D_n=\cE(b_n\cdot W)_T,
 \qquad \dd Q_n=D_n\dd\PP,\qquad Q_n=\PP^{Z^n}.
\]
We prove convergence on an interval $[\sigma,\tau]$ of the partition for $(\bar Y,\bar Z)$, under the assumption
\begin{equation}\label{eq:endpointconv}
 U_\tau^n\longrightarrow\bar U_\tau\quad\text{in probability}.
\end{equation}
Write the reference segment as $A+e[0,1]$, with $A,e$ measurable with respect to $\FF_\sigma$. Let $c\le|e|\le C$ be its deterministic length bounds.

\emph{An auxiliary martingale.}
On $[\sigma,\tau]$, define
\begin{equation}\label{eq:replacement}
 \widehat U_t^n=\E^{Q_n}[\bar U_\tau\mid\FF_t],
 \qquad \dd\widehat U_t^n=\widehat V_t^n\dd W_t^{Q_n}.
\end{equation}
The terminal difference $U_\tau^n-\bar U_\tau$ is uniformly bounded and converges in probability. Equation \eqref{eq:PtoQ} gives its convergence to zero in $L^p(Q_n)$ for every $1\le p<\infty$. For every $q>1$, Doob's maximal inequality and the Burkholder--Davis--Gundy inequalities give convergence of $U^n-\widehat U^n$ in $\Sp^q(Q_n)$ and of $V^n-\widehat V^n$ in $\Hp^q(Q_n)$ on $[\sigma,\tau]$. For a given $p\ge1$, take $q=p(1+\varepsilon)/\varepsilon$ in \eqref{eq:QtoP}. This gives
\begin{equation}\label{eq:replacementconvergence}
 \norm{U^n-\widehat U^n}_{\Sp^p[\sigma,\tau]}
 +\norm{V^n-\widehat V^n}_{\Hp^p[\sigma,\tau]}
 \longrightarrow0\qquad(1\le p<\infty).
\end{equation}
These martingale estimates apply to the processes stopped at $\tau$ and restricted to $[\sigma,\tau]$. In particular, they have constants independent of the stopping times. They also give uniform bounds on $\norm{\widehat V^n}_{\Hp^p}$ for every $1\le p<\infty$.

\emph{Reduction to the reference segment.}
Since conditional expectations preserve the $\FF_\sigma$-measurable segment, \eqref{eq:replacement} gives
\[
 \widehat U^n=A+er_n,\qquad
 \widehat V^n=eh_n,\qquad 0\le r_n\le1.
\]
Similarly, write $\bar U=A+e\bar r$ and $\bar V=e\bar h$. Set
\[
 \varepsilon_n=b_n+a_\theta(\widehat U^n)\cdot\widehat V^n.
\]
Since $b_n+a_\theta(U^n)\cdot V^n=0$, the error is
\[
 \varepsilon_n
 =[a_\theta(\widehat U^n)-a_\theta(U^n)]\cdot V^n
   +a_\theta(\widehat U^n)\cdot(\widehat V^n-V^n).
\]
The coefficient $a_\theta$ is bounded and Lipschitz on the common rectangle. H\"older's inequality, \eqref{eq:replacementconvergence}, and the uniform bounds on $\norm{V^n}_{\Hp^{2p}}$ imply
\[
 \norm{\varepsilon_n}_{\Hp^p[\sigma,\tau]}\longrightarrow0
 \qquad(1\le p<\infty).
\]
In particular, the $\Hp^p[\sigma,\tau]$ norm of $[a_\theta(\widehat U^n)-a_\theta(U^n)]\cdot V^n$ is bounded by a constant times
\[
 \norm{\widehat U^n-U^n}_{\Sp^{2p}[\sigma,\tau]}
 \norm{V^n}_{\Hp^{2p}[\sigma,\tau]}.
\]

Let $\gamma=\gamma_{A,e}$ be \eqref{eq:gamma} and put
\[
 P_n=\gamma(r_n),\qquad B_n=\gamma'(r_n)h_n,
 \qquad \bar P=\gamma(\bar r),\qquad
 \bar B=\gamma'(\bar r)\bar h.
\]
By \eqref{eq:replacement}, in the original probability measure
\[
 \dd\widehat U^n
 =\widehat V^n\big(a_\theta(\widehat U^n)\cdot\widehat V^n\big)\dt
 -\varepsilon_n\widehat V^n\dt+\widehat V^n\dd W.
\]
It\^o's formula and $\gamma''(r)=-2\gamma'(r)a_\theta(A+er)\cdot e$ therefore give
\begin{equation}\label{eq:scalarerror}
 \dd P_n=-\varepsilon_nB_n\dt+B_n\dd W,
 \qquad \dd\bar P=\bar B\dd W,
 \qquad P_{n,\tau}=\bar P_\tau.
\end{equation}
The lower bound on $|e|$ and the bound on $\gamma'$ give a uniform bound on $\norm{B_n}_{\Hp^p}$ for every $1\le p<\infty$. Consequently,
\begin{equation}\label{eq:forcing0}
 \left\|\int_\sigma^\tau|\varepsilon_nB_n|\dt\right\|_{L^p}
 \le\norm{\varepsilon_n}_{\Hp^{2p}[\sigma,\tau]}
      \norm{B_n}_{\Hp^{2p}[\sigma,\tau]}
 \longrightarrow0.
\end{equation}
Conditional expectation in \eqref{eq:scalarerror} and Doob's inequality give $P_n\to\bar P$ in $\Sp^p$ for $p>1$, and hence also for $p=1$. Applying It\^o's formula to $(P_n-\bar P)^2$ gives
\[
 \E\int_\sigma^\tau|B_n-\bar B|^2\dt
 \le2\E\left[\sup_{[\sigma,\tau]}|P_n-\bar P|
                  \int_\sigma^\tau|\varepsilon_nB_n|\dt\right]
 \longrightarrow0.
\]
The process $\bar P$ is a bounded martingale, so $\bar B$ also belongs to $\Hp^q[\sigma,\tau]$ for every $1\le q<\infty$. Together with the uniform bounds on $\norm{B_n}_{\Hp^q}$ for $q>p$, this gives uniform integrability and hence convergence in $\Hp^p$ for every $1\le p<\infty$. The inverse $\gamma^{-1}$ is uniformly Lipschitz, and $1/\gamma'$ is uniformly Lipschitz on $[0,1]$. It follows that $\widehat U^n\to\bar U$ and $\widehat V^n\to\bar V$ in $\Sp^p$ and $\Hp^p$, respectively, for every $1\le p<\infty$. Combining this with \eqref{eq:replacementconvergence} establishes the same conclusion for $(U^n,V^n)$ on $[\sigma,\tau]$.

\emph{Backward induction.}
On the final reference interval, \eqref{eq:endpointconv} follows from $\xi_n\to\bar\xi$. Convergence at the left endpoint supplies the terminal convergence needed on the preceding interval. A finite backward induction proves convergence on $[0,T]$. To return to $Y=v_\theta(U)$ and $Z_i=V_i/(1-\alpha_iU_i)$, use that $v_\theta$ and the multipliers $u\mapsto(1-\alpha_i u_i)^{-1}$ are bounded and Lipschitz on the common rectangle. H\"older's inequality and the uniform moment bounds for $V^n$ give \eqref{eq:coreconvergence}. Applying this conclusion to a constant sequence of any bounded solution with terminal condition $\bar\xi$ proves uniqueness.
\end{proof}

\subsection{A local diameter criterion for uniqueness and stability}
Equip $\mathscr E=\Sp^2\times\Hp^2$ with the sum norm. For a ball around $\xi$ in $X_M$, take the union of $\Sol_\theta(\eta)$ over all terminal conditions $\eta$ in that ball. If its diameter in $\mathscr E$ tends to zero as the radius decreases to zero, then $\Sol_\theta(\xi)$ is a singleton and every choice of solutions for $\xi_n\to\xi$ converges to its unique element.
\begin{lemma}\label{lem:abstract}
Let $X$ and $E$ be metric spaces, and let $F$ be a map assigning a nonempty subset of $E$ to each point of $X$. Suppose a dense set $D\subset X$ consists of points $x$ at which $F(x)=\{u_x\}$ and $x_n\to x$, $u_n\in F(x_n)$ imply $u_n\to u_x$. With $F(A)=\bigcup_{y\in A}F(y)$, set
\begin{equation}\label{eq:abstractopen}
 O_j=\{x\in X:\ \diam_E F(B_X(x,r))<1/j\text{ for some }r>0\}.
\end{equation}
Then $G=\bigcap_{j\ge1}O_j$ is a dense $G_\delta$ set containing $D$. At every point of $G$ the same uniqueness and convergence properties hold.
\end{lemma}
\begin{proof}
Each $O_j$ is open by inclusion of balls. For $x\in D$, the convergence hypothesis places all values on a sufficiently small ball within $1/(3j)$ of $u_x$. Thus $D\subset O_j$ for every $j$, so $G$ is a $G_\delta$ set containing the dense set $D$. If $x\in G$, the nonempty set $F(x)$ has diameter zero. For $x_n\to x$, both $u_n\in F(x_n)$ and the unique element of $F(x)$ eventually belong to a set of diameter less than $1/j$ for each $j$.
\end{proof}

\begin{proof}[Proof of Theorem~\ref{thm:main}]
Apply Lemma~\ref{lem:abstract} with
\[
 X=X_M,\qquad E=\mathscr E=\Sp^2\times\Hp^2,
 \qquad F(\xi)=\Sol_\theta(\xi),\qquad D=\mathcal D_M^\theta.
\]
Universal existence gives $F(\xi)\ne\varnothing$ for every $\xi\in X_M$. Lemma~\ref{lem:tree-density} gives density of $D$, and Theorem~\ref{thm:corestability} gives stability at every $\bar\xi\in D$ for arbitrary choices of solutions along sequences in $X_M$ converging to $\bar\xi$. With these choices, the open sets in \eqref{eq:abstractopen} and their intersection are
\begin{equation}\label{eq:canonicalG}
 \begin{aligned}
 O^\theta_{M,j}
 &=\left\{\xi\in X_M:\exists r>0,\quad
 \diam_{\mathscr E}\bigcup_{\eta\in B_{X_M}(\xi,r)}\Sol_\theta(\eta)<1/j\right\},\\
 \mathcal G_M^\theta&=\bigcap_{j\ge1}O^\theta_{M,j}.
 \end{aligned}
\end{equation}
Fix $\xi\in\mathcal G_M^\theta$, let $(Y,Z)$ be its unique solution, and take any $\xi_n\to\xi$ in $X_M$ and $(Y^n,Z^n)\in\Sol_\theta(\xi_n)$. The lemma gives $(Y^n,Z^n)\to(Y,Z)$ in $\mathscr E$. For each $1\le p<\infty$, Proposition~\ref{prop:bounds} gives a uniform bound on $\|Z^n\|_{\Hp^{2p}}$, so $\big(\int_0^T|Z_t^n-Z_t|^2\dt\big)^{p/2}$ is uniformly integrable. This yields convergence in $\Hp^p$, and the coordinate bound $M$ yields convergence in $\Sp^p$. By Lemma~\ref{lem:segment}, every $\xi\in X_M$ taking at most two values admits a solution for which $u_\theta(Y)$ lies on a single segment. Thus each such $\xi$ belongs to $\mathcal G_M^\theta$.
\end{proof}

For $(\alpha_1-1)(\alpha_2-1)\ne1$, the terminal condition constructed in Proposition~\ref{prop:three-counterexample} belongs to $X_M$ for all sufficiently large $M$.
\begin{corollary}[Terminal conditions admitting multiple solutions]\label{cor:exceptional}
For every $\theta\in\R^2$ and $M>0$, the set $\{\xi\in X_M:\#\Sol_\theta(\xi)>1\}$ is either empty or dense and meagre in $X_M$.
\end{corollary}
\begin{proof}
Meagreness follows from Theorem~\ref{thm:main}. Suppose $\eta\in X_M$ has two distinct solutions $(Y^k,Z^k)$, $k=1,2$. Fix $0<s<T$, put $c=(T-s)/T$ and let $B_t=(W_{s+ct}-W_s)/\sqrt c$ for $0\le t\le T$. Evaluate the path representations of $\eta$ and both solutions on this same Brownian motion $B$, which is independent of $\FF_s$. The time change of Section~\ref{subsec:ex-stopping} gives a common terminal condition $\eta'=\eta(B)$ on $[s,T]$ and two integrands $Z^{\prime,k}_{s+ct}=c^{-1/2}Z^k_t(B)$, $k=1,2$.
For any $\xi\in X_M$, take a solution on $[0,T]$. Choose $A\in\FF_s$ of arbitrarily small positive probability. On $A$, use either of the two solutions for $\eta'$, and on $A^c$ use the restriction to $[s,T]$ of the chosen solution for $\xi$. For each choice the state at $s$ is a bounded $\FF_s$-measurable terminal condition for the equation on $[0,s]$, so universal existence gives a solution up to $s$. Pasting gives two solutions with terminal condition $\xi'=\ind_A\eta'+\ind_{A^c}\xi$. On $A$ their integrands are $Z^{\prime,1}$ and $Z^{\prime,2}$. Since these depend only on increments after $s$,
\[
 \E\!\left[\ind_A\int_s^T|Z_t^{\prime,1}-Z_t^{\prime,2}|^2\dt\right]
 =\PP(A)\,\E\int_s^T|Z_t^{\prime,1}-Z_t^{\prime,2}|^2\dt>0 .
\]
The strict inequality follows because a solution with terminal condition $\eta'$ is determined by its integrand. Finally, $\xi'\in X_M$ and $\E|\xi'-\xi|^2\le8M^2\PP(A)$. This proves density.
\end{proof}

Suppose $\xi\in X_M$ admits multiple solutions. For any prescribed $(Y,Z)\in\Sol_\theta(\xi)$ and $\varepsilon>0$, there are terminal conditions $\xi_n\in X_{M+\varepsilon}$ converging to $\xi$ in probability whose unique bounded solutions converge to $(Y,Z)$ in $\Sp^p\times\Hp^p$ for every $1\le p<\infty$.
Indeed, taking $K=u_\theta([-M-\varepsilon,M+\varepsilon]^2)$ in Lemma~\ref{lem:forward} gives piecewise segment solutions with the required convergence. Their terminal conditions lie in $X_{M+\varepsilon}$, and Theorem~\ref{thm:corestability} gives uniqueness for each of them.

\section{Nash equilibria in a two-player stochastic game}\label{sec:games}

Two players choose efforts and receive bounded terminal payoffs determined by the path of a common diffusion. When each effort enters the drift with a nonzero coefficient and both players pay strictly convex quadratic costs for their own efforts, the two scalar equations for their best responses give \eqref{eq:bsde} after a change of units. The critical curve then determines when every bounded terminal payoff vector admits a unique Nash equilibrium. For related verification results in risk-sensitive nonzero-sum games, see \cite[Section~5]{EKH2003}.
Section~\ref{subsec:game-extension} considers running costs $r_i(u_i-S_{ij}u_j)^2/2$, where $j=3-i$ and $S$ has zero diagonal. For $r_1,r_2>0$, $b_1b_2\ne0$ and $\gamma_1,\gamma_2\in\R$, $I-S$ is invertible if and only if every bounded terminal payoff vector admits a Nash equilibrium in BMO controls. If $I-S$ is singular, then for every $M>0$ equilibrium fails to exist on a dense $G_\delta$ subset of the space of $[-M,M]^2$-valued terminal payoff vectors equipped with the topology of convergence in probability. Section~\ref{subsec:game-horizon} extends the results to the infinite horizon.

\subsection{Weak formulation of the stochastic game}\label{subsec:game-model}

Fix $0<T<\infty$ and $x\in\R$. On canonical Wiener space let $W$ be the coordinate Brownian motion under the reference probability $\PP$, let $(\FF_t)_{0\le t\le T}$ be its augmented natural filtration and put $X=x+W$.
An admissible control for player $i\in\{1,2\}$ is a predictable effort process $u_i$ such that $u_i\cdot W\in\BMO(\PP)$. We identify controls up to $\dt\otimes\dd\PP$-almost everywhere equality and write $j=3-i$ for the opponent.
For drift coefficients $b=(b_1,b_2)^\top\in\R^2$, a control profile $u=(u_1,u_2)^\top$ induces the probability measure $\PP^{b,u}$ given by
\[
 \mu^u=b^\top u,\qquad
 \frac{\dd\PP^{b,u}}{\dd\PP}=\cE(\mu^u\cdot W)_T,\qquad
 W^{b,u}_t=W_t-\int_0^t\mu^u_s\ds .
\]
The stochastic exponential is uniformly integrable and $\PP^{b,u}$ is equivalent to $\PP$.
Under $\PP^{b,u}$ the process $W^{b,u}$ is Brownian and the state satisfies $\dd X_t=\mu^u_t\dt+\dd W^{b,u}_t$.
Thus $b_i u_i$ is player $i$'s contribution to the state drift.
Under this change of measure, $W^{b,u}$ has the martingale representation property in the same filtration. The Girsanov isomorphism also gives $a\cdot W\in\BMO(\PP)$ if and only if $a\cdot W^{b,u}\in\BMO(\PP^{b,u})$ \cite[Chapter~3, Theorem~3.6]{Kazamaki1994}.

Let $G_i:C([0,T],\R)\to\R$ be bounded and Borel. Player $i$ receives the terminal payoff $\eta_i=G_i(X_{[0,T]})$.
Since $X$ generates the filtration, every $\eta\in L^\infty(\FF_T;\R^2)$ has such a representation, common to all the measures $\PP^{b,u}$.
Write $\ell_i(u_t)$ for player $i$'s running cost.
For a stopping time $\tau\le T$, her net payoff from $\tau$ to $T$ and its conditional evaluation at risk sensitivity $\gamma_i\in\R$ are
\begin{equation}\label{eq:game-objective}
\begin{aligned}
 \Pi_{i,\tau}(u)&=\eta_i-\int_\tau^T\ell_i(u_t)\dt,\\
 J_{i,\tau}^{\gamma_i}(u)&=
 \begin{cases}
 \E^{\PP^{b,u}}[\Pi_{i,\tau}(u)\mid\FF_\tau],&\gamma_i=0,\\
 -\gamma_i^{-1}\log \E^{\PP^{b,u}}[e^{-\gamma_i\Pi_{i,\tau}(u)}\mid\FF_\tau],&\gamma_i\ne0 .
 \end{cases}
\end{aligned}
\end{equation}
Put $\Pi_i(u)=\Pi_{i,0}(u)$ and $J_i^{\gamma_i}(u)=J_{i,0}^{\gamma_i}(u)$.
Player $i$ maximizes her expected net payoff when $\gamma_i=0$ and its exponential certainty equivalent otherwise. Positive $\gamma_i$ gives risk aversion and negative $\gamma_i$ gives risk seeking.
For the nonnegative quadratic costs below the expected cost under each profile measure is finite by the BMO property. When $\gamma_i>0$ an infinite exponential expectation gives objective value $-\infty$.

For an admissible control $a$, let $u^{i,a}$ be the profile obtained by replacing $u_i$ with $a$ and fixing $u_j$ as a predictable functional of the state path.
A \emph{Nash equilibrium in BMO controls} is an admissible profile $u$ such that
\[
 J_i^{\gamma_i}(u)\ge J_i^{\gamma_i}(u^{i,a})
 \qquad\text{for each }i\in\{1,2\}\text{ and every admissible }a .
\]
Under a unilateral deviation the terminal payoff functional $G_i$ is evaluated under $\PP^{b,u^{i,a}}$ and the cost is computed from $u^{i,a}$.

\subsection{Conditions for existence and uniqueness of Nash equilibria}\label{subsec:game-responses}

Let $\ell_i(u_t)=r_i u_{i,t}^2/2$ and $\gamma_i\in\R$. We first take $r_1,r_2\ge0$.
The parameters $b_1=b_2=r_1=r_2=1$ and $\gamma_1=\gamma_2=0$ give the game of \cite[Section~6.4]{Jackson2023}.
We first solve the optimization problem against a fixed opponent for $r_i>0$. This gives a change of units relating equilibria to \eqref{eq:bsde} when both players can influence the drift and both costs are strictly convex.
The cases $b_i=0$ and $r_i=0$ have direct interpretations in the original game.

\subsubsection{Best response to the opponent's strategy}\label{subsubsec:game-best-response}

Fix $i$, let $v$ be the admissible control of player $j$, and write $Q_i^v$ for the measure corresponding to $u_i=0$, $u_j=v$. Thus
\[
 \frac{\dd Q_i^v}{\dd\PP}=\cE((b_jv)\cdot W)_T,\qquad
 W^{Q_i^v}_t=W_t-\int_0^t b_jv_s\ds .
\]
For any profile $u$ with $u_j=v$, player $i$'s continuation value at $\tau$ is the essential supremum of $J_{i,\tau}^{\gamma_i}(u^{i,a})$ over admissible $a$ agreeing with $u_i$ through $\tau$.
For these deviations the measures $\PP^{b,u^{i,a}}$ all agree with $\PP^{b,u}$ on $\FF_\tau$.

\begin{lemma}[Best response to a fixed opponent]\label{lem:game-best-response}
Suppose $r_i>0$ and put $\kappa_i=b_i^2/r_i$.
The scalar BSDE
\begin{equation}\label{eq:game-best-response}
 \dd V_{i,t}^v=-\frac{\kappa_i-\gamma_i}{2}(H_{i,t}^v)^2\dt
                 +H_{i,t}^v\dd W^{Q_i^v}_t,\qquad V_{i,T}^v=\eta_i
\end{equation}
has a unique solution $(V_i^v,H_i^v)\in\Sp^\infty\times\BMO$.
At every stopping time $\tau$, the continuation value is $V_{i,\tau}^v$.
Player $i$'s unique best response on $[0,T]$ is $(b_i/r_i)H_i^v$.
\end{lemma}

\begin{proof}
Write $Q=Q_i^v$. The solution of \eqref{eq:game-best-response} has first component
\[
 V_{i,t}^v=
 \begin{cases}
 (\kappa_i-\gamma_i)^{-1}\log \E^Q[e^{(\kappa_i-\gamma_i)\eta_i}\mid\FF_t],
       &\kappa_i\ne\gamma_i,\\
 \E^Q[\eta_i\mid\FF_t],&\kappa_i=\gamma_i.
 \end{cases}
\]
Indeed, in the first case the exponential of $(\kappa_i-\gamma_i)V_i^v$ is a bounded martingale bounded away from zero.
Martingale representation under $Q$ and It\^o's formula determine $H_i^v$ uniquely and give $H_i^v\in\BMO(Q)$.
In the second case we represent the bounded martingale $V_i^v$ directly. The BMO property transfers to $\PP$ in both cases.

Fix a deviation $a$, put $\widehat Q=\PP^{b,u^{i,a}}$, $W^{\widehat Q}=W^{b,u^{i,a}}$, and write $H=H_i^v$.
Under $\widehat Q$ we have $\dd W^Q_t=\dd W_t^{\widehat Q}+b_i a_t\dt$.
Integrating \eqref{eq:game-best-response} and completing the square gives
\begin{equation}\label{eq:game-square}
 \Pi_{i,\tau}(u^{i,a})
 =V_{i,\tau}^v+\int_\tau^T H_t\dd W_t^{\widehat Q}
    +\frac{\gamma_i}{2}\int_\tau^T H_t^2\dt
    -\frac{r_i}{2}\int_\tau^T\left(a_t-\frac{b_i}{r_i}H_t\right)^2\dt .
\end{equation}
For $\gamma_i=0$, taking conditional expectations gives $J_{i,\tau}^0(u^{i,a})\le V_{i,\tau}^v$.
For $\gamma_i\ne0$ the process $D=\cE(-\gamma_i H\cdot W^{\widehat Q})$ is a strictly positive uniformly integrable $\widehat Q$-martingale. By \eqref{eq:game-square},
\[
 e^{-\gamma_i\Pi_{i,\tau}(u^{i,a})}
  =e^{-\gamma_i V_{i,\tau}^v}\frac{D_T}{D_\tau}
    \exp\!\left(\frac{\gamma_i r_i}{2}\int_\tau^T
          \left(a_t-\frac{b_i}{r_i}H_t\right)^2\dt\right).
\]
The last factor is at least one when $\gamma_i>0$ and at most one when $\gamma_i<0$.
Since $\E^{\widehat Q}[D_T/D_\tau\mid\FF_\tau]=1$, in both cases $J_{i,\tau}^{\gamma_i}(u^{i,a})\le V_{i,\tau}^v$.
The control equal to $u_i$ through $\tau$ and to $(b_i/r_i)H$ afterwards attains equality. It is predictable, and its squared BMO norm under $\PP$ is bounded by the sum of the squared BMO norms of $u_i$ and $(b_i/r_i)H$.
At time zero equality requires the squared difference in \eqref{eq:game-square} to vanish $\dt\otimes\dd\widehat Q$-almost everywhere, also in the exponential cases since $D_T>0$. Equivalence of $\widehat Q$ and $\PP$ proves uniqueness.
\end{proof}

\subsubsection{A bijection between Nash equilibria and bounded BSDE solutions}\label{subsubsec:game-bijection}

Suppose $b_1b_2\ne0$ and $r_1,r_2>0$. An effort $u_i$ contributes $h_i=b_i u_i$ to the state drift at a cost of $h_i^2/(2\kappa_i)$ per unit time.
Multiplying player $i$'s payoff by $\kappa_i$ gives the normalized quantities
\begin{equation}\label{eq:game-units}
 K=\operatorname{diag}(\kappa_1,\kappa_2),\qquad
 \xi=K\eta,\qquad \widehat\gamma_i=\gamma_i/\kappa_i,\qquad
 \alpha_i=1+\widehat\gamma_i,\qquad \theta=(\alpha_1,\alpha_2).
\end{equation}
With efforts $h_i$, drift $h_1+h_2$ and running costs $h_i^2/2$, the profile measure is unchanged.
The new net payoff is $\kappa_i\Pi_i(u)$, so its evaluation with risk sensitivity $\widehat\gamma_i$ equals $\kappa_i J_i^{\gamma_i}(u)$.
Since $\kappa_i>0$ and $u_i\mapsto h_i$ is bijective, the change of units preserves the ordering of all unilateral deviations. In particular, $\gamma_i=0$ gives $\alpha_i=1$, and $\gamma_i=-\kappa_i$ gives $\alpha_i=0$.

\begin{proposition}[Equilibrium controls and continuation values]\label{prop:game-verification}
Let $b_1b_2\ne0$, $r_1,r_2>0$ and $\gamma_1,\gamma_2\in\R$.
For every bounded terminal payoff vector $\eta$, the map
\[
 (Y,Z)\longmapsto (Z_1/b_1,Z_2/b_2)^\top
\]
is a bijection from $\Sol_\theta(K\eta)$ to the Nash equilibria in BMO controls for $\eta$.
For the equilibrium $u$ corresponding to $(Y,Z)$, player $i$'s equilibrium continuation value is $V_i^u=Y_i/\kappa_i$.
\end{proposition}

\begin{proof}
Given an equilibrium $u$, write $V_i^u=V_i^{u_j}$ and $H_i=H_i^{u_j}$.
Lemma~\ref{lem:game-best-response} gives $u_i=(b_i/r_i)H_i$.
Thus for $Y_i=\kappa_i V_i^u$ and $Z_i=\kappa_iH_i$ we have $Z_i=b_i u_i$ and $Y_{i,T}=\xi_i$.
Returning \eqref{eq:game-best-response} for $v=u_j$ to $\PP$ gives the generator component
\[
 z_i z_j+\tfrac12(1-\widehat\gamma_i)z_i^2
 =z_i(z_1+z_2)-\tfrac12\alpha_i z_i^2 .
\]
Hence $(Y,Z)\in\Sol_\theta(K\eta)$.

Conversely, given such a solution, put $u_i=Z_i/b_i$. Under $Q_i^{u_j}$ the pair $(Y_i/\kappa_i,Z_i/\kappa_i)$ solves \eqref{eq:game-best-response} with terminal payoff $\eta_i$.
Uniqueness for the scalar equation identifies this pair with $(V_i^{u_j},H_i^{u_j})$.
Since $(b_i/r_i)(Z_i/\kappa_i)=u_i$, each control is the unique best response to the other. These constructions are inverse to one another.
\end{proof}

At equilibrium $\mu^u=Z_1+Z_2$, so $\PP^{b,u}$ is the measure $\PP^Z$ of \eqref{eq:Q}.
In particular, \eqref{eq:conditionalU} gives
\[
 \phi_{\alpha_i}(Y_{i,t})=\E^{\PP^{b,u}}[\phi_{\alpha_i}(\xi_i)\mid\FF_t],
\]
including $Y_{i,t}=\E^{\PP^{b,u}}[\xi_i\mid\FF_t]$ when $\alpha_i=0$.

\subsubsection{Equilibria on and off the critical curve}\label{subsec:game-classification}

\begin{theorem}[Universal existence and a criterion for universal uniqueness of Nash equilibria]\label{thm:game-classification}
Fix $0<T<\infty$, $b_1b_2\ne0$, $r_1,r_2>0$ and $\gamma_1,\gamma_2\in\R$.
Every bounded terminal payoff vector $\eta$ admits a Nash equilibrium in BMO controls.
The equilibrium is unique for every such $\eta$ if and only if $\gamma_1\gamma_2=\kappa_1\kappa_2$.
If $\gamma_1\gamma_2\ne\kappa_1\kappa_2$, some bounded terminal payoff vector admits a continuum of distinct Nash equilibria.
\end{theorem}

\begin{proof}
By \eqref{eq:game-units}, the critical curve $(\alpha_1-1)(\alpha_2-1)=1$ has equation $\gamma_1\gamma_2=\kappa_1\kappa_2$ in the game parameters.
The map $\eta\mapsto K\eta$ is a bijection on $L^\infty(\FF_T;\R^2)$.
The conclusions follow from Theorem~\ref{thm:universal} and Corollary~\ref{cor:cardinality} by the bijection in Proposition~\ref{prop:game-verification}.
\end{proof}

In the risk-neutral game of \cite[Section~6.4]{Jackson2023} we have $\alpha_1=\alpha_2=1$.
Every bounded terminal payoff vector admits an equilibrium, and on every finite positive horizon some three-valued terminal payoff vector admits a continuum of equilibria.
More generally, the multiplicity examples of Theorem~\ref{thm:terminal-laws} and the stability result of Theorem~\ref{thm:main} have the following counterparts for the game.

\begin{corollary}[Multiplicity and stability with respect to terminal payoffs]\label{cor:game-payoffs}
Fix the parameters of Theorem~\ref{thm:game-classification}.
\begin{enumerate}
\item For each bounded terminal payoff vector $\eta$, the set of Nash equilibria in BMO controls has cardinality one or the continuum. If $\eta$ takes at most two values in $\R^2$, the equilibrium is unique.
Suppose $\gamma_1\gamma_2\ne\kappa_1\kappa_2$. For each integer $n\ge3$ there is an $n$-valued terminal payoff vector admitting a continuum of equilibria, with all $n$ values attained with positive probability.
There are also bounded terminal payoff vectors $\eta^{\mathrm{at}}$ and $\eta^{\mathrm{ac}}$, each admitting a continuum of equilibria, such that under $\PP$ the law of $\eta^{\mathrm{at}}$ is purely atomic with infinitely many atoms and the law of $\eta^{\mathrm{ac}}$ has a density in $C_c^\infty(\R^2)$.
\item For $M>0$, equip $K^{-1}X_M=\{\eta:|\kappa_i\eta_i|\le M\ \text{a.s.},\ i=1,2\}$ with the topology of convergence in $\PP$-probability.
Every $\eta$ in the dense $G_\delta$ subset $K^{-1}\mathcal G_M^\theta$ admits a unique equilibrium $u$.
If $\eta_n\to\eta$ in $K^{-1}X_M$ with $\eta\in K^{-1}\mathcal G_M^\theta$, let $u^n$ be any equilibrium for $\eta_n$ and $V^n$ its continuation value process. Then
\[
 (V^n,u^n)\longrightarrow(V^u,u)
 \quad\text{in }\Sp^p\times\Hp^p\text{ under }\PP,\qquad 1\le p<\infty .
\]
If $K^{-1}X_M$ contains a terminal payoff vector with multiple equilibria, the set of all such vectors is dense and meagre in $K^{-1}X_M$.
\end{enumerate}
\end{corollary}

\begin{proof}
The invertible linear map $K$ preserves the number of atoms and takes a law with a density in $C_c^\infty(\R^2)$ to another such law.
Theorem~\ref{thm:terminal-laws} and Corollary~\ref{cor:cardinality} thus give~\textup{(i)} through the equilibrium bijection.
For~\textup{(ii)}, $K$ maps $K^{-1}X_M$ homeomorphically onto $X_M$. Apply Theorem~\ref{thm:main} and Corollary~\ref{cor:exceptional}, using $V_i^u=Y_i/\kappa_i$ and $u_i=Z_i/b_i$.
\end{proof}

For each $\eta\in K^{-1}X_M$, Theorem~\ref{thm:encoding} also identifies its equilibria with those for the three-valued terminal payoff vector $K^{-1}\Theta(K\eta)$ on a Brownian extension to $[0,3T]$.
The three payoff vectors in its range depend only on the parameters and $M$. Restricting an equilibrium control to $[0,T]$ gives the bijection.
Uniqueness also holds if one player's terminal payoff $\eta_i$ is constant and $\eta_j$ is arbitrary. In that case $H_i^v=0$ for every $v$, so $u_i=0$ and player $j$'s unique best response to zero gives the unique equilibrium.

\subsubsection{Universal uniqueness when \texorpdfstring{$r_1,r_2>0$ and $b_1b_2=0$}{both cost coefficients are positive and a drift coefficient vanishes}}\label{subsubsec:game-zero-drift}

Suppose $r_1,r_2>0$ and $b_i=0$. Player $i$'s control affects only her running cost, and Lemma~\ref{lem:game-best-response} gives zero as her unique best response to every opponent.
Consequently, for every bounded $\eta$ the unique Nash equilibrium is given by
\[
 u_i=0,\qquad u_j=(b_j/r_j)H_j^0 ,
\]
where $H_j^0$ is obtained from \eqref{eq:game-best-response} under $Q_j^0=\PP$. In particular, if $b_1=b_2=0$ the equilibrium is $(0,0)$.

\subsubsection{Existence and multiplicity when \texorpdfstring{$r_1r_2=0$}{a cost coefficient vanishes}}\label{subsubsec:game-zero-cost}

If player $i$ pays no running cost, the role of her effort depends on whether it can influence the state drift. When $b_i\ne0$, she can choose an equivalent profile measure that concentrates arbitrarily closely on the largest realizations of her terminal payoff.

\Needspace{6\baselineskip}
\begin{proposition}[Equilibria with a zero cost coefficient]\label{prop:game-zero-cost}
Suppose $r_1,r_2\ge0$ and $r_1r_2=0$.
A bounded terminal payoff vector $\eta$ admits a Nash equilibrium in BMO controls if and only if $\eta_i$ is almost surely constant for every $i$ with $r_i=0$ and $b_i\ne0$.
When this condition holds, the controls of all players with $r_i=0$ can be chosen arbitrarily among the admissible controls.
If $r_j>0$ for the other player, her control is the unique best response $(b_j/r_j)H_j^{u_i}$.
This gives all equilibria, and their cardinality is the continuum.
\end{proposition}

\begin{proof}
Fix $i$ with $r_i=0$ and fix the opponent's effort $v$.
If $b_i=0$, all admissible efforts give the same measure and the same objective. If $\eta_i$ is constant, all admissible efforts likewise attain the same objective.

Suppose $b_i\ne0$, and write $m_i=\operatorname{ess\,sup}_{\PP}\eta_i$.
Against $v$ every admissible effort gives objective at most $m_i$. For nonconstant $\eta_i$ the inequality is strict because every profile measure is equivalent to $\PP$ and the expectation and exponential certainty equivalent are both strictly monotone on bounded random variables.
We show that $m_i$ is the supremum. Put $Q=Q_i^v$ and, for $n\ge1$, define
\[
 \frac{\dd Q_n}{\dd Q}=\frac{e^{n\eta_i}}{\E^Q e^{n\eta_i}}.
\]
For each $n$ the density process $D^n_t=\E^Q[\dd Q_n/\dd Q\mid\FF_t]$ is bounded above and away from zero.
Martingale representation gives $\dd D^n_t=D^n_t\vartheta^n_t\dd W^Q_t$, with $\vartheta^n\in\BMO(Q)$ since $D^n$ is bounded and $1/D^n$ is bounded.
By the Girsanov isomorphism $\vartheta^n\in\BMO(\PP)$.
Composing the two changes of measure gives $\dd Q_n/\dd\PP=\cE((b_jv+\vartheta^n)\cdot W)_T$, so player $i$ can realize $Q_n$ by taking effort $\vartheta^n/b_i$.
For every $\epsilon>0$,
\[
 Q_n(\eta_i\le m_i-\epsilon)
 \le \frac{e^{-n\epsilon/2}}{Q(\eta_i>m_i-\epsilon/2)}
 \longrightarrow0 .
\]
Boundedness of $\eta_i$ now gives convergence of the corresponding expected net payoffs and exponential certainty equivalents to $m_i$. Thus a player with $r_i=0$ and $b_i\ne0$ has a best response exactly when her terminal payoff is almost surely constant.
Combining this conclusion for each zero cost coefficient with Lemma~\ref{lem:game-best-response} for any remaining player gives the asserted parametrization.
Distinct constant efforts of a player with $r_i=0$ already give a continuum of equilibria. For the reverse bound, all admissible profiles belong to the separable space $L^2(\dt\otimes\dd\PP;\R^2)$.
\end{proof}

Finally, suppose $r_i<0$ for some $i$, so that player $i$ receives a subsidy for effort. In this case we require both players' objective values at equilibrium to be finite. For $\gamma_i<0$, the exponential expectation in \eqref{eq:game-objective} can otherwise give $J_i^{\gamma_i}(u)=+\infty$. Against any opponent, the constant effort $a_t=n$ gives player $i$ a bounded net payoff at least $-\norm{\eta_i}_\infty+|r_i|n^2T/2$. Her objective under \eqref{eq:game-objective} is finite and satisfies the same lower bound for every $\gamma_i\in\R$. Thus every profile with finite objective values admits a profitable deviation. Table~\ref{tab:game-individual} collects the conclusions for all cost and drift coefficients.

\begin{table}[htbp]
\caption{Nash equilibria for running costs $r_i u_i^2/2$ and bounded terminal payoff vectors.}\label{tab:game-individual}
\centering\small
\begin{tabular}{@{}p{.45\linewidth}@{\hspace{.04\linewidth}}p{.49\linewidth}@{}}
\hline
Parameters & Equilibria in BMO controls\\ \hline
$r_1,r_2>0$, $b_1b_2\ne0$,
$\gamma_1\gamma_2=\kappa_1\kappa_2$
& A unique equilibrium for every $\eta$.\\
$r_1,r_2>0$, $b_1b_2\ne0$,
$\gamma_1\gamma_2\ne\kappa_1\kappa_2$
& An equilibrium for every $\eta$. For some $\eta$ there is a continuum of equilibria.\\
$r_1,r_2>0$, $b_1b_2=0$ & A unique equilibrium for every $\eta$.\\
$r_1,r_2\ge0$, $r_1r_2=0$
& A continuum if $\eta_i$ is constant a.s.\ for each $i$ with $r_i=0$, $b_i\ne0$. Otherwise none.\\
$r_i<0$ for some $i$ & No equilibrium with both objective values finite.\\
\hline
\end{tabular}
\end{table}

\subsection{Extension to costs depending on both players' efforts}\label{subsec:game-extension}

Let $r_i>0$, $b_1b_2\ne0$ and $\gamma_i\in\R$.
For a constant matrix $S=(S_{ij})$ with $S_{11}=S_{22}=0$, take the running cost
\[
 \ell_i(u)=\tfrac12 r_i(u_i-S_{ij}u_j)^2,\qquad j=3-i .
\]
Given the opponent's effort $u_j$, player $i$ incurs a cost proportional to the squared distance of her effort from $S_{ij}u_j$.
Her net payoff is $\Pi_i(u)=\eta_i-\int_0^T\ell_i(u_t)\dt$ and her objective is \eqref{eq:game-objective} under $\PP^{b,u}$.
The choice $S=0$ gives the preceding game with these parameters.

Fix the opponent's effort $v$.
Let $Q_i^{S,v}$ be the measure induced by $u_i=S_{ij}v$, $u_j=v$, and write $W^{Q_i^{S,v}}$ for its Brownian motion.
Then
\[
 \frac{\dd Q_i^{S,v}}{\dd\PP}
 =\cE\big(((b_iS_{ij}+b_j)v)\cdot W\big)_T,\qquad Q_i^{0,v}=Q_i^v .
\]
Relative to this profile, an effort $a$ adds $b_i(a-S_{ij}v)$ to the state drift and incurs running cost $r_i(a-S_{ij}v)^2/2$.
The map $a\mapsto a-S_{ij}v$ is a bijection on BMO controls.
Write $V_i^{S,v}$ for the continuation value against $v$ and $H_i^{S,v}$ for its Brownian integrand under $Q_i^{S,v}$. With $S$ fixed throughout this subsection, we write $(V_i^v,H_i^v)=(V_i^{S,v},H_i^{S,v})$. Applying the verification in Lemma~\ref{lem:game-best-response} under $Q_i^{S,v}$ gives the value equation and the unique best response
\begin{equation}\label{eq:tracking-response}
 \begin{gathered}
 \dd V_i^v=-\tfrac12(\kappa_i-\gamma_i)(H_i^v)^2\dt+H_i^v\dd W^{Q_i^{S,v}},
 \qquad V_{i,T}^v=\eta_i,\\
 u_i=S_{ij}v+p_iH_i^v,\qquad p_i=b_i/r_i,\qquad \kappa_i=b_i^2/r_i .
 \end{gathered}
\end{equation}

At a Nash equilibrium put $H_i=H_i^{u_j}$, $V_i=V_i^{u_j}$ and $Q=\PP^{b,u}$.
The difference between the state drifts under $Q$ and under $Q_i^{S,u_j}$ is $b_i(u_i-S_{ij}u_j)=\kappa_i H_i$.
Consequently the two best-response conditions and the two value equations take the form
\begin{equation}\label{eq:tracking-equilibrium}
 (I-S)u=\operatorname{diag}(p_1,p_2)H,\qquad
 \dd V_i=\tfrac12d_iH_i^2\dt+H_i\dd W^Q,\qquad
 V_{i,T}=\eta_i,\qquad d_i=\kappa_i+\gamma_i .
\end{equation}
Conversely, if an admissible $u$ and $(V,H)\in\Sp^\infty\times\BMO$ satisfy \eqref{eq:tracking-equilibrium} with $Q=\PP^{b,u}$, changing back to $Q_i^{S,u_j}$ gives \eqref{eq:tracking-response} for each player.
Thus $u$ is a Nash equilibrium and $V=V^u$.
The determinant $\det(I-S)=1-S_{12}S_{21}$ determines whether \eqref{eq:tracking-equilibrium} recovers the efforts from the integrands.

\subsubsection{Invertibility of \texorpdfstring{$I-S$}{I-S} and universal existence}\label{subsubsec:game-invertible}

Suppose $I-S$ is invertible and put
\[
 L=(I-S)^{-1}\operatorname{diag}(p_1,p_2),\qquad c=L^\top b .
\]
Then $u=LH$ and the state drift is $\mu^u=c^\top H$.
Under $\PP$, the value equations \eqref{eq:tracking-equilibrium} therefore form a BSDE with terminal condition $\eta$ and generator
\begin{equation}\label{eq:tracking-generator}
 \widetilde q_i(w)=w_i(c^\top w)-\tfrac12d_i w_i^2,\qquad w\in\R^2 .
\end{equation}
The map $(V,H)\mapsto LH$ is a bijection from the solutions in $\Sp^\infty\times\BMO$ of this BSDE to the Nash equilibria in BMO controls.
Indeed, $u$ determines $H$ since $L$ is invertible, and \eqref{eq:tracking-equilibrium} with the prescribed terminal condition then determines $V$.

\begin{proposition}[Existence and the criterion for universal uniqueness]\label{prop:game-tracking-regular}
If $I-S$ is invertible, every bounded terminal payoff vector admits a Nash equilibrium in BMO controls.
The equilibrium is unique for every such vector if and only if
\[
 c_1c_2=0
 \quad\text{or}\quad
 (d_1-c_1)(d_2-c_2)=c_1c_2 .
\]
For all other parameter choices, some bounded terminal payoff vector admits a continuum of distinct Nash equilibria in BMO controls.
\end{proposition}

\begin{proof}
When $c_1c_2\ne0$, the change of variables $Y_i=c_iV_i$, $Z_i=c_iH_i$, $\xi_i=c_i\eta_i$ gives \eqref{eq:bsde} with $\alpha_i=d_i/c_i$.
This is an invertible change of variables for any signs of $c_i,d_i$.
It maps the critical curve $(\alpha_1-1)(\alpha_2-1)=1$ to the displayed equality in the game parameters.
Theorem~\ref{thm:universal} and Corollary~\ref{cor:cardinality} give the assertions.

When $c_1c_2=0$, precisely one component $c_k$ is nonzero, since $L$ is invertible and $b\ne0$.
The $k$th equation has scalar generator $(c_k-d_k/2)w_k^2$.
Solve it with terminal condition $\eta_k$ by the scalar exponential formula, using conditional expectation when the coefficient vanishes.
Under the measure with density $\cE((c_kH_k)\cdot W)_T$, the other equation, indexed by $j=3-k$, has scalar generator $-d_jw_j^2/2$.
The same formula gives its unique bounded solution.
The BMO property of both integrands under $\PP$ follows from the scalar equations and the Girsanov isomorphism.
\end{proof}

For example, if $b_i=r_i=1$, $S_{12}=S_{21}=\rho\in(-1,1)$ and $\gamma_1=\gamma_2=\gamma$, then $c_i=1/(1-\rho)$ and $\alpha_i=(1-\rho)(1+\gamma)$.
Universal uniqueness holds at the two values $\gamma=-1$ and $\gamma=(1+\rho)/(1-\rho)$, corresponding respectively to $(\alpha_1,\alpha_2)=(0,0)$ and $(2,2)$.

\subsubsection{Compatibility of terminal payoffs when \texorpdfstring{$I-S$}{I-S} is singular}\label{subsubsec:game-compatibility}

Suppose $S_{12}S_{21}=1$ and put
\[
 s=S_{12},\qquad m=-\frac{p_1}{s p_2},\qquad B=sb_1+b_2,\qquad
 \beta=\tfrac12(d_2m^2-md_1).
\]
The kernel of $I-S$ is spanned by $(s,1)^\top$. Adding $h(s,1)^\top$ to the effort profile preserves both differences $u_i-S_{ij}u_j$ and changes the state drift by $Bh$.
The two rows of the control equation in \eqref{eq:tracking-equilibrium} give $H_2=mH_1$ and $u_1=su_2+p_1H_1$, so $\mu^u=Bu_2+\kappa_1H_1$.
In the difference of the value equations the Brownian terms cancel, leaving $\dd(V_2-mV_1)=\beta H_1^2\dt$.

\begin{proposition}[Compatibility of terminal payoffs]\label{prop:game-tracking}
Fix a bounded terminal payoff vector $\eta$ and an admissible control $v$ for player~$2$.
Let $H_1^v$ be the integrand in \eqref{eq:tracking-response} for player~$1$.
A Nash equilibrium with $u_2=v$ exists if and only if, for some deterministic $C\in\R$,
\begin{equation}\label{eq:tracking-compatibility}
 \eta_2-m\eta_1=C+\beta\int_0^T(H_{1,t}^v)^2\dt\quad\text{a.s.}
\end{equation}
When it exists, the equilibrium with this $u_2$ is unique and $u_1=sv+p_1H_1^v$.
\end{proposition}

\begin{proof}
Integrating the equation for $V_2-mV_1$ gives \eqref{eq:tracking-compatibility} with $C=V_{2,0}-mV_{1,0}$, which is deterministic since $\FF_0$ is trivial up to null sets.
Conversely, assume \eqref{eq:tracking-compatibility}, take $u_2=v$, $u_1=sv+p_1H_1^v$, and set $V_1=V_1^v$, $H_1=H_1^v$, $H_2=mH_1$.
These choices satisfy both rows of the control equation in \eqref{eq:tracking-equilibrium}.
Under $Q=\PP^{b,u}$, the first value equation follows from \eqref{eq:tracking-response}.
Put
\[
 V_{2,t}=mV_{1,t}+C+\beta\int_0^tH_{1,a}^2\dd a .
\]
The terminal condition for $V_2$ follows from \eqref{eq:tracking-compatibility}.
Moreover, $|\beta|\int_0^tH_{1,a}^2\dd a\le|\eta_2-m\eta_1-C|$, so $V_2$ is bounded.
Since $md_1/2+\beta=d_2m^2/2$, the second value equation in \eqref{eq:tracking-equilibrium} holds as well.
The converse following \eqref{eq:tracking-equilibrium} proves the assertion.
\end{proof}

Thus a given $\eta$ admits an equilibrium if and only if \eqref{eq:tracking-compatibility} holds for some admissible $v$.
If $\beta=0$, this reduces to $\eta_2-m\eta_1$ being constant.
If $\beta\ne0$ and $B=0$, then $Q_1^{S,v}=\PP$ for every $v$, so $H_1^v=H_1^0$ and a single application of the scalar equation decides \eqref{eq:tracking-compatibility}.
In both cases, whenever the condition holds, each admissible $v$ gives one equilibrium.

Whenever $I-S$ is singular, the bounded terminal payoff vector $\eta=(0,\ind_{\{X_T>x\}})$ admits no Nash equilibrium.
Indeed, $\eta_1=0$ gives $H_1^v=0$ for every $v$, while $\PP(X_T>x)=1/2$.
Together with Proposition~\ref{prop:game-tracking-regular}, this shows that $I-S$ is invertible if and only if every bounded terminal payoff vector admits a Nash equilibrium in BMO controls.
For constant terminal payoff vectors, \eqref{eq:tracking-compatibility} gives precisely the equilibria $u_1=su_2$ with arbitrary admissible $u_2$.

\subsubsection{A generic nonexistence theorem in the singular case}\label{subsubsec:game-nonexistence}

Suppose $S_{12}S_{21}=1$. When $B\beta=0$, Proposition~\ref{prop:game-tracking} decides whether an equilibrium exists by a test on the terminal payoffs alone. When $B\beta\ne0$, the integrand $H_1^v$ in that test may depend on the opponent's control $v$. At any equilibrium with $\beta\ne0$, however, the terminal payoffs and the reference filtration determine the entire process $\int_0^\cdot H_{1,a}^2\dd a$.
For a bounded random variable $Z$, $\operatorname*{ess\,inf}(Z\mid\FF_\sigma)$ denotes its largest $\FF_\sigma$-measurable lower bound up to null sets.
A general criterion for recovering an increasing process from its terminal value by conditional essential infima appears in \cite[Theorem~2.12]{Larsson2018}.

\begin{lemma}[Recovery of the cumulative energy from terminal payoffs]\label{lem:game-tracking-energy}
Suppose $\beta\ne0$ and put $\psi(y_1,y_2)=(y_2-my_1)/\beta$.
At any Nash equilibrium, $A_t=\int_0^tH_{1,a}^2\dd a$ satisfies, for each stopping time $\sigma\le T$,
\begin{equation}\label{eq:tracking-energy}
 A_\sigma=\operatorname*{ess\,inf}\big(\psi(\eta)\mid\FF_\sigma\big)
               -\operatorname*{ess\,inf}\psi(\eta)\quad\text{a.s.}
\end{equation}
\end{lemma}

\begin{proof}
By \eqref{eq:tracking-compatibility}, $\psi(\eta)=C/\beta+A_T$ and $A_T$ is bounded.
Under $Q=\PP^{b,u}$, the process $N=V_1-d_1A/2$ is a bounded martingale with quadratic variation $A$.
Since $A_T$ is bounded, $\cE(n(N-N_0))$ is a uniformly integrable martingale for each positive integer $n$.
Hence
\[
 \E^Q\!\left[\exp\!\left(n(N_T-N_\sigma)-\tfrac12n^2(A_T-A_\sigma)\right)
                   \,\middle|\,\FF_\sigma\right]=1 .
\]
If the conditional essential infimum of $A_T-A_\sigma$ were positive on a set of positive probability, it would be at least some $\delta>0$ on a nonnull $\FF_\sigma$-event.
On that event the conditional expectation would be bounded by $\exp(2n\norm{N}_{\Sp^\infty}-n^2\delta/2)$ for every $n$.
It follows that $\operatorname*{ess\,inf}(A_T\mid\FF_\sigma)=A_\sigma$.
At $\sigma=0$ this gives $\operatorname*{ess\,inf}A_T=0$.
Equivalent measures give the same conditional essential infima, proving \eqref{eq:tracking-energy}.
\end{proof}

\begin{theorem}[Terminal payoffs admitting equilibria when \texorpdfstring{$I-S$}{I-S} is singular]\label{thm:game-tracking-generic}
Let $0<T<\infty$, $r_1,r_2>0$, $b_1b_2\ne0$, $\gamma_1,\gamma_2\in\R$, and $S_{12}S_{21}=1$.
For $M>0$, impose the bound $|\eta_i|\le M$ a.s., $i=1,2$, in the original payoff units and equip the resulting space $X_M$ with the topology of convergence in $\PP$-probability. Let $E_M$ be the set of $\eta\in X_M$ admitting a Nash equilibrium in BMO controls.
For every such choice of parameters, $X_M\setminus E_M$ contains a dense $G_\delta$ subset of $X_M$. More precisely,
\begin{enumerate}
\item if $B\beta=0$, then $E_M$ is closed and has empty interior in $X_M$;
\item if $B\beta\ne0$, then $E_M$ is dense and meagre in $X_M$. The terminal payoff vectors admitting a Nash equilibrium with bounded state drift are already dense in $X_M$.
\end{enumerate}
\end{theorem}

\begin{proof}
\emph{Closedness and empty interior when $B\beta=0$.}
If $\beta=0$, put $R(\eta)=\eta_2-m\eta_1$. By Proposition~\ref{prop:game-tracking}, $\eta\in E_M$ if and only if $R(\eta)$ is a.s.\ constant. On $X_M$ convergence in probability implies convergence in $L^1(\PP)$, and the constants form a closed subspace of $L^1(\PP)$. Thus $E_M$ is closed.

Suppose next that $B=0$ and $\beta\ne0$. For a terminal payoff $y$ of player~1, let $H(y)$ denote the integrand in \eqref{eq:tracking-response} under $Q_1^{S,v}=\PP$ and put $A_T(y)=\int_0^T H_t(y)^2\dt$. In this case put $R(\eta)=\eta_2-m\eta_1-\beta A_T(\eta_1)$. By Proposition~\ref{prop:game-tracking}, $\eta\in E_M$ if and only if $R(\eta)$ is a.s.\ constant.
We show that $y\mapsto A_T(y)$ is continuous into $L^1(\PP)$ on $\{|y|\le M\}$ with the topology of convergence in probability. Let $y_n\to y$ in this space and write $a=\kappa_1-\gamma_1$. If $a=0$, the martingale representation isometry and $y_n\to y$ in $L^2(\PP)$ give $H(y_n)\to H(y)$ in $\Hp^2$. If $a\ne0$, let $U^n_t=\E[e^{ay_n}\mid\FF_t]$, $U_t=\E[e^{ay}\mid\FF_t]$, with martingale integrands $\Gamma^n,\Gamma$. These martingales take values in $[e^{-|a|M},e^{|a|M}]$. Since $e^{ay_n}\to e^{ay}$ in $L^2(\PP)$, Doob's inequality and the isometry give $U^n\to U$ in $\Sp^2$ and $\Gamma^n\to\Gamma$ in $\Hp^2$.
By the scalar exponential formula, $H(y_n)=\Gamma^n/(aU^n)$ and $H(y)=\Gamma/(aU)$. The inverses $1/U^n$ are uniformly bounded, and $\sup_{t\le T}|1/U^n_t-1/U_t|^2$ is bounded and tends to zero in probability. Its product with the fixed integrable random variable $\int_0^T\Gamma_t^2\dt$ therefore tends to zero in $L^1(\PP)$. It follows that $H(y_n)\to H(y)$ in $\Hp^2$ also when $a\ne0$. In both cases Cauchy--Schwarz gives
\[
 \|A_T(y_n)-A_T(y)\|_{L^1(\PP)}
 \le \|H(y_n)-H(y)\|_{\Hp^2}\big(\|H(y_n)\|_{\Hp^2}+\|H(y)\|_{\Hp^2}\big).
\]
Thus $R:X_M\to L^1(\PP)$ is continuous, and $E_M$ is closed.

In either case fix $\eta\in E_M$. Choose $\chi=0$ if $\eta_2$ is nonconstant, and $\chi=M\ind_{\{W_T>0\}}$ if $\eta_2$ is constant. For $0<\varepsilon<1$ the vector $\eta^\varepsilon=(\eta_1,(1-\varepsilon)\eta_2+\varepsilon\chi)$ belongs to $X_M$, and
\[
 R(\eta^\varepsilon)=R(\eta)+\varepsilon(\chi-\eta_2).
\]
Here $R(\eta)$ is constant and $\chi-\eta_2$ is nonconstant. Thus $\eta^\varepsilon\notin E_M$, whereas $\eta^\varepsilon\to\eta$ as $\varepsilon\downarrow0$. This proves that $E_M$ has empty interior and its complement is open and dense in $X_M$.

\emph{Generic nonexistence when $B\beta\ne0$.}
We construct a dense $G_\delta$ set on which the conditional lower bounds in \eqref{eq:tracking-energy} force $A$ to vanish.
Let $\psi_*=\min_{[-M,M]^2}\psi$.
Since both coefficients of $\psi$ are nonzero, its minimum is attained at a unique vertex $y_*$.
For each rational $t\in[0,T)$ and integer $n\ge1$ impose the condition
\begin{equation}\label{eq:tracking-category}
 \E\!\left[(\psi_*+1/n-\psi(\eta))_+\mid\FF_t\right]>0\quad\text{a.s.}
\end{equation}
The left side depends continuously on $\eta$ as a map from $X_M$ to $L^1(\PP)$, by boundedness of the payoffs and the contraction property of conditional expectation.
Strict positivity a.s.\ defines a $G_\delta$ set in the nonnegative cone of $L^1$, since for $Z\ge0$ it is equivalent to $\sup_{k\ge1}\E[\min(1,kZ)]=1$.
Thus the intersection of the conditions \eqref{eq:tracking-category} is a $G_\delta$ set.
It is dense, since replacing any $\eta$ by $y_*$ on $\{W_T>k\}$ enforces every condition and changes $\eta$ on a set of probability tending to zero.
Here $\PP(W_T>k\mid\FF_t)>0$ a.s.\ at each $t<T$.

For every $\eta$ in this intersection the conditional essential infimum of $\psi(\eta)$ equals $\psi_*$ at each rational $t<T$, including $t=0$.
If such an $\eta$ admits an equilibrium, \eqref{eq:tracking-energy} gives $A_t=0$ at these times.
Continuity gives $A_T=0$, whence $\psi(\eta)=\psi_*$ and $\eta=y_*$ a.s.
The space $X_M$ has no isolated points. Removing this singleton therefore leaves a dense $G_\delta$ set of terminal payoff vectors admitting no equilibrium.

\emph{Density of $E_M$ when $B\beta\ne0$.}
It remains to approximate terminal payoffs by ones admitting an equilibrium.
As $t_0\uparrow T$, the conditional expectations $\E[\eta\mid\FF_{t_0}]$ converge to $\eta$ in $L^2$ by continuity of Brownian martingales.
Multiplying by a constant slightly less than one, it suffices to approximate an $\FF_{t_0}$-measurable vector $\zeta$ taking values in $[-M',M']^2$ with $M'<M$.
Choose a deterministic point $(v_0,w_0)\in(-M,M)^2$ with $\psi(v_0,w_0)<\min_{[-M',M']^2}\psi$, and put $a_*=\psi(\zeta)-\psi(v_0,w_0)$.
There are deterministic upper and strictly positive lower bounds on $a_*$.

Along the line segment from $(v_0,w_0)$ to $\zeta$, the function $\psi$ increases at the constant rate $a_*$ in unit time. To realize the energy identity $\dd\psi(V)=H_1^2\dt$, we add a Brownian perturbation to the segment in the direction $(1,m)$. This leaves $\psi(V)$ unchanged and supplies the required quadratic variation of $V_1$. Mean reversion keeps the perturbation small as $n\to\infty$.
Set $\lambda=(t-t_0)/(T-t_0)$ on $[t_0,T]$ and let $\widehat W_\lambda=(W_t-W_{t_0})/\sqrt{T-t_0}$.
For $n\ge1$, let
\[
 \dd e^{(n)}_\lambda=-ne^{(n)}_\lambda\dd\lambda+\sqrt{a_*}\dd\widehat W_\lambda,\qquad e^{(n)}_0=0 .
\]
For fixed $n$ write $e=e^{(n)}$. Before $t_0$ take $V=V^{(n)}=(v_0,w_0)$.
After $t_0$, up to $\sigma_n\wedge1$ in the rescaled time, put
\[
 V_{1,t}=v_0+(\zeta_1-v_0)\lambda+e_\lambda,\qquad
 V_{2,t}=w_0+(\zeta_2-w_0)\lambda+me_\lambda ,
\]
where $\sigma_n=\inf\{\lambda\in[0,1]:|e^{(n)}_\lambda|\ge\varepsilon\}$, with $\inf\varnothing=\infty$.
Freeze both processes at $\sigma_n\wedge1$.
The convex hull of $(v_0,w_0)$ and $[-M',M']^2$ is compactly contained in $(-M,M)^2$.
Choose $\varepsilon>0$ so that its translates by $(y,my)$, $|y|\le\varepsilon$, lie in $[-M,M]^2$.

On $t_0<t\le t_0+(T-t_0)(\sigma_n\wedge1)$ the integrand of $V_1$ is $H_1=\sqrt{a_*/(T-t_0)}$ and its drift under $\PP$ is $b^V=(\zeta_1-v_0-ne_\lambda)/(T-t_0)$.
On this interval set $\mu=(d_1H_1^2/2-b^V)/H_1$ and set $H_1=\mu=0$ elsewhere.
The lower bound on $a_*$ and the stopping rule make $\mu$ bounded for each $n$.
Define
\[
 H_2=mH_1,\qquad
 u_2=(\mu-\kappa_1H_1)/B,\qquad u_1=su_2+p_1H_1 .
\]
Both controls are bounded, $\mu^u=\mu$, and $V_1$ satisfies its equation in \eqref{eq:tracking-equilibrium} under $\PP^{b,u}$.
Throughout the construction,
\[
 V_{2,t}-mV_{1,t}=w_0-mv_0+\beta\int_0^tH_{1,a}^2\dd a .
\]
This gives the second value equation and, by \eqref{eq:tracking-equilibrium}, an equilibrium for the terminal payoff vector $V_T^{(n)}\in X_M$. As in Section~\ref{subsec:game-model}, a bounded state-path version of this vector specifies the same terminal payoff under every deviation.

To see that stopping does not affect the limit, consider the unstopped error $e_\lambda^{(n)}=\sqrt{a_*}\int_0^\lambda e^{-n(\lambda-a)}\dd\widehat W_a$. Extend $\widehat W$ by zero to negative times. Integration by parts gives
\[
 e_\lambda^{(n)}/\sqrt{a_*}
   =\int_0^\infty ne^{-nr}\big(\widehat W_\lambda-\widehat W_{\lambda-r}\big)\dd r .
\]
Let $\omega_{\widehat W}$ be the modulus of continuity on $[0,1]$. Splitting the integral at $0<\delta\le1$ gives
\[
 \sup_{0\le\lambda\le1}\frac{|e_\lambda^{(n)}|}{\sqrt{a_*}}
 \le \omega_{\widehat W}(\delta)+2e^{-n\delta}\sup_{0\le a\le1}|\widehat W_a|.
\]
Taking $\delta=n^{-1/2}$ gives $\sup_{0\le\lambda\le1}|e_\lambda^{(n)}|\to0$ almost surely. Hence almost surely there is no exit for all sufficiently large $n$, and $V_T^{(n)}\to\zeta$.
\end{proof}

\subsection{Equivalence of the finite- and infinite-horizon games}\label{subsec:game-horizon}

On $[0,\infty)$ let $\FF_\infty$ be the completion of $\sigma(\bigcup_{t<\infty}\FF_t)$ and take $\eta_i\in L^\infty(\FF_\infty)$.
Thus the terminal payoff may be any bounded Borel functional of the entire state path, up to $\PP$-almost sure equality.
A predictable effort $u_i$ is admissible if
\[
 \sup_\tau \left\|\E\left[\int_\tau^\infty u_{i,t}^2\dt
                        \,\middle|\,\FF_\tau\right]\right\|_\infty<\infty ,
\]
where $\tau$ ranges over stopping times with values in $[0,\infty]$ and the integral is zero on $\{\tau=\infty\}$.
For either running cost $\ell_i$ considered above, the net payoff is
$\Pi_i(u)=\eta_i-\int_0^\infty\ell_i(u_t)\dt$.
The objective is given by \eqref{eq:game-objective} under the measure with density
$\dd\PP^{b,u}/\dd\PP=\cE(\mu^u\cdot W)_\infty$, where $\mu^u=b^\top u$.

\begin{proposition}[Equivalence under a deterministic time change]\label{prop:game-horizon}
For each bounded terminal payoff vector on $[0,\infty)$, the time change $\varphi(t)=t/(1+t)$ identifies the game with a game on $[0,1]$ having the same coefficients.
It gives a bijection of admissible controls that preserves both players' objectives for every control profile and every unilateral deviation.
In particular, it gives a bijection of the Nash equilibria in BMO controls.
\end{proposition}

\begin{proof}
Define $\widetilde W_{\varphi(t)}=\int_0^t(1+a)^{-1}\dd W_a$ and take the almost sure and $L^2$ limit at time $1$.
The deterministic integrand gives independent Gaussian increments with variance clock $\int_0^t(1+a)^{-2}\dd a=\varphi(t)$.
Thus $\widetilde W$ is a Brownian motion on $[0,1]$.
Since $W_t=\int_0^{\varphi(t)}(1-s)^{-1}\dd\widetilde W_s$ for every finite $t$, the augmented natural filtration $\widetilde\FF$ of $\widetilde W$ satisfies
$\widetilde\FF_{\varphi(t)}=\FF_t$ and $\widetilde\FF_1=\FF_\infty$.
In particular, each $\eta_i$ has a version $\widetilde G_i(\widetilde X_{[0,1]})$ for a bounded Borel function $\widetilde G_i$ and the finite-horizon state $\widetilde X=x+\widetilde W$.
Use this version as the finite-horizon terminal payoff.

For a predictable integrand $g$ put $\widetilde g_{\varphi(t)}=(1+t)g_t$.
This correspondence is invertible up to equality almost everywhere in time and probability.
With $\varphi(\infty)=1$, the stopping times correspond by $\tau\mapsto\varphi(\tau)$, with the same stopped $\sigma$-fields. Moreover,
\[
 \int_{\varphi(\tau)}^1 \widetilde g_s^2\ds=\int_\tau^\infty g_t^2\dt,
 \qquad
 \int_0^{\varphi(t)}\widetilde g_s\dd\widetilde W_s
   =\int_0^t g_a\dd W_a .
\]
The first identity gives $\|\widetilde g\|_{\BMO([0,1])}=\|g\|_{\BMO([0,\infty))}$.
Applying both identities to $g=\mu^u$ shows that the two stochastic exponentials have the same terminal value.
The exponential for $\widetilde{\mu^u}=b^\top\widetilde u$ on $[0,1]$ is a uniformly integrable martingale with strictly positive terminal value, so the two profiles induce the same measure on $\FF_\infty=\widetilde\FF_1$.
Since $\ell_i$ is homogeneous of degree two, $\int_0^1\ell_i(\widetilde u_s)\ds=\int_0^\infty\ell_i(u_t)\dt$.
The terminal payoff, accumulated cost and profile measure therefore give the same objective to each player.
Finally $\widetilde{u^{i,a}}=(\widetilde u)^{i,\widetilde a}$, and every admissible deviation on $[0,1]$ is obtained this way.
The same calculation on $[\tau,\infty)$ preserves the conditional objective given $\FF_\tau=\widetilde\FF_{\varphi(\tau)}$.
Taking essential suprema over continuations gives $\widetilde V^{\widetilde u}_{\varphi(t)}=V^u_t$ for the equilibrium continuation values.
\end{proof}

The identification of terminal payoff vectors preserves convergence in $\PP$-probability. Thus the category conclusions of Corollary~\ref{cor:game-payoffs} and Theorem~\ref{thm:game-tracking-generic} hold on the infinite horizon in their respective payoff spaces. The time change also preserves the $\Sp^p$ norms of the continuation values and the $\Hp^p$ norms of the controls, so the stability assertion of Corollary~\ref{cor:game-payoffs} holds for every $1\le p<\infty$.
In the density assertion of Theorem~\ref{thm:game-tracking-generic}, a bounded finite-horizon state drift gives a bounded infinite-horizon state drift by $\mu^u_t=(1+t)^{-1}b^\top\widetilde u_{\varphi(t)}$.
For a normalized BSDE solution on $[0,1]$ the corresponding processes are $Y_t=\widetilde Y_{\varphi(t)}$ and $Z_t=(1+t)^{-1}\widetilde Z_{\varphi(t)}$.
The identity $q_\theta(cz)=c^2 q_\theta(z)$ for $c>0$ gives the equation on every finite interval, and continuity at $1$ gives $Y_t\to\xi$ almost surely as $t\to\infty$.

\appendix
\section{Construction of the periodic coefficients}\label{app:phase}
We use $Q,A,\mathsf L,t$ and the reduced parameters $\beta_1,\beta_2$ from Section~\ref{subsec:phase-coordinates}. An overbar denotes the mean over one period. For either form of $A$, the rotation identity \eqref{eq:nu-rotation} leaves one scalar periodic function free. We choose it to give $\overline{Q(z)-rz}=0$, which yields a periodic $F$ with $F''/2=Q(z)-rz$. Two linearly independent Fourier coefficients of $z+F'$ then show that the affine span of its range is $\R^2$.

\begin{proof}[Proof of Lemma~\ref{lem:nu-phase}]
\emph{The zero-mean condition.} We first give an argument used in both parameter regimes. Suppose $B,D,C$ are smooth periodic vector functions and $0$ lies in the interior of the convex hull of the range of $B$. Choose nonnegative smooth functions $H_i$ with disjoint supports such that
\[
 b_i=\overline{BH_i^2},\qquad
 \sum_iw_i b_i=0\quad(w_i>0),
\]
where the $b_i$ span $\R^2$. Such functions are obtained from sufficiently small neighborhoods of finitely many phases whose $B$ values positively span $\R^2$.
For $f=f_0+\sum_i t_iH_i$, with $f_0$ any fixed smooth function, the equation $\overline{C+Df+Bf^2}=0$ has the form
\[
 m+\sum_i l_it_i+\sum_i b_it_i^2=0.
\]
Choose two independent $b_i$, put $t_i=a\sqrt{w_i}+d_i$, and vary only the corresponding two $d_i$. Dividing the equation by $a$ gives, at $a^{-1}=0$,
\[
 \sum_i l_i\sqrt{w_i}+2\sum_i b_i\sqrt{w_i}d_i=0.
\]
Its derivative in the two selected variables is invertible. The implicit function theorem therefore gives an exact solution for large finite $a$, with all $t_i>0$ and with a rank-two derivative of the mean map. Trigonometric approximation of $f$ and the two adjustment directions, followed by the same theorem, gives a trigonometric solution with rank-two derivative. Any strict positivity of $f$ is preserved by sufficiently close approximation. After adding a sufficiently small harmonic of higher order, we correct the mean in the same two lower-order directions, preserving the coefficient of the new harmonic.

\emph{The symmetric form of $A$.} This includes $\beta_1\beta_2>0$ with $R<1$ and $\beta_1=0$ with $\beta_2<0$. Set
\[
 a(\vartheta)=\mathsf L^{-1}(\cos2\vartheta,\sin2\vartheta),\quad
 b(\vartheta)=\mathsf L^{-1}(-\sin2\vartheta,\cos2\vartheta).
\]
For any real periodic $f$, put
\[
 z=af+b,\qquad r=t(z)+f,\quad h=t(z)-f,\quad\ell=1.
\]
These formulas give the matrix identity in \eqref{eq:nu-rotation}.
Let $P(z)=Q(z)-t(z)z$. Direct calculation gives
\[
 P(x,y)=-\tfrac12\operatorname{adj}(\mathsf L)(x^2-y^2,2xy).
\]
The coefficient of $f^2$ in $Q(z)-rz$ is $B=P(a)-a$.
Every nonzero linear functional applied to $P$ is a nonzero harmonic quadratic polynomial, and hence changes sign on the ellipse traced by $a$. Also,
\[
 B(\vartheta)+B(\vartheta+\pi/2)=2P(a(\vartheta)).
\]
Thus no nonzero linear functional is nonnegative on the whole range of $B$. The separation theorem gives $0\in\operatorname{int}\conv B(\mathbb S^1)$, so the preceding argument produces $f$ with $\overline{Q(z)-rz}=0$.

To ensure that the range of $v$ is not contained in an affine line, add a small term $\varepsilon\cos(n\vartheta)$, where $n>2$ exceeds the degrees of $f$ and the adjustment directions, and then correct the mean. To check nondegeneracy, write $U,V$ for the columns of $\mathsf L^{-1}$. The cosine and sine coefficients of order four in
$(a_1^2-a_2^2,2a_1a_2)$ have determinant
\[
 \tfrac12\det(U,V)(|U|^2+|V|^2)\ne0.
\]
The same is true of the order-four coefficients of $B$. Consequently the coefficients of order $2n+4$ in $Q(z)-rz$ are independent. All remaining terms have lower order, and $z$ has degree at most $n+2$. Solving for $F$ shows that $v=z+F'$ has two independent coefficients of order $2n+4$. Its range cannot lie in an affine line.

\emph{The second form of $A$.} Write $S=\sin2\vartheta$ and $C=\cos2\vartheta$. For a strictly positive periodic function $g$, define
\[
 u=-gS,\qquad w=1-gS^2,\qquad z=\mathsf L^{-1}(u,w),
\]
\[
 r=t(z)+uC,\qquad h=t(z)-uC,\qquad\ell=-1+2gS^2.
\]
Since $A(z)=\left(\begin{smallmatrix}t(z)+u&w\\-w&t(z)-u\end{smallmatrix}\right)$, the upper-right entry of the rotated matrix is $w-uS=1$. Its lower-left entry is $-w-uS=\ell$. In particular, $\ell+1=2gS^2\ge0$ and is not identically zero.

Write $z=z_0+gz_g$. For $\beta_1\beta_2>0$ and $R>1$ we have $k=s>0$ and
\[
 z_0=\left(-\frac{2e}{Rs},-\frac1s\right),\qquad
 z_g=\left(-\frac{2S}{R}+\frac{2eS^2}{Rs},\frac{S^2}{s}\right).
\]
Here
\[
 P(x,y)=Q(x,y)-t(x,y)(x,y)
 =\left(-\tfrac s2(x^2+y^2)-exy,-\tfrac R2xy\right),
\]
and the coefficient of $g^2$ in $Q(z)-rz$ is $B=P(z_g)+SCz_g$.
For $S\ne0$, direct calculation gives
\[
 (B/S^2)_2=\frac{S(1+C)}s-\frac{eS^2}{s^2},\qquad
 \lim_{\substack{S\to0\\ C\to\pm1}}(B/S^2)_1
 =-\frac{2(s+RC)}{R^2}.
\]
Near $C=1$, the values for opposite small $S$ point to the left with opposite vertical components. Near $C=-1$, the direction tends to the positive horizontal axis. Taking the latter phase sufficiently close to $C=-1$, these three values positively span $\R^2$. Thus $0\in\operatorname{int}\conv B(\mathbb S^1)$.
For $\beta_1=0$ and $\beta_2>0$, the inverse of $\mathsf L$ instead gives
\[
 z_0=(-2e,-1),\qquad z_g=(-2S+2eS^2,S^2),\qquad e=-\sqrt{\beta_2}/2.
\]
Now $P(x,y)=(-\tfrac12(x^2+y^2)-exy,-xy/2)$ and again $B=P(z_g)+SCz_g$. On the branch $C=+\sqrt{1-S^2}$,
\[
 B_1=-4S^2+O(S^3),\qquad B_2=2S^3+O(S^4).
\]
For small fixed $S$ of opposite signs, the resulting vectors span a cone containing the negative horizontal ray in its interior. On the branch $C=-\sqrt{1-S^2}$ the leading terms are
\[
 B_1=4eS^3+O(S^4),\qquad B_2=-eS^4+O(S^5).
\]
Since $e\ne0$, taking $S\to0$ with $eS>0$ gives directions approaching the positive horizontal ray. Thus three values of $B$ positively span $\R^2$ also in this case.

In both cases $0\in\operatorname{int}\conv B(\mathbb S^1)$. Apply the zero-mean construction with the baseline $g_0=\delta>0$. The resulting function $g=\delta+\sum_i t_iH_i$ is strictly positive, and trigonometric approximation and the mean correction preserve this property. Add a small nonzero term $\varepsilon\cos(n\vartheta)$, with $n>4$ exceeding all the degrees already chosen, and correct the mean in the lower-order directions. The order-four part of $z_g$ is $z_0\cos(4\vartheta)/2$. It follows that the order-eight part of $B$ is
\[
 \tfrac18P(z_0)\cos(8\vartheta)+\tfrac18z_0\sin(8\vartheta).
\]
Since $A(z_0)=t(z_0)I-J$ and $Q(z_0)=A(z_0)z_0/2$,
\[
 \det(z_0,P(z_0))=-\tfrac12|z_0|^2\ne0.
\]
The coefficients of order $2n+8$ in $Q(z)-rz$ are therefore independent. All other terms have lower degree, and $z$ has degree at most $n+4$. As before, $v=z+F'$ has two independent highest-order coefficients. This completes the proof.
\end{proof}

\section{Regularity of the local solution map}\label{app:local}
All processes in this appendix are defined on $[0,L]$, and $\E_t$ denotes conditional expectation under $\PP$ given the Brownian filtration of $B$ at time $t$. We first solve \eqref{eq:nu-local-bsde} by a contraction uniform in the initial state $a\in C$ of the diffusion. To obtain a $C^2$ map into $\mathcal X=C(C;\R^2)$, we show that the time-zero correction $V_0^{a,c,x}$ is continuous in $a$ and that its first two derivatives in $(c,x)$ exist uniformly in $a$.

\begin{proof}[Proof of Lemma~\ref{lem:nu-local}]
\emph{Linear estimates and local existence.} Put $A_s^0=DQ(\mathcal Z_s^0)$. Its operator norm is bounded uniformly in $a\in C$.
For the linear equation
\[
 \dd V=-(A^0H+f)\ds+H\dd B,\qquad V_L=\eta,
\]
the right fundamental matrix $\Gamma_{t,s}$ solves
$\dd \Gamma_{t,s}=\Gamma_{t,s}A_s^0\dd B_s$, $\Gamma_{t,t}=I$. In particular, $\Gamma_{t,s}=\Gamma_{0,t}^{-1}\Gamma_{0,s}$.
Boundedness of $A^0$ implies, for each finite $p$,
\[
 \sup_{\tau\le L}\left\|\E_\tau\sup_{\tau\le s\le L}
 |\Gamma_{\tau,s}|^p\right\|_\infty\le C_{p,L},
\]
where the supremum is over stopping times.
Conditional representation gives
\[
 V_t=\E_t\left[\Gamma_{t,L}\eta+
                   \int_t^L\Gamma_{t,s}f_s\ds\right].
\]
For $k,l\in\BMO$ and a source $|f_s|\le C|k_s||l_s|$, conditional
Cauchy--Schwarz and the BMO energy inequalities yield
\[
 \|V\|_{\Sp^\infty}\le C_L
 \big(\|\eta\|_{L^\infty}+\|k\|_{\BMO}\|l\|_{\BMO}\big).
\]
Indeed the conditional second moment of $\int|k|^2$ is at most a
constant times $\|k\|_{\BMO}^4$, and likewise for $l$.
It\^o's formula for $|V|^2$, absorbing
$2|V||A^0||H|$ into half of $|H|^2$, then gives
\begin{equation}\label{eq:nu-linearbound}
 \|V\|_{\Sp^\infty}+\|H\|_{\BMO}
 \le C_L\big(\|\eta\|_{L^\infty}+
                   \|k\|_{\BMO}\|l\|_{\BMO}\big).
\end{equation}
To construct the linear solution, let
\[
 N_t=\E_t\left[\Gamma_{0,L}\eta+
                       \int_0^L\Gamma_{0,s}f_s\ds\right],
 \qquad
 V_t=\Gamma_{0,t}^{-1}\left(N_t-\int_0^t\Gamma_{0,s}f_s\ds\right).
\]
The inverse matrix has moments of every finite order, and the BMO energy inequalities give the same integrability for $\int_0^L|f_s|\ds$. Brownian representation of $N$ and the product formula define $H$ and give the linear equation. The preceding estimates place $(V,H)$ in $\Sp^\infty\times\BMO$. The same representation gives uniqueness in $\Sp^\infty\times\BMO$.

Apply this linear inverse with $f=Q(H)$. On a sufficiently small BMO
ball the quadratic term has Lipschitz constant less than one in
\eqref{eq:nu-linearbound}. The resulting contraction is uniform in
$a$ and $c$. It is a polynomial map in $H$ and a smooth map in its
terminal input; the implicit function theorem, or differentiation of
the contraction, gives uniform first and second derivative estimates
and Taylor remainders in $\Sp^\infty\times\BMO$.

\emph{Continuity in the initial state.} Drive the diffusions $X^a$ and $X^b$, started at $a$ and $b$, by the same Brownian motion. The smooth bounded
coefficients in \eqref{eq:nu-state} give $X^b\to X^a$ in
$\Sp^p$ as $b\to a$ for every finite $p$. For fixed $(c,x)$, the function $\eta_{c,x}$ in \eqref{eq:nu-active-terminal} is bounded and continuous, so the terminal values $\eta_{c,x}(X_L^b)$ converge to $\eta_{c,x}(X_L^a)$ in $L^p$. For two
local solutions, the difference equation has coefficient
$A^{0,a}$ plus a BMO term bounded by
$C(\|H^a\|_{\BMO}+\|H^b\|_{\BMO})$ and source
$(A^{0,a}-A^{0,b})H^b$. Shrink the common ball, if necessary, so that
the linear equation with this small BMO perturbation satisfies an
$\Sp^p\times\Hp^p$ estimate for one fixed $p>2$.
For a
matrix-valued integrand $\kappa$ and a vector-valued $H$, the
Fefferman inequality~\cite{Kazamaki1994} and $L^p$ duality give
\[
 \left\|\int_I|\kappa_s||H_s|\ds\right\|_{L^p}
 \le C_p\|\kappa\|_{\BMO(I)}\|H\|_{\Hp^p(I)}.
\]
This follows from
the usual $H^1$--BMO Fefferman inequality: for a nonnegative dual
variable $D\in L^{p/(p-1)}$, replace $D$ under the time integral by
$\E_sD$, apply Fefferman to $(\E_sD)H_s$, and then use H\"older's
and Doob's inequalities.

This can be applied to the martingale representation operator for a
zero-generator BSDE on $I$. Partition $[0,L]$ into finitely many
deterministic intervals with $\sqrt{|I|}$ small enough that the
bounded coefficient contributes less than half the contraction
threshold. Then shrink the common local BMO ball so that the
additional coefficient contributes less than the remaining half.
On each interval the linear equation is solved by contraction in
$\Sp^p\times\Hp^p$. Finite backward pasting gives the
uniform estimate on $[0,L]$: if the difference has terminal value $\Delta\eta$ and additional source $f$, then
\[
 \|\Delta V\|_{\Sp^p}+\|\Delta H\|_{\Hp^p}
 \le C_{L,p}\left(\|\Delta\eta\|_{L^p}
       +\left\|\int_0^L|f_s|\ds\right\|_{L^p}\right).
\]
The constant and the smallness threshold are uniform in the initial state.
For the source term, H\"older's inequality gives
\[
 \left\|\int_0^L|(A_s^{0,a}-A_s^{0,b})H_s^b|\ds\right\|_{L^p}
 \le\sqrt L\left\|\sup_{s\le L}|A_s^{0,a}-A_s^{0,b}|\right\|_{L^{2p}}
                   \|H^b\|_{\Hp^{2p}}\longrightarrow0
 \quad(b\to a).
\]
The last factor is uniformly bounded by the BMO energy inequalities. Together with convergence of the terminal values, the stability estimate gives, for fixed $(c,x)$,
\begin{equation}\label{eq:nu-initial-continuity}
 \|V^{b,c,x}-V^{a,c,x}\|_{\Sp^p}
 +\|H^{b,c,x}-H^{a,c,x}\|_{\Hp^p}\longrightarrow0
 \qquad (b\to a).
\end{equation}
In particular, $V_0^{a,c,x}$ is continuous in $a$.

\emph{Differentiability in $(c,x)$.} Perform the uniform contraction
on the product Banach space
$\ell^\infty(a\in C;\Sp^\infty\times\BMO)$.
By \eqref{eq:nu-active-terminal}, the terminal correction
$\eta_{c,x}(X_L^a)$ and its derivatives in $(c,x)$ vanish on $\{X_L^a\notin C\}$.
On $\{X_L^a\in C\}$, evaluation at $X_L^a$ has operator norm at most one on $\mathcal X$.
The uniform $C^2$ bounds and Taylor remainders for $\Psi$ thus give a $C^2$ terminal map into $\ell^\infty(a\in C;L^\infty)$.
This proves $C^2$ dependence with values in $\ell^\infty(C;\R^2)$.
The continuity just established places every output in its closed
subspace $C(C;\R^2)$. Difference quotients and their uniform limits
therefore place the derivatives in the same closed subspace and
prove the asserted Fr\'echet $C^2$ property.
\end{proof}

\section{Spectral estimates for the diagonal operators}\label{app:operator}
The linearization $D_x\mathcal F_c(0)$ is lower triangular, with diagonal operators $cT_+$ and $cT_-$. For the bifurcation argument we use the spectral radius $r(T_+)$. Compactness and strong positivity make $r(T_+)$ an algebraically simple eigenvalue. A return estimate for $X$ under $Q^r$ gives $r(T_+)\ge p_*e^{L/2}$ for some $p_*>0$ independent of $L$. We bound $\|T_-\|$ under $Q^h$ to show that $r(T_+)$ exceeds both $1$ and $\|T_-\|$, as required in Lemma~\ref{lem:nu-operator}. Throughout, $z,r,h,\ell,F,v$ are fixed as in Lemma~\ref{lem:nu-phase}. We first choose $b$ and then take $L$ sufficiently large.

\subsection{Compactness of \texorpdfstring{$T_+$}{T\unichar{"208A}}}
Under $Q^r$, the free process, obtained by setting $\rho=1$, has Stratonovich vector fields
\[
 V_1=\partial_\vartheta+v(\vartheta)\cdot\partial_m,
 \qquad V_0=-\tfrac12v'(\vartheta)\cdot\partial_m.
\]
The brackets $[V_1,V_0]$, $[V_1,[V_1,V_0]],\ldots$ have horizontal coefficients $-v''/2,-v'''/2,\ldots$. At each phase these coefficients span $\R^2$. Indeed, otherwise some nonzero linear functional would annihilate every derivative of $v$ of order at least two at that phase. Analyticity would make this functional applied to $v$ affine on the real line, and periodicity would make it constant, contrary to Lemma~\ref{lem:nu-phase}. Compactness of the circle gives a uniform finite bound on the number of brackets needed. Together with $V_1$, they span all three state directions.
The parabolic H\"ormander theorem~\cite{Hormander1967} therefore gives a jointly smooth transition density at every positive time.

Write $P_t^0$ for the free semigroup and $Q_a^{r,0}$ for the law of the free process started at $a$. For fixed $\delta>0$, its transition probability is continuous in the initial state in total variation: pointwise convergence of the smooth densities and Scheff\'e's lemma give convergence in $L^1$. On $C$, this continuity is uniform. Hence $P_\delta^0$ maps the unit ball of bounded measurable functions in the supremum norm to a uniformly bounded, equicontinuous family on $C$.

Let $S_t$ be the confined semigroup under $Q^r$ with weight $e^{A_t/2}$. For $\|f\|_\infty\le1$, put $G_f=S_{L-\delta}(\phi f)$, where $\phi f$ is extended by zero. The operator
\[
 K_\delta f=e^{\delta/2}P_\delta^0G_f\big|_C
\]
is compact by Arzel\`a--Ascoli. Couple the free and confined processes with the same Brownian motion, and let $\tau_K=\inf\{s\ge0:M_s\notin D_K\}$ for the spatial component of the free process. The coupled processes and their weights agree up to $\tau_K$. Consequently
\[
 \|T_+-K_\delta\|
 \le2e^{L/2}\sup_{a\in C}Q_a^{r,0}(\tau_K\le\delta).
\]
For each initial state in $C$, the spatial component starts at distance at least $K-(b+1)\sqrt L=(b-1)\sqrt L>0$ from $\partial D_K$. Boundedness of $v$ and the martingale maximal inequality make the right-hand side tend to zero as $\delta\downarrow0$. Thus $T_+$ is compact.

\subsection{Strong positivity of \texorpdfstring{$T_+$}{T\unichar{"208A}}}
Put $g_0=-v'/2$. Its values span $\R^2$, since a linear functional vanishing on $v'$ is constant on the range of $v$. Any linear functional nonnegative on $g_0(\mathbb S^1)$ must vanish there because $\overline{g_0}=0$. Hence $0\in\operatorname{int}\conv g_0(\mathbb S^1)$. Choose finitely many phases whose $g_0$ values have zero in the interior of their convex hull.

We show that the free Stratonovich skeleton can join any two points of $C$ in time $L$ while staying inside $D_K\times\mathbb S^1$, for large $L$. Along a controlled phase path,
\[
 \dot M=v(\vartheta)\dot\vartheta+g_0(\vartheta).
\]
Use initial and final transitions of fixed total duration to move from the prescribed initial phase to a fixed base phase and from that base phase to the prescribed final phase. Choose lifts of all these phases in a fixed bounded interval. Their spatial displacements and excursions are then bounded independently of $L$ and of the endpoints.

During the remaining time $L'=L-O(1)$, repeat $n=\lceil L'\rceil$ identical tours through the selected phases, returning to the base phase with zero winding number. Give all transitions in these tours a fixed total duration $\delta>0$, divided equally among the tours. The integral of $v(\vartheta)\dd\vartheta$ over each closed tour is zero. The total displacement $d_{\rm tr}$ from $g_0$ during transitions is bounded by $\delta\|g_0\|_\infty$.
If $\Delta$ is the displacement remaining after the initial and final transitions, then $|\Delta|=O(\sqrt L)$. For large $L$,
\[
 \frac{\Delta-d_{\rm tr}}{L'-\delta}
\]
lies in the interior of the convex hull of the selected $g_0$ values. Choose positive holding times with this average velocity and distribute them equally among the $n$ tours. Each tour then has displacement $\Delta/n$ and bounded duration. Within each tour the spatial excursion is bounded by a constant, because both its holding time and the total variation of its phase path are bounded. Tour endpoints lie on the segment joining the two intermediate spatial endpoints. The entire skeleton consequently lies within a fixed distance of the segment joining the prescribed endpoints, so
\[
 |M|\le(b+1)\sqrt L+O(1)<2b\sqrt L=K
\]
for large $L$.
The finitely many transitions and holding intervals define a Cameron--Martin control. Since $\rho=1$ along this path, it is also a skeleton of the confined process under $Q^r$. The support theorem~\cite{SV1972} gives
\[
 Q_a^r(X_L\in U)>0
\]
for every $a\in C$ and every nonempty open set $U\subset\operatorname{int}C$.
For any nonzero $f\in C(C)$ with $f\ge0$, there is such a set on which both $f$ and $\phi$ have positive lower bounds. Thus $T_+f(a)>0$ for every $a\in C$. Continuity and compactness give $\min_C T_+f>0$, proving strong positivity.

\subsection{A lower bound for \texorpdfstring{$r(T_+)$}{r(T\unichar{"208A})} from return probabilities}
For the free process under $Q^r$, set
\[
 \Sigma=\frac1{2\pi}\int_0^{2\pi}v(\vartheta)v(\vartheta)^{\mathsf T}\dd\vartheta.
\]
This matrix is positive definite, since otherwise the range of $v$ would lie in a line through zero. For each entry $f$ of $vv^{\mathsf T}-\Sigma$, solve the periodic equation $G''/2=f$. It\^o's formula gives
\[
 \frac1L\int_0^{Lt}f(\vartheta_s)\ds
 =\frac{G(\vartheta_{Lt})-G(\vartheta_0)}L
  -\frac1L\int_0^{Lt}G'(\vartheta_s)\dd B_s^r,
 \qquad 0\le t\le1.
\]
Boundedness of $G,G'$ and Doob's inequality imply convergence to zero in $L^2$, uniformly in $t$ and in the initial phase. The central limit theorem for continuous martingales, applied along any sequence of initial phases, therefore yields
\[
 L^{-1/2}(M_L-M_0)\Longrightarrow N(0,\Sigma)
\]
uniformly in the initial phase.

Fix $b>3$. For $m=M_0/\sqrt L$, $|m|\le b$, choose the first coordinate axis along $m$, with any choice when $m=0$. In these coordinates the rectangle
$[-1,-1/5]\times[-1,1]$ lies strictly inside $\{w:|m+w|<b\}$. Indeed, for every $w$ in this rectangle,
\[
 |m+w|^2\le\max\{2,(b-1/5)^2+1\}<b^2.
\]
The minimum Gaussian probability of this fixed rectangle over all rotations is positive and independent of $b$. Uniform convergence for these rectangles follows by compactness of the rotation group and zero Gaussian measure of their boundaries. Taking $p_0$ to be half this minimum gives, for each fixed $b>3$ and all sufficiently large $L$,
\[
 \inf_{a\in C_0}Q_a^{r,0}(|M_L|<b\sqrt L)\ge p_0.
\]
Also,
\[
 \sup_{a\in C_0}Q_a^{r,0}\!\left(
 \sup_{s\le L}|M_s-M_0|\ge b\sqrt L\right)
 \le\frac{\|v\|_\infty^2}{b^2}.
\]
Choose $b$ so that the latter bound is less than $p_0/2$, and then choose $L$ sufficiently large. Coupling free and confined processes gives
\[
 \inf_{a\in C_0}Q_a^r\!\left(
 \sup_{s\le L}|M_s|<K,\ X_L\in C_0\right)\ge p_*:=p_0/2.
\]
On this event $A_L=L$ and $\phi(X_L)=1$. Iterating the return event in the kernel defining $T_+$ yields
\[
 T_+^n\mathbf1(a)\ge(p_*e^{L/2})^n,\qquad a\in C_0.
\]
The spectral radius formula gives $r(T_+)\ge p_*e^{L/2}$.

\subsection{An upper bound for \texorpdfstring{$\|T_-\|$}{\unichar{"2016}T\unichar{"208B}\unichar{"2016}}}
For the symmetric form of $A$ we have $\ell=1$, so $\|T_-\|\le1$ directly. For the second form put $\mathcal V=\ell+1$. This is a nonnegative periodic function which is not identically zero. Under $Q^h$, the phase satisfies
\[
 \dd\vartheta=\rho\dd B^h+\rho^2(h-r)\ds.
\]
Let $\mathcal L_h=\tfrac12\partial_\vartheta^2+(h-r)\partial_\vartheta$ on the circle. The compact strongly positive semigroup of $\mathcal L_h-\mathcal V$ has a smooth positive eigenfunction $\psi$, with
\[
 (\mathcal L_h-\mathcal V)\psi=-\lambda\psi.
\]
Compactness and strong positivity follow from ellipticity on the circle and boundedness of $\mathcal V$. The Krein--Rutman theorem gives a positive eigenfunction of the time-one operator. Simplicity and commutation with the semigroup make it an eigenfunction of the generator. If $\pi$ is the positive invariant density of $\mathcal L_h$, integration gives
\[
 \lambda\int\psi\pi=\int \mathcal V\psi\pi>0,
\]
so $\lambda>0$.

It\^o's formula shows that
\[
 N_t=\exp\!\left(\int_0^t\rho^2(\lambda-\mathcal V)\ds\right)
       \frac{\psi(\vartheta_t)}{\psi(\vartheta_0)}
\]
is the stochastic exponential with integrand $\rho\psi'/\psi$ under $Q^h$. This integrand is bounded, so $N$ is a true martingale. Since $0\le A_L\le L$, the definition of $T_-$ gives
\[
 \|T_-\|\le\frac{\max\psi}{\min\psi}
       \exp\big(\max\{\tfrac12-\lambda,0\}L\big).
\]
The constants $\lambda$ and $\max\psi/\min\psi$ depend only on the fixed phase functions. Comparing with $p_*e^{L/2}$ proves \eqref{eq:nu-spectral} for sufficiently large $L$ in both cases. Since $T_+$ is compact and strongly positive, the Krein--Rutman theorem~\cite[Theorem 2.1]{Shi2005} shows that $r(T_+)$ is an algebraically simple eigenvalue with a strictly positive eigenfunction and a strictly positive eigenfunctional.

\section{Further results on generic uniqueness and stability}\label{sec:parameters}
For terminal conditions with values in the three-point set of Theorem~\ref{thm:encoding}, Lemma~\ref{lem:tree-density} gives approximations with values in the same set. They yield generic uniqueness within this space of terminal conditions and stability under perturbations in $X_{M_\Delta}$, where $M_\Delta$ bounds the three values. We then prove generic uniqueness and stability when both $\theta$ and $\xi$ vary. Finally, for fixed $\theta$ and $M$, we show that a generic sampling law on $X_M$ assigns probability one to $\mathcal G_M^\theta$.

\subsection{Terminal conditions with values in a prescribed three-point set}
Fix $\theta\in\R^2$ and $M>0$, and let $A,B,C$ be the vertices in \eqref{eq:triangle} for this choice of $\theta$ and $M$. Set $\Delta=\conv\{A,B,C\}$ and write $c_0=v_\theta(A)$, $c_1=v_\theta(B)$ and $c_2=v_\theta(C)$. On $[0,T]$, set
\[
 \mathcal T_\Delta=\{\xi\in L^2(\FF_T;\R^2):\xi\in\{c_0,c_1,c_2\}\text{ a.s.}\},
 \qquad M_\Delta=\max_{0\le j\le2}\max_{1\le i\le2}|c_{j,i}|.
\]
The space $\mathcal T_\Delta$ is a closed, hence complete, subset of $X_{M_\Delta}$. If $d_*$ and $d^*$ are the minimum and maximum distances between distinct points of $\{c_0,c_1,c_2\}$, then
\[
 d_*^2\PP(\xi\ne\zeta)
 \le\norm{\xi-\zeta}_{L^2}^2
 \le(d^*)^2\PP(\xi\ne\zeta).
\]
Thus convergence in $\mathcal T_\Delta$ is equivalent to convergence in probability of the corresponding $\{0,1,2\}$-valued indices.

\begin{theorem}[Generic uniqueness and stability for terminal conditions with values in a prescribed three-point set]\label{thm:threegeneric}
The set $\mathcal D_{M_\Delta}^\theta\cap\mathcal T_\Delta$ is dense in $\mathcal T_\Delta$. Consequently,
\[
 \mathcal G_{M_\Delta}^\theta\cap\mathcal T_\Delta
\]
is a dense $G_\delta$ subset of $\mathcal T_\Delta$. At every $\xi$ in this set, the solution is unique in $\Sp^\infty\times\BMO$ and satisfies \eqref{eq:mainstability} for perturbations in the full space $X_{M_\Delta}$.
\end{theorem}
\begin{proof}
Take $K=\{A,B,C\}$ in Lemma~\ref{lem:tree-density}. For a given $\xi\in\mathcal T_\Delta$ the target $u_\theta(\xi)$ takes values in $K$, as does every approximating terminal state. Applying $v_\theta$ gives terminal conditions in $\mathcal D_{M_\Delta}^\theta\cap\mathcal T_\Delta$ converging in $L^2$ to $\xi$.

By \eqref{eq:canonicalG}, $\mathcal G_{M_\Delta}^\theta\cap\mathcal T_\Delta$ is a relative $G_\delta$. It contains $\mathcal D_{M_\Delta}^\theta\cap\mathcal T_\Delta$ and is therefore dense in $\mathcal T_\Delta$. Theorem~\ref{thm:main} gives uniqueness and stability for perturbations in $X_{M_\Delta}$.
\end{proof}

For a horizon $I=[0,S]$, write $\mathcal D_{M,I}^\theta$ and $\mathcal G_{M,I}^\theta$ for the sets of Definition~\ref{def:core} and \eqref{eq:canonicalG}, respectively, with terminal time $S$. We omit $I$ when $S=T$.

\begin{corollary}[Generic uniqueness and stability under the reduction to three-valued terminal conditions]\label{cor:pullback}
Let $\Theta$ be the map in Theorem~\ref{thm:encoding} and $M'$ the terminal bound in \eqref{eq:amplitudegrowth}. Then
\[
 \Theta^{-1}(\mathcal G_{M',[0,3T]}^\theta)
\]
is a dense $G_\delta$ subset of $X_M$. For each $\xi$ in this set, equation~\eqref{eq:bsde} on $[0,T]$ has a unique solution in $\Sp^\infty\times\BMO$ and satisfies \eqref{eq:mainstability} for perturbations in $X_M$.
\end{corollary}
\begin{proof}
The map $\Theta$ is continuous by \eqref{eq:encodingcontinuous}. For $\xi\in\mathcal D_M^\theta$, append the two segment solutions of Theorem~\ref{thm:encoding} to a piecewise segment solution on $[0,T]$. Writing $u_\theta(\xi)=A+x(B-A)+s(C-A)$ as in its proof, we see that the first added segment has length $(1-s)|B-A|$, bounded away from zero because $u_\theta([-M,M]^2)$ is a compact subset of $\operatorname{int}\Delta$. The second segment is either $[A,C]$ or $[B,C]$. Thus $\Theta(\xi)\in\mathcal D_{M',[0,3T]}^\theta$. The inverse image of each open set defining $\mathcal G_{M',[0,3T]}^\theta$ is therefore open and contains $\mathcal D_M^\theta$, so it is dense. This proves the residual assertion. The bijection \eqref{eq:solbijection} transfers existence and uniqueness. For stability, extend each approximating solution to $[0,3T]$, apply Theorem~\ref{thm:main} at $\Theta(\xi)$, and restrict to $[0,T]$.
\end{proof}

\subsection{Joint stability at piecewise segment solutions}
Equip $\R^2$ with its Euclidean metric and $\R^2\times X_M$ with the product metric. When the parameters vary, the domains $D_\theta$ of the transformed equations vary as well. Near a fixed $\theta_0$, the transformed states for all terminal conditions in $X_M$ lie in a common compact rectangle on which $a_\theta$ depends smoothly on $\theta$.
\begin{proposition}\label{prop:jointcore}
Fix $\theta_0\in\R^2$ and a piecewise segment solution $(\bar Y,\bar Z)$ for $\theta_0$, with terminal condition $\bar\xi\in X_M$. Suppose $\theta_n=(\alpha_{1,n},\alpha_{2,n})\to\theta_0=(\alpha_{1,0},\alpha_{2,0})$, $\xi_n\to\bar\xi$ in $X_M$, and
\[
 (Y^n,Z^n)\in\Sol_{\theta_n}(\xi_n).
\]
Then $(Y^n,Z^n)\to(\bar Y,\bar Z)$ in $\Sp^p\times\Hp^p$ for every $1\le p<\infty$.
\end{proposition}
\begin{proof}
By Proposition~\ref{prop:bounds}, $Y^n,\bar Y$ take values in $[-M,M]^2$. Choose a compact rectangle $K\subset D_{\theta_0}$ whose interior contains $u_{\theta_0}([-M,M]^2)$. Joint continuity of $(\theta,y)\mapsto u_\theta(y)$ and compactness give a closed ball $N$ about $\theta_0$ of positive radius such that $u_\theta([-M,M]^2)\subset K$ for all $\theta\in N$. Reducing its radius if necessary, we have
\[
 \inf_{\theta=(\alpha_1,\alpha_2)\in N,\;u\in K}\min_{i=1,2}(1-\alpha_i u_i)>0 .
\]
In particular $K\subset D_\theta$ throughout $N$. The functions $a_\theta(u)$, $v_\theta(u)$ and $D v_\theta(u)$ are smooth in $(\theta,u)$ on a neighborhood of $N\times K$, including when a parameter is zero. For the inverse coordinate this follows directly from
\[
 (v_\theta(u))_i=\int_0^{u_i}\frac{\dd x}{1-\alpha_i x}.
\]
For all sufficiently large $n$, $\theta_n\in N$. Proposition~\ref{prop:bounds} then bounds $\norm{Z^n}_{\Hp^p}$ uniformly for each $1\le p<\infty$, and Lemma~\ref{lem:density} gives uniform moment bounds for the Girsanov densities. Write $U^n=u_{\theta_n}(Y^n)$, $V_i^n=(1-\alpha_{i,n}U_i^n)Z_i^n$, $\bar U=u_{\theta_0}(\bar Y)$ and $b_n=Z_1^n+Z_2^n$. Then $b_n+a_{\theta_n}(U^n)\cdot V^n=0$, and $U^n$ is a bounded martingale under $Q_n=\PP^{Z^n}$, with density $\cE(b_n\cdot W)_T$.

On an interval $[\sigma,\tau]$ of the partition for $(\bar Y,\bar Z)$, define
\[
 \widehat U_t^n=\E^{Q_n}[\bar U_\tau\mid\FF_t],\qquad
 \dd\widehat U_t^n=\widehat V_t^n\dd W_t^{Q_n}.
\]
Assuming $Y_\tau^n\to\bar Y_\tau$ in probability, the parameter convergence gives $U_\tau^n\to\bar U_\tau$ in probability. Estimate \eqref{eq:replacementconvergence} therefore holds. Write $A+e[0,1]$ for the reference segment containing $\bar U$ and use its scale function $\gamma^{\theta_0}_{A,e}$ from \eqref{eq:gamma}. The drift error is
\[
 \varepsilon_n=b_n+a_{\theta_0}(\widehat U^n)\cdot\widehat V^n
 =a_{\theta_0}(\widehat U^n)\cdot\widehat V^n-a_{\theta_n}(U^n)\cdot V^n .
\]
By convexity $\widehat U^n$ stays in $K$. The bounds on $a_\theta$ and its derivatives on $N\times K$ give, with all norms taken on $[\sigma,\tau]$,
\[
 \norm{\varepsilon_n}_{\Hp^p}
 \le C\left(\norm{\widehat V^n-V^n}_{\Hp^p}
 +\norm{\widehat U^n-U^n}_{\Sp^{2p}}\norm{V^n}_{\Hp^{2p}}
 +|\theta_n-\theta_0|\norm{V^n}_{\Hp^p}\right).
\]
By \eqref{eq:replacementconvergence}, $\theta_n\to\theta_0$, and the uniform moment bounds for $V^n$, this tends to zero for every $1\le p<\infty$.
The scalar equation \eqref{eq:scalarerror} and estimate \eqref{eq:forcing0} give convergence of $(U^n,V^n)$ to $(\bar U,\bar V)$ on this interval. Since $v_{\theta_n}$ and $D v_{\theta_n}$ converge uniformly on $K$ to $v_{\theta_0}$ and $D v_{\theta_0}$, respectively, the inverse transformation gives the asserted convergence of $(Y^n,Z^n)$ there. Starting with $\xi_n\to\bar\xi$ and proceeding backwards through the partition proves the result on $[0,T]$.
\end{proof}

\subsection{Generic uniqueness and stability in \texorpdfstring{$\R^2\times X_M$}{\unichar{"211D}\unichar{"00B2}\space\unichar{"00D7}\space X\textunderscore M}}
\begin{theorem}[Joint generic uniqueness and stability]\label{thm:joint}
For every $M>0$, there is a dense $G_\delta$ set $\mathcal J_M\subset\R^2\times X_M$ with the following properties.
\begin{enumerate}
\item At each $(\theta,\xi)\in\mathcal J_M$, equation \eqref{eq:bsde} has a unique solution $(Y,Z)\in\Sp^\infty\times\BMO$. If $(\theta_n,\xi_n)\to(\theta,\xi)$ in $\R^2\times X_M$, then every choice $(Y^n,Z^n)\in\Sol_{\theta_n}(\xi_n)$ converges to $(Y,Z)$ in $\Sp^p\times\Hp^p$ for every $1\le p<\infty$.
\item For each $\theta\in\R^2$, the section
\[
 \mathcal J_M(\theta,\cdot)=\{\xi\in X_M:(\theta,\xi)\in\mathcal J_M\}
\]
is a dense $G_\delta$ subset of $X_M$ satisfying
$\mathcal D_M^\theta\subseteq\mathcal J_M(\theta,\cdot)\subseteq\mathcal G_M^\theta$.
\item There is a dense $G_\delta$ set $H_M\subset X_M$ such that for every $\xi\in H_M$, the parameter section
\[
 \mathcal J_M(\cdot,\xi)=\{\theta\in\R^2:(\theta,\xi)\in\mathcal J_M\}
\]
is a dense $G_\delta$ subset of $\R^2$.
\end{enumerate}
\end{theorem}
\begin{proof}
Apply Lemma~\ref{lem:abstract} on $\R^2\times X_M$ to $F(\theta,\xi)=\Sol_\theta(\xi)$ and the dense set of all pairs $(\theta,\xi)$ with $\xi\in\mathcal D_M^\theta$. Proposition~\ref{prop:jointcore} verifies the stability assumption on this set. Let $U_{M,j}$ be the open sets defined by \eqref{eq:abstractopen} for this product space and set
\[
 \mathcal J_M=\bigcap_{j\ge1}U_{M,j}.
\]
The lemma gives convergence in $\Sp^2\times\Hp^2$. The uniform state and moment bounds of Proposition~\ref{prop:bounds} along each convergent sequence $\theta_n\to\theta$ extend it to $\Sp^p\times\Hp^p$ for every $1\le p<\infty$, as in the proof of Theorem~\ref{thm:main}. For each $\theta\in\R^2$, every section $U_{M,j}(\theta,\cdot)$ is open and contains $\mathcal D_M^\theta$. Their intersection $\mathcal J_M(\theta,\cdot)$ is therefore dense. Restricting a product neighborhood of $(\theta,\xi)$ to the slice at $\theta$ can only decrease the diameter of the union of solution sets. By \eqref{eq:canonicalG}, this gives $\mathcal J_M(\theta,\cdot)\subseteq\mathcal G_M^\theta$.

For (iii), take a countable base $(B_\ell)_{\ell\ge1}$ of nonempty open sets in $\R^2$. For every $j,\ell$, the projection onto $X_M$ of
\[
 U_{M,j}\cap(B_\ell\times X_M)
\]
is open by the product topology. It is dense because $U_{M,j}$ meets $B_\ell\times V$ for every nonempty open $V\subset X_M$. Let $H_M$ be the intersection of these projections. If $\xi\in H_M$, the open section $U_{M,j}(\cdot,\xi)$ meets every $B_\ell$ and is therefore dense. Its countable intersection over $j$ is the required dense $G_\delta$ parameter section.
\end{proof}

\begin{corollary}[Countable parameter families]\label{cor:countable}
For every $M>0$ and every countable set $\mathsf P_0\subset\R^2$, there is a dense $G_\delta$ set $H\subset X_M$ such that for every $(\theta,\xi)\in\mathsf P_0\times H$ the solution is unique in $\Sp^\infty\times\BMO$ and stable under perturbations in $\R^2\times X_M$ in the sense of Theorem~\ref{thm:joint}\,\textup{(i)}.
\end{corollary}
\begin{proof}
Take $H=\bigcap_{\theta\in\mathsf P_0}\mathcal J_M(\theta,\cdot)$ and use Theorem~\ref{thm:joint}.
\end{proof}

\subsection{Sampling laws on \texorpdfstring{$X_M$}{X\textunderscore M}}
Here a sample is an entire terminal condition $\xi\in X_M$, so the sampling law is a probability measure on $X_M$.
Equip the space $\mathcal P(X_M)$ of Borel probability measures with the weak topology.

\begin{proposition}[Sampling terminal conditions from a generic law]\label{prop:sampling}
For every $\theta\in\R^2$ and $M>0$, the set
\[
 \{\nu\in\mathcal P(X_M):\nu(\mathcal G_M^\theta)=1\}
\]
is a dense $G_\delta$ subset of $\mathcal P(X_M)$.
\end{proposition}
\begin{proof}
Write $\mathcal G_M^\theta=\bigcap_{j\ge1}O^\theta_{M,j}$ as in \eqref{eq:canonicalG}. For $k\ge2$, the set
\[
 V_{j,k}=\{\nu:\nu(O^\theta_{M,j})>1-1/k\}
\]
is open by the Portmanteau theorem. It is dense because finitely supported probability measures are dense in $\mathcal P(X_M)$ and each of their atoms can be moved arbitrarily slightly into $O^\theta_{M,j}$. A measure $\nu$ belongs to every $V_{j,k}$ if and only if $\nu(O^\theta_{M,j})=1$ for every $j$. Hence
\[
 \{\nu\in\mathcal P(X_M):\nu(\mathcal G_M^\theta)=1\}
 =\bigcap_{j\ge1,\,k\ge2}V_{j,k}.
\]
Since $X_M$ is Polish, so is $\mathcal P(X_M)$, and Baire's theorem completes the proof.
\end{proof}

\section{Approximation of a given solution by piecewise segment solutions}\label{app:trajectory}
Lemma~\ref{lem:tree-density} approximates a prescribed transformed terminal state $u_\theta(\xi)$ by the leaves of finite Brownian trees. Here we fix $\theta\in\R^2$ and a particular solution $(Y,Z)\in\Sol_\theta(\xi)$. Approximating its transformed integrand $V$ by simple predictable processes yields piecewise segment solutions converging to this $(Y,Z)$ on $[0,T]$.
\begin{lemma}\label{lem:forward}
Let $(Y,Z)\in\Sol_\theta(\xi)$, and suppose $U=u_\theta(Y)$ stays in a deterministic compact set $K_0$. For every compact convex set $K\subset D_\theta$ with $K_0\subset\operatorname{int}K$, there are piecewise segment solutions $(Y^n,Z^n)\in\Sol_\theta(\xi_n)$ with $u_\theta(Y^n_t)\in K$ for $0\le t\le T$ and
\[
 \norm{Y^n-Y}_{\Sp^p}+\norm{Z^n-Z}_{\Hp^p}\longrightarrow0,
 \qquad 1\le p<\infty.
\]
\end{lemma}
\begin{proof}
We solve the transformed equation forward with a simple predictable integrand and stop the resulting process at the boundary of $K$. On each interval of the deterministic time grid its direction of motion is then known at the left endpoint.

Let $(U,V)$ be defined from $(Y,Z)$ by \eqref{eq:coordinates}. Extend $a_\theta$ from a neighborhood of $K$ to a bounded globally Lipschitz vector field $a_*$ on $\R^2$. Choose bounded simple predictable processes on finite deterministic grids,
\[
 V_t^n=\sum_{j=0}^{m_n-1}v_j^n\ind_{(t_j^n,t_{j+1}^n]}(t),
 \qquad v_j^n\text{ is }\FF_{t_j^n}\text{-measurable},\qquad
 \norm{V^n-V}_{\Hp^2}\longrightarrow0.
\]
Solve the forward SDE
\[
 \dd X^n=V^n\big(a_*(X^n)\cdot V^n\big)\dt+V^n\dd W,
 \qquad X_0^n=U_0.
\]
For each $n$, its coefficients are globally Lipschitz in the state with a deterministic Lipschitz constant, so it has a unique strong solution.

Let $C,L$ be a bound and a Lipschitz constant for $a_*$. Gronwall's inequality applied pathwise gives
\begin{align*}
 \sup_{t\le T}|X_t^n-U_t|
 &\le\left\{\sup_{t\le T}\left|\int_0^t(V^n-V)\dd W\right|
       +C\int_0^T(|V^n|+|V|)|V^n-V|\dt\right\}\\
 &\hspace{25mm}\times\exp\left(L\int_0^T|V^n|^2\dt\right).
\end{align*}
The term in braces tends to zero in probability by the martingale maximal inequality and Cauchy--Schwarz. The random variables $\int_0^T|V_t^n|^2\dt$ are bounded in $L^1$, so their exponentials are tight. It follows that $X^n\to U$ uniformly in probability.

Let $\tau_n$ be the first exit of $X^n$ from $\operatorname{int}K$, truncated at $T$, and define
\[
 \widetilde U_t^n=X^n_{t\wedge\tau_n},\qquad
 \widetilde V_t^n=\ind_{\{t\le\tau_n\}}V_t^n.
\]
Because $\dist(K_0,\partial K)>0$, we have $\PP(\tau_n<T)\to0$. The stopped pair solves \eqref{eq:Uequation}, stays in $K$, and on each grid interval moves along a line whose direction is known at the interval's beginning. To obtain the segments in Definition~\ref{def:core}, choose a compact rectangle in $D_\theta$ containing $K$ in its interior. At time $t_j^n$ take the line through the stopped state $\widetilde U^n_{t_j^n}$ in direction $v_j^n$ on $\{\tau_n>t_j^n,\ v_j^n\ne0\}$, and in a fixed nonzero direction otherwise. Its intersection with the rectangle has $\FF_{t_j^n}$-measurable endpoints. Since $\widetilde U^n_{t_j^n}\in K$, the segment lengths have deterministic upper and positive lower bounds. If $\tau_n$ falls inside this grid interval, $\widetilde U^n$ stays at $X^n_{\tau_n}$ for the rest of the interval. Thus the stopped solution is a piecewise segment solution on the original deterministic grid.

The inverse transformation gives
\[
 Y^n=v_\theta(\widetilde U^n),\qquad
 Z^n_{i,t}=\frac{\widetilde V^n_{i,t}}{1-\alpha_i\widetilde U^n_{i,t}},
 \qquad \xi_n=Y_T^n.
\]
The processes $Y^n$ are uniformly bounded and converge uniformly in probability to $Y$, hence in $\Sp^p$ for every $1\le p<\infty$. For the integrands,
\begin{align*}
 \E\int_0^T|\widetilde V_t^n-V_t|^2\dt
 &\le2\norm{V^n-V}_{\Hp^2}^2
 +2\E\left[\ind_{\{\tau_n<T\}}\int_0^T|V_t|^2\dt\right]
 \longrightarrow0.
\end{align*}
The multipliers $u\mapsto(1-\alpha_i u_i)^{-1}$ are bounded and Lipschitz on $K$. Evaluated at $\widetilde U^n$, they converge uniformly in time in probability to their values at $U$. Boundedness and $\E\int_0^T |V_t|^2\dt<\infty$ thus give $Z^n\to Z$ in $\Hp^2$. The terminal values $\xi_n$ lie in $v_\theta(K)$, so Proposition~\ref{prop:bounds} bounds $\norm{Z^n}_{\Hp^q}$ uniformly in $n$ for each finite $q$. Uniform integrability gives convergence in $\Hp^p$ for every $1\le p<\infty$.
\end{proof}

\section*{Use of generative AI}
This work was developed in collaboration with GPT-6 Astra.

\makeatletter
\renewcommand{\@biblabel}[1]{#1.}
\makeatother
\providecommand{\bysame}{\leavevmode\hbox to3em{\hrulefill}\thinspace}
\providecommand{\MR}{\relax\ifhmode\unskip\space\fi MR }
\providecommand{\MRhref}[2]{%
  \href{http://www.ams.org/mathscinet-getitem?mr=#1}{#2}
}
\providecommand{\href}[2]{#2}

\end{document}